\documentclass[10pt]{article}

\usepackage{amsmath}
\usepackage{amssymb}
\usepackage[margin=2.5cm]{geometry}
\usepackage{enumerate}  
\usepackage{graphicx}
\usepackage[symbol]{footmisc}
\usepackage{amsthm}
\usepackage[affil-it]{authblk}

\newcommand{\R}[0]{\mathbb{R}}

\newcommand{\N}[0]{\mathbb{N}}

\newcommand{\sE}[0]{\mathcal{E}}
\newcommand{\sF}[0]{\mathcal{F}}

\newcommand{\sH}[0]{\mathcal{H}}
\newcommand{\sA}[0]{\mathcal{A}}
\newcommand{\sB}[0]{\mathcal{B}}
\newcommand{\sC}[0]{\mathcal{C}}
\newcommand{\sS}[0]{\mathcal{S}}

\newcommand{\sR}[0]{\mathcal{R}}
\newcommand{\il}[0]{\langle}
\newcommand{\ir}[0]{\rangle}

\newcommand{\tr}[0]{\text{tr}\,}

\newcommand{\vs}{\varsigma}
\newcommand{\zz}{z}
\newcommand{\cls}{\mathcal{S}}
\newcommand{\clu}{\mathcal{U}}
\newcommand{\clv}{\mathcal{V}}
\newcommand{\clk}{\mathcal{K}}
\newcommand{\cly}{\mathcal{Y}}
\newcommand{\cla}{\mathcal{A}}
\newcommand{\clb}{\mathcal{B}}
\newcommand{\clc}{\mathcal{C}}
\newcommand{\clh}{\mathcal{H}}
\newcommand{\clr}{\mathcal{R}}
\newcommand{\clt}{\mathcal{T}}
\newcommand{\clf}{\mathcal{F}}
\newcommand{\clp}{\mathcal{P}}
\newcommand{\clz}{\mathcal{Z}}
\newcommand{\clx}{\mathcal{X}}
\newcommand{\clw}{\mathcal{W}}
\newcommand{\clm}{\mathcal{M}}
\newcommand{\cle}{\mathcal{E}}
\newcommand{\Om}{\Omega}
\newcommand{\om}{\omega}
\newcommand{\NN}{\mathbb{N}}
\newcommand{\RR}{\mathbb{R}}
\newcommand{\la}{\lambda}
\newcommand{\veps}{\varepsilon}
\newcommand{\eps}{\epsilon}

\numberwithin{equation}{section}

\newtheorem{theorem}{Theorem}[section]

\newtheorem{lemma}[theorem]{Lemma}

\newtheorem{condition}{Condition}[section]
\theoremstyle{remark}
\newtheorem{remark}{Remark}[section]

\allowdisplaybreaks

\begin{document}

\title{Empirical Measure and Small Noise Asymptotics under Large Deviation Scaling for Interacting Diffusions}
\author{Amarjit Budhiraja \,and \,Michael Conroy%\footnotemark
}
\affil{University of North Carolina, Chapel Hill}
%\authorrunning{Amarjit Budhiraja \and Michael Conroy}
%\institute{A. Budhiraja \at
%             Department of Statistics and Operations Research \\
%University of North Carolina\\
%		Chapel Hill, NC 27599 \\
%              \email{budhiraj@email.unc.edu}           %  \\
%%             \emph{Present address:} of F. Author  %  if needed
%           \and
%           M. Conroy \at
%             Department of Statistics and Operations Research \\
%University of North Carolina\\
%		Chapel Hill, NC 27599 \\
%              \email{mconroy@live.unc.edu$^*$}  
%}

%\date{}

\maketitle
%\footnotetext[1]{Corresponding author}

\begin{abstract}
Consider a collection of particles whose state evolution is described through a system of interacting diffusions in which each particle
is driven by an independent individual source of noise and also by a small amount of noise that is common to all particles. The interaction between the particles is due to the common noise and also through the drift and diffusion coefficients that depend on the state empirical measure. We study large deviation behavior of the empirical measure process which is governed by two types of scaling, one corresponding to mean field asymptotics and the other to the Freidlin-Wentzell small noise asymptotics. 
Different levels of intensity of the small common noise lead to different types of large deviation behavior, and we provide a precise characterization of the various regimes. 
The rate functions can be interpreted as  the value functions of certain stochastic control problems in which there are two types of controls;  one of the controls is random and nonanticipative and 
arises from the aggregated contributions of the individual Brownian noises, whereas the second control is nonrandom and corresponds to the small common Brownian noise that impacts all particles.
We also study large deviation behavior of  interacting particle systems approximating various types of Feynman-Kac functionals. Proofs are based on stochastic control representations for exponential functionals of Brownian motions and on uniqueness results for weak solutions of stochastic differential equations associated with controlled nonlinear Markov processes
\newline



%\keywords{Large deviation principle \and Weakly interacting diffusions \and Mean field systems with common noise \and Feynman-Kac functionals \and Controlled nonlinear Markov processes \and Mean field stochastic control problems \and Controlled McKean-Vlasov equations \and Friedlin-Wentzell asymptotics}
\end{abstract}

\section{Introduction}

In this work we study large deviation properties of interacting particle systems that are described through a certain collection of stochastic differential equations. Our main interest is in diffusions interacting through the empirical measure of the particle system with both individual and common sources of noises, given by a system of equations of the following form:
\begin{equation}\label{eq:interemp}
	\begin{aligned}
dX_i^n(t) &=  b(X_i^n(t), \mu^n(t))\,dt +  \sigma(X_i^n(t), \mu^n(t))\,dW_i(t) + \kappa(n) \alpha(X_i^n(t), \mu^n(t))\,dB(t), \\ 
X^n_i(0) &= x^n_i, \quad \mu^n(t) = \frac{1}{n} \sum_{i=1}^n \delta_{X^n_i(t)}, \quad 1 \le i \le n, \quad t\in [0,T], 
\end{aligned}
\end{equation}
where $\{W_i, i\in \NN\}$ are independent $m$-dimensional Brownian motions, $B$ is a $k$-dimensional Brownian motion, independent of $\{W_i, i\in \NN\}$, $b:\RR^d \times \clp(\RR^d) \to \RR^d$, $\sigma:\RR^d \times \clp(\RR^d) \to \RR^{d\times m}$, and
$\alpha:\RR^d \times \clp(\RR^d) \to \RR^{d\times k}$ are appropriate maps, and $\{x_i^n\}_{1\le i\le n} \subset \RR^d$ (see Section \ref{sec:intempdis} for precise conditions on the coefficients and the initial conditions).  
However, in order to motivate the questions of interest, we begin with the following elementary setting. 
For fixed $x \in \RR^m$ consider the empirical measure 
\begin{equation*}
	\mu^n \doteq \frac{1}{n} \sum_{i=1}^n \delta_{\{x+ W_i\}}.
\end{equation*}
By Sanov's theorem, $\{\mu^n\}$ satisfies a large deviation principle (LDP) on $\clp(\clc([0,T]: \RR^m))$ with rate function $I(\cdot)$ given as
\begin{equation}\label{eq:ratefnsanov}
	I(\gamma) = R(\gamma \| \theta_x),\; \gamma \in \clp(\clc([0,T]: \RR^m)).
\end{equation}
Here $\clc([0,T]: \RR^m)$ is the space of continuous functions from $[0,T]$ to $\RR^m$ equipped with the  uniform topology, and for a Polish space $S$, $\clp(S)$ denotes the space of probability measures on $S$ that is equipped with the topology of weak convergence. Also, $\theta_x$ denotes the Wiener measure with initial value $x$, and the quantity $R(\gamma \| \theta_x)$ denotes the relative entropy of $\gamma$ with respect to $\theta_x$. A precise definition of a large deviation principle is given in Section \ref{sec:notandcon}, however formally such a result says that, for large $n$, 
$$
P(\mu^n \in A) \approx \exp\left\{-n \inf_{\gamma \in A}I(\gamma)\right\} \; \mbox{ for Borel sets } A \subset \clp(\clc([0,T]: \RR^m)).$$

There is another representation for the rate function which, although notationally more demanding, is useful when studying more general settings. For a Polish space $S$, we will denote by $\clb(S)$ the Borel $\sigma$-field on $S$.
Let $\mathcal{R}$ denote the set of all finite measures $r$ on $\sB(\R^m \times [0,T])$ such that $r(\R^m\times [0,t]) = t$ for all $t\in [0,T]$. 
This space is equipped with the topology of weak convergence. Let $\mathcal{R}_1 \subset \mathcal{R}$ be defined as 
\[
	\mathcal{R}_1 \doteq \left\{ r\in \mathcal{R} : \int_{\R^m \times [0,T]} \|y\|\,r(dy\,dt) < \infty \right\}. 
\]
Then $\mathcal{R}_1$ is a Polish space when equipped with the the Wasserstein-$1$ metric. Under this metric, $r_n\to r$ in $\mathcal{R}_1$ if and only if $r_n\to r$ as a sequence in $\clr$
and $\int_{\RR^m \times [0,T]} y\, r_n(dy\, dt) \to \int_{\RR^m \times [0,T]} y\, r(dy\, dt)$. 
Let
\[
	\mathcal{Z}_1 \doteq \mathcal{X}\times \mathcal{R}_1\times \mathcal{W}, \;\mbox{ where }  \mathcal{X} = \mathcal{W} = \sC([0,T]:\R^m),
\]
and denote by $(X, \rho , W)$ the three coordinate maps on this space.
Define 
\[
	\mathcal{P}_2(\mathcal{Z}_1) \doteq \left\{ \Theta \in \mathcal{P}(\mathcal{Z}_1) : E_\Theta\left[ \int_{\R^m\times [0,T]} \|y\|^2\,\rho(dy\,dt) \right] < \infty \right\}, 
\]
where $E_\Theta$ denotes expectation on $(\mathcal{Z}_1, \mathcal{B}(\mathcal{Z}_1), \Theta)$.
Let $\mathcal{E}_1$ denote the subset of $\mathcal{P}_2(\mathcal{Z}_1)$ consisting of probability measures $\Theta$ such that, under $\Theta$, 
$W(t)$ is a standard Brownian motion with respect to the canonical filtration $\clf_t \doteq \sigma\{ X(s), W(s), \rho(A\times [0,s]); A \in \clb(\RR^m), s\le t\}$, and a.s.
\begin{equation}\label{eq:contbm}
	X(t) = x + \int_{\RR^m\times [0,t]} y\, \rho(dy\, ds) + W(t), \; t \in [0,T].
\end{equation}	
Then the rate function $I(\cdot)$ in \eqref{eq:ratefnsanov} has the following alternative representation:
\begin{equation}\label{eq:ratefncont}
	I(\gamma) = \inf_{\Theta \in \cle_1: [\Theta]_1 = \gamma}  E_{\Theta}\left[\frac{1}{2} \int_{\RR^m\times [0,T]} \|y\|^2 \rho(dy\, ds)\right],
\end{equation}
where $[\Theta]_1$ is the marginal of $\Theta$ on the first coordinate.
Viewing $\rho$ as a (relaxed) control, the right side of the above display gives a representation for the rate function as the value function of a stochastic control problem in which the goal is to produce a state process $X$ with a specified law 
$\gamma$ using the state dynamics \eqref{eq:contbm} and a (nonanticipative) control process $\rho$ which has the least cost, where the cost is given by the  expectation on the right side of \eqref{eq:ratefncont}. 

The above interpretation is a useful point of view and analogous stochastic control representations can be given more generally.  Consider for example the case where we are given an iid collection of $d$-dimensional diffusions $\{X_i\}_{i\in \NN}$ described through
the stochastic differential equations
\begin{equation}
\label{eq:iidsde}
X_i(t) = x + \int_0^t b(X_i(s))\,	ds + \int_0^t \sigma(X_i(s))\, dW_i(s), \quad t \in [0,T], \quad i \in \NN,
\end{equation}
where $x \in \R^d$ is an initial condition, and 
where for simplicity we assume that the coefficients $b: \RR^d \to \RR^d$ and $\sigma: \RR^d \to \RR^{d\times m}$ are Lipschitz functions so that the equations have a unique pathwise solution.
Letting 
\begin{equation}\label{eq:empiidsde}
	\mu^n = \frac{1}{n} \sum_{i=1}^n \delta_{X_i}, 
	\end{equation}
	the rate function associated with the LDP for $\mu^n$
on $\clp(\clc([0,T]: \RR^d))$ takes the same form as \eqref{eq:ratefncont} except $\clx = \sC([0,T]:\R^d)$ and the class $\cle_1$ is now the collection of all probability measures in $\mathcal{P}_2(\mathcal{Z}_1)$ under which
$W$ is as before and $(X, \rho, W)$ are related as
\begin{equation*}
	X(t) = x + \int_0^t b(X(s)) \,ds + \int_0^t \sigma(X(s))\, dW(s) + \int_{\RR^m\times [0,t]} \sigma(X(s)) y \,\rho(dy\, ds), \quad t \in [0,T].
\end{equation*}	
The system of equations in \eqref{eq:iidsde} have no interaction. We now introduce a small amount of coupling between the equations given through a  Brownian motion that is common to all particles as follows:
\begin{equation}\label{eq:smcoup}
X^n_i(t) = x + \int_0^t b(X^n_i(s))\,	ds + \int_0^t \sigma(X^n_i(s)) \,dW_i(s) + \kappa(n)\int_0^t \alpha(X^n_i(s)) \,dB(s), \quad t \in [0,T], \quad i \in \NN,	
\end{equation}
where $B$ is a $k$-dimensional standard Brownian motion independent of $\{W_i\}_{i\in \NN}$, $\alpha: \RR^d \to \RR^{d\times k}$
is a Lipschitz map, and $\kappa(n)\to 0$ as $n\to \infty$.
In this case, since $\{X^n_i\}_{1\le i\le n}$ are not independent, the large deviation behavior of 
\begin{equation}\label{eq:empsmdep}
	\mu^n = \frac{1}{n} \sum_{i=1}^n \delta_{X^n_i}
\end{equation} 
cannot be deduced from Sanov's theorem, and in fact this behavior crucially depends on the manner 
in which $\kappa(n)\to 0$. The measures $\mu^n$ in \eqref{eq:empiidsde} and in \eqref{eq:empsmdep} converge to the same limit but the rates of convergence as measured by the large deviation rate function are different. Indeed, as an elementary corollary of Theorems \ref{resulttype1} and \ref{thm:diffspeed}  we give a complete characterization of the convergence rate for different choices of the small noise coefficient $\kappa(n)$ (see Remark \ref{rem:bmcase}). Specifically, when $\kappa(n)= n^{-1/2}$ the rate function is governed by a different type of stochastic control problem than \eqref{eq:ratefncont}
that can be described as follows.
For $\varphi \in L^2([0,T]:\R^k)$, the space of square-integrable functions from $[0,T]$ into $\R^k$, let $\mathcal{E}_1[\varphi]$ denote the subset of $\mathcal{P}_2(\mathcal{Z}_1)$ consisting of all probability measures under which $W$ is,  as before, a $m$-dimensional Brownian motion with respect to the canonical filtration $\{\clf_t\}$, and the coordinate processes $X, \rho$, and $W$ are related to $\varphi$ through the equation
\begin{align*}
	X(t) &= x + \int_0^t b(X(s))\, ds + \int_0^t \sigma(X(s))\, dW(s) + \int_{\RR^m\times [0,t]} \sigma(X(s)) y \,\rho(dy\, ds) \\
	&\quad + \int_0^t \alpha(X(s)) \varphi(s)\, ds, \quad t \in [0,T].
\end{align*}
Then the rate function $I(\cdot)$ associated with the empirical measures $\mu^n$ in \eqref{eq:empsmdep}, with $\kappa(n) = n^{-1/2}$, is given as
\begin{equation}\label{eq:eq1135b}
	I(\gamma) = \inf_{\varphi \in L^2([0,T]:\R^k)} \left\{\inf_{\Theta\in \mathcal{E}_1[\varphi] : [\Theta]_1=\gamma} E_{\Theta}\left[ \frac{1}{2} \int_{\R^m\times[0,T]} \|y\|^2\,\rho(dy\,dt) \right] + \frac{1}{2} \int_0^T \|\varphi(t)\|^2\,dt \right\}.
\end{equation}
The right side of \eqref{eq:eq1135b} is once more the value function of a stochastic control problem, however this time there are two types of controls.  One of the controls, represented by $\rho$, is random and nonanticipative and 
arises from the aggregated contributions of the individual Brownian noises, whereas the second control, represented by $\varphi$, is nonrandom and corresponds to the small common Brownian noise that impacts all particles.

We will also study large deviation asymptotics for a second class of models that are given as particle approximations for Feynman-Kac functionals of the form
\begin{equation}\label{eq:fkfuncl}
	E \left[e^{\int_0^T c(X_1(s)) ds} g(X_1(T))\right], 
\end{equation}
where $g$ and $c$ are bounded and continuous functions and $X_1$ is given by \eqref{eq:iidsde} (with $i=1$).
Denote by $\clm_+(\RR^d)$ the space of finite measures on $\RR^d$ equipped with the topology of weak convergence (see Section \ref{sec:notandcon}), and consider the $\clc([0,T]:\clm_+(\RR^d))$-valued random variables $\nu^n$ defined as
\begin{equation}\label{eq:fkmzr}
	\nu^n(t) = \frac{1}{n}\sum_{i=1}^n e^{\int_0^t c(X^n_i(s)) ds + \kappa(n)\int_0^t \beta(X^n_i(s)) dB(s)} \delta_{X^n_i(t)}, \quad t \in [0,T],
\end{equation}
where $\{X^n_i\}$ are given by \eqref{eq:smcoup} and $\beta$ is a bounded and continuous function.
Then,  as $n \to \infty$, $\langle g, \nu^n(T)\rangle \doteq \int g(x)\, \nu^n(T)(dx)$ converges to the Feynman-Kac functional in \eqref{eq:fkfuncl}
for all choices of sequences $\kappa(n)\to 0$.  As a special case of Theorems \ref{resulttype2} and \ref{thm:FKdiffrat} 
we obtain large deviation principles for $\nu^n$, for different choices of $\kappa(n)$.

The above results correspond to the simple setting where the law of large number behavior of the system of particles is the same as that for an iid  particle system.  As noted previously, our main interest in the current work is in  interacting diffusions of the form in \eqref{eq:interemp}.
The law of large number behavior of such systems of particles is described by nonlinear equations of McKean-Vlasov type (cf. \cite{mck1,szn1}).
The large deviation behavior of the associated empirical measure process  is governed by two types of scaling, one corresponding to mean field asymptotics (as the number of particles $n \to \infty$) and the other to the Freidlin-Wentzell small noise asymptotics (as the noise intensity $\kappa(n)\to 0$). 

 Models as in \eqref{eq:interemp} are often referred to as {\it weakly interacting particle systems} and have been extensively studied, see
 \cite{mck1,brahep1,daw1,dawgar,tan2,oel,szn1,gramel3,shitan,mel} and many others. 
Originally motivated by problems in statistical physics, in recent years such systems have arisen in many applied probability problems such as stochastic networks \cite{antfrirobtib,grarob2}, information theory \cite{bormacpro,bormacpro2}, mathematical neuroscience \cite{balfasfautou}, population opinion dynamics \cite{gomgraleb}, nonlinear filtering \cite{kurxio1,del1}, and mathematical finance \cite{garpapyan,giespisow}, among others. 

In the setting where there is no common Brownian motion, i.e. $\kappa(n)=0$, large deviation principles for the empirical measure have been studied in \cite{dawgar}.  A different approach, based on certain variational representations for exponential functionals of finite dimensional Brownian motions \cite{boudup} and weak convergence arguments, was taken in \cite{buddupfis}. 
The latter paper, in contrast to \cite{dawgar}, allowed for degenerate diffusion coefficients and for a mean field interaction in the
diffusion coefficient. Large deviation properties of a system related to \eqref{eq:interemp} were studied recently in \cite{orr}, in which there is no common noise term but the independent Brownian motions $\{W_i\}$ are made to be small and vanish in the limit. 
In the systems with common noise that are considered in the current work, one needs to analyze the interplay between the contributions of two distinct sources of noise to non-typical behavior of the empirical measures. In the  rate function (see \eqref{eq:eq1135}),
this interplay is manifested through certain stochastic control problems in which there are two types of controls that play somewhat different roles in the dynamics. As already noted below \eqref{eq:eq1135b} in a simpler setting, the control that arises from the  individual noises is random and nonanticipative whereas the control from the common Brownian motion is nonrandom. In game theoretic terminology, the first control arises from the aggregated actions of the $n$ individual 
players whereas the second control corresponds to the action of a single major agent that impacts the dynamics of all $n$ players.

%%% Addition after review %%%
 
 %%%%%%%%%%%%%%%%%
Our results give a complete characterization of the asymptotic behavior for different choices of $\kappa(n)$. Specifically, taking $\kappa(n)= n^{-1/2}$, Theorems \ref{resulttype1}  and \ref{thm:diffspeed} show that
 rates of decay of $P(\mu^n \in A)$ for non-typical events $A$ are of the form $e^{-nI(A)}$, where the exponent $I(A)$ is described through a stochastic control problem with controls for both the aggregated player and the major agent.  However, when $\kappa(n)n^{1/2} \to 0$, the contribution of the common Brownian motion to deviations in the empirical measure becomes negligible and the rate function only involves the aggregated player control. Finally, when $\kappa(n)n^{1/2} \to \infty$, the decay rates of $P(\mu^n \in A)$ are slower , given as $e^{-\kappa(n)^{-2}I(A)}$, and this time the dominating contribution to deviations to the empirical measure are due to the common Brownian motion and the corresponding  stochastic control problem  is described in terms of  nonlinear Markov processes with deterministic controls. 
 %%% Addition after review %%%
 
 %%%%%%%%%%%%%%%%%
 
 %The LDP results herein have a similar flavor to those for two-scale stochastic systems, see for example the recent works \cite{husalspi,sunwanxuyan} which analyze the large deviations behavior of reaction-diffusion equations with slow and fast time scales. The results here describe related situations where different manners in which parameters approach their limit result in different rate function forms. 

In order to study rates of convergence of Feynman-Kac functionals analogous to those in \eqref{eq:fkfuncl}, we consider the following system of coupled equations:
\begin{equation}\label{eq:interempwtd}
	\begin{aligned}
dX_i^n(t) &=  b(X_i^n(t), \mu^n(t))\,dt +  \sigma(X_i^n(t), \mu^n(t))\,dW_i(t) + \kappa(n) \alpha(X_i^n(t), \mu^n(t))\,dB(t), \\ 
dA_i^n(t) &=  A_i^n(t) c(X_i^n(t), \mu^n(t))\,dt + A_i^n(t) \gamma^T(X_i^n(t), \mu^n(t))\,dW_i(t) 
		+ \kappa(n) A_i^n(t) \beta^T(X_i^n(t), \mu^n(t))\,dB(t), \\
X^n_i(0) &= x^n_i,\quad  A^n_i(0) = a^n_i, \quad  \mu^n(t) = \frac{1}{n} \sum_{i=1}^n \theta(A^n_i(t)) \delta_{X^n_i(t)}, \quad 1 \le i \le n, 
\quad t \in [0,T],
\end{aligned}
\end{equation}
where $c:\RR^d\times \clp(\RR^d) \to \RR$, $\gamma:\RR^d\times \clp(\RR^d) \to \RR^m$, and $\beta:\RR^d\times \clp(\RR^d) \to \RR^k$ are suitable maps and $\{(x^n_i,a^n_i)\}_{1\le i \le n} \subset \R^d \times \R_+$ (see Section \ref{sec:wtedempmzr} for precise conditions).
Note that in the special case where $\theta(x)=x$, $\gamma(x,\mu)=0$, and the coefficients do not depend on the empirical measure (i.e. $b(x,\mu)=b(x)$, and similarly for $\sigma, \alpha, c, \beta$),
$\mu^n$ reduces to \eqref{eq:fkmzr} (with $c$ replaced by $c - \kappa(n)^2\beta^T\beta/2$). In the general case the finite weighted empirical measures $\mu^n(t)$ take the form
\begin{equation}\label{eq:wtedempmzr}
	\mu^n(t) = \frac{1}{n}\sum_{i=1}^n \theta \left (e^{\int_0^t c_n(X^n_i(s), \mu^n(s)) ds + \int_0^t  \gamma^T(X_i^n(s), \mu^n(s))\,dW_i(s) + \kappa(n)\int_0^t \beta^T(X_i^n(s), \mu^n(s))\,dB(s)}\right)\delta_{X^n_i(t)}, 
	\end{equation}
where $c_n = c - \gamma^T\gamma/2 - \kappa(n)^2\beta^T\beta/2$, 
which covers a broad family of interacting particle models for  Feynman-Kac distribution flows (cf. \cite{del1}).
Our main result is Theorem \ref{resulttype2}, which gives a large deviation principle for $\{\mu^n\}$ in $\clc([0,T]:\clm^+(\RR^d))$ under appropriate conditions on the coefficients and the initial conditions.

The LDP results herein have a somewhat similar flavor to those for two-scale stochastic systems, see for example the recent works \cite{husalspi,sunwanxuyan} which analyze the large deviations behavior of reaction-diffusion equations with slow and fast time scales in a particular limiting regime of the parameters, as well as \cite{dupspi} which considers multiple regimes in a finite dimensional problem. As in the problems studied here, in two-scale systems
as well there are two natural parameters of interest, one (denoted as $\delta$) representing the speed of the fast system, and the other (denoted as
$\varepsilon$) representing the magnitude of the noise in the slow system. Depending on the manner in which $\delta$ and $\varepsilon$ approach $0$
in relation to each other, one expects different forms of large deviation behavior. Specifically the papers \cite{husalspi,sunwanxuyan} considered the regime $\delta/\sqrt{\varepsilon} \to 0$ while the other regimes, namely $\delta/\sqrt{\varepsilon} \to c \in (0,\infty)$ and
$\delta/\sqrt{\varepsilon} \to \infty$ were left open and are expected to be more challenging. Although there  are formal similarities with the problem studied here, it is not immediately clear whether the methods developed in the current paper can be directly used to study the harder regimes that were left unaddressed in \cite{husalspi,sunwanxuyan}.

We now make some comments on proof techniques. For an LDP for $\mu^n$ associated with the system in \eqref{eq:interemp}, the goal is to characterize the asymptotics of Laplace functionals of the form on the left side of \eqref{eq:seq}. Since $\mu^n$ is a functional of individual Brownian motions $W_i$ and the common Brownian motion $B$, using
the variational formula for exponential functionals of finite dimensional Brownian motions \cite{boudup}, one can give a stochastic control 
representation for the Laplace functional of interest (see Theorem \ref{reptype1}) that involves two types of controls.
 The first type, denoted as $u^n_i$, captures the deviations from the individual Brownian motions $W_i$ (one control for each $i$) and the other type, denoted as $v^n$, is associated with the common Brownian motion $B$. The two types of controls are scaled differently in the representation, and the analysis of this scaling, which depends on  $\kappa(n)$, 
is key to understanding  the different types of large deviation behavior for various choices of $\kappa(n)$.  
In proving the large deviation upper bound one needs to argue the convergence of the cost on the right side of \eqref{eq:seq} associated with near optimal choices of control sequences and to characterize the limits. For this, following \cite{buddupfis}, we consider certain augmented empirical measures $Q^n$ that include, in addition to particle states, the associated controls and the driving individual noises. The convergence of the costs (along subsequences) is shown by establishing the tightness of the collection $(Q^n, v^n)$. 
Tightness properties depend crucially on the rate at which $\kappa(n)\to 0$, and the forms of the limit points under different conditions on $\kappa(n)$ reveal the  different types of large deviation behavior. Next step is to characterize the form of the limit cost.
This is done by establishing that the limit points of $Q^n$ solve certain nonlinear controlled martingale problems. The controls arise from two sources, one is from the limits of $v^n$ (this is the control associated with the common noise); and the other  is from the second marginal of $Q^n$. This characterization leads to the forms of rate functions described previously.
In order to prove the lower bound one needs to construct a suitable  collection of controls for which the associated costs converge to  certain near optimal costs for the limiting stochastic control problems. This time tightness is not enough as one needs to prove convergence of (augmented) empirical measures to a specific limiting measure.
The key step in the proof of the lower bound is establishing uniqueness of weak solutions of stochastic differential equations associated with certain controlled nonlinear Markov processes. Such results are given in Lemmas \ref{pathunique} and \ref{pathunique2}. With such a uniqueness result one can then construct the desired sequence of controls and controlled processes on certain infinite product path spaces such that the associated state processes and costs converge in an appropriate manner.

Proof for large deviation asymptotics of Feynman-Kac measures as in \eqref{eq:interempwtd} rely on analyzing the properties of $\theta$.
One may  attempt to deduce this result as a corollary of large deviation results for \eqref{eq:interemp} by first establishing an LDP for 
the empirical measure of $\left(X^n_i(\cdot), A^n_i(\cdot)\right)$. However, with this approach, the conditions needed  appear to be too restrictive (see Remark \ref{rem:remwted}(a)).
We will instead analyze the weighted empirical measure $\mu^n$ in \eqref{eq:wtedempmzr} directly via variational representations for Laplace functionals
associated with $\mu^n$.
We prove the result under two different types of conditions.  The first set of conditions requires in particular that $\gamma=0$ and $\theta$ is a Lipschitz function (e.g. $\theta(x)=x$). When $\theta(x)=x$, and $\gamma=0$ is violated, a large deviation principle is not available even in the most elementary settings (see Remark \ref{rem:remwted}(c)).  The second set of conditions allows $\gamma$ to be more general but imposes logarithmic growth conditions on $\theta$. 

The paper is organized as follows. Section \ref{sec:mainresults} introduces the models, gives our precise assumptions, and presents the main results. In particular, Section \ref{sec:intempdis} considers the empirical measure problem while
Section \ref{sec:wtedempmzr} presents results for interacting particle models for Feynman-Kac functionals. The first two sections consider the case where the common noise intensity $\kappa(n)$ is of order $n^{-1/2}$, and in section \ref{sec:sizeofnoise} we present results for other choices of $\kappa(n)$ (i.e. of larger or smaller order than $n^{-1/2}$). Sections \ref{sec:proofthm1} through \ref{sec:proofsketchthm3}
contain the proofs of our main results.  The two appendices contain proofs of some auxiliary results.

\subsection{Notation and Conventions} 
\label{sec:notandcon}
We will denote by $\sC([0,T]: \R^d)$ the space of continuous functions from $[0,T]$ to $\R^d$, equipped with the sup-norm topology corresponding to the distance 
$$d(\psi_1, \psi_2) = \sup_{0\le t\le T} \|\psi_1(t) - \psi_2(t)\| \quad \mbox{for} \quad \psi_1, \psi_2 \in \sC([0,T]: \R^d).$$
For a Polish space $S$, $\sC(S)$ will denote the space of continuous functions from $S$ into $\R$, and $\sC_b(S)$ will denote the space of continuous and bounded functions from $S$ into $\R$. 
We denote by $L^2([0,T], \R^k)$ the space of functions from $[0,T]$ into $\R^k$ that are square integrable with respect to Lebesgue measure. 
Let $\mathcal{P}(S)$ denote the space of all probability measures on $S$ equipped with the usual weak convergence topology. If $S$ is a product space of the form $S_1 \times \cdots \times S_k$, then for $\Theta \in \mathcal{P}(S)$ and $i = 1, \ldots, k$, we denote by $[\Theta]_i$ the $i$th marginal of $\Theta$, which is a probability measure on $S_i$. 
Notations $[\Theta]_{(i_1, \ldots, i_r)}$, for $1\le r \le k$ and $1 \le i_1 < i_2 < \cdots <i_r \le k$, will be interpreted in a similar manner.
Let $\mathcal{M}_+(S)$ denote the space of finite positive measures on $S$, also with the topology of weak convergence. 
In particular, for $\gamma_n, \gamma \in \mathcal{M}_+(S)$, $\gamma_n \to \gamma$ under this topology if and only if for every $f \in \sC_b(S)$, $\int f \,d\gamma_n \to \int f \,d\gamma$.
For $\gamma \in \mathcal{M}_+(S)$ and a $\gamma$-integrable function $f: S\to \RR$, we will denote $\int_S f(x)\,\gamma(dx)$ as 
$\il f, \gamma\ir$.
$\sC^k(\RR^d)$ [resp. $\sC^k_c(\RR^d)$] will denote the space of functions [resp. functions with compact support] from $\RR^d$ to $\RR$ that are continuously differentiable  up to order $k$. 
For a bounded map $f: S \to \RR$, we denote $\sup_{x\in S}|f(x)|$ as $\|f\|_{\infty}$.

We call a function $I : S \to [0, \infty]$ a rate function if it has compact level sets, i.e. for each $a \in [0, \infty)$, $\{x \in S : I(x) \le a\}$ is compact in $S$. A collection $\{X_n\}_{n\in \N}$ of $S$-valued random variables is said to satisfy the {\it Laplace principle on} $S$ {\it with rate function} $I$ (and speed $a(n)\to \infty$) if for every $F \in \sC_b(S)$, 
\begin{equation}\label{LPdef}
	\lim_{n\to\infty} \frac{1}{a(n)} \log E \left[ e^{-a(n) F(X_n)} \right] = - \inf_{x\in S} \left( F(x) + I(x) \right). 
\end{equation}
The {\it Laplace upper bound} (with rate function $I$ and speed $a(n)$) refers to the inequality (for every $F \in \sC_b(S)$)
\begin{equation}\nonumber\label{LUBdef}
	\liminf_{n\to\infty}\left( -\frac{1}{a(n)} \log E\left[ e^{-a(n) F(X_n)} \right] \right) \ge \inf_{x\in S} \left( F(x) + I(x) \right), 
\end{equation}
and  the {\it Laplace lower bound}  (with rate function $I$ and speed $a(n)$) refers to the complementary inequality
\begin{equation}\nonumber\label{LLBdef}
	\limsup_{n\to\infty}\left( -\frac{1}{a(n)} \log E\left[ e^{-a(n) F(X_n)} \right] \right) \le \inf_{x\in S} \left( F(x) + I(x) \right). 
\end{equation}
It is well known \cite{dupell4} that the Laplace principle (resp. Laplace upper bound, Laplace lower bound) holds with rate function $I$ (and speed $a(n)\to \infty$) if and only if the large deviation principle (resp. large deviation upper bound, large deviation lower bound) holds with rate function $I$ (with the same speed function), where the {\it large deviation upper bound} refers to the inequality
\[
	\limsup_{n\to\infty} \frac{1}{a(n)} \log P(X_n \in F) \le - \inf_{x\in F} I(x) \mbox{ for each closed $F \subset S$},
\]
the {\it large deviation lower bound} to the inequality 
\[
	\liminf_{n\to\infty} \frac{1}{a(n)} \log P(X_n \in G) \ge -\inf_{x\in G} I(x) \mbox{ for each open $G \subset S$},
\] 
and the {\it large deviation principle} to the validity of both sets of inequalities. 
In view of this equivalence, throughout this work we will only consider Laplace asymptotics.

\section{Main Results}
\label{sec:mainresults}
In this section we introduce the models of interest, state our precise assumptions, and present the main results.

  \subsection{Diffusions Interacting Through the Empirical Distribution}
\label{sec:intempdis}

Consider a filtered probability space $(\Om, \clf, P, \{\clf_t\})$ where the filtration satisfies the usual conditions. Let $\{W_i\}_{i=1}^{\infty}$ be an iid collection of $m$-dimensional Brownian motions on this space.
Also, let $B$ be a $k$-dimensional Brownian motion that is independent of the collection $\{W_i\}_{i=1}^{\infty}$ .  We assume that, for every $s$, $\{W_i(t)-W_i(s), B(t)-B(s), i \ge 1, t \ge s\}$ is independent of $\clf_s$, so that $W_i$ and $B$ are $\{\clf_t\}$-martingales.

Consider, for $n \in \NN$, a collection of stochastic processes $\{X^n_i\}_{i=1}^n$ with sample paths in $\clc([0,T]: \RR^d)$ given by the  system of equations in \eqref{eq:interemp} where $\kappa: \NN \to \RR_+$
satisfies $\kappa(n)\to 0$ as $n\to \infty$, and $b,\sigma$, and $\alpha$ are suitable coefficients.

We will make the following assumption on the initial conditions.

\begin{condition}\label{empmzrinit} There exists $\xi_0\in \mathcal{P}(\R^d)$ such that for all $\xi_0$-integrable $f: \R^d \to \R$,
	\[
		\lim_{n\to\infty} \frac{1}{n}\sum_{i=1}^n f(x_i^n) = \il f, \xi_0 \ir.	
	\]
	Furthermore,
	$\sup_{n\ge 1} \frac{1}{n} \sum_{i=1}^n \|x_i^n\|^2 < \infty. 
	$

\end{condition}
We will require the coefficients $b, \alpha$, and $\sigma$ to be  Lipschitz continuous.  In order to state this condition precisely, we recall the bounded-Lipschitz metric on the space of measures. Recall that $\mathcal{M}_+(\R^d)$ denotes the space of positive measures on $\R^d$  equipped with the weak topology. This topology can be metrized by the bounded Lipschitz metric
\[
	d_{BL}(\nu_1, \nu_2) = \sup_{f\in BL(\R^d)} \left| \il f, \nu_1 \ir - \il f, \nu_2 \ir \right|, \quad \nu_i \in \mathcal{M}_+(\R^d), \quad i=1,2,
\]
where 
\[
	BL(\R^d) = \left\{f\in \sC(\R^d) : \|f\|_{\infty} \le 1 \text{ and $f$ is Lipschitz with  Lipschitz constant bounded by 1}\right\}.  
\]
The following is the main condition on the coefficients.

\begin{condition}\label{Lip} The map $b$ is Lipschitz and the maps $\sigma, \alpha$ are bounded and Lipschitz from $\RR^d \times \mathcal{M}_+(\R^d)$ to $\RR^d$, $\RR^{d\times m}$, and $\RR^{d\times k}$ respectively. Namely,
	there is a  $K \in (0,\infty)$ such that for each $x, y\in \R^d$ and $\mu, \nu \in \mathcal{M}_+(\R^d)$, 
	
	\begin{enumerate}
	
	\item $\displaystyle  \|\sigma(x,\mu)\|^2 + \|\alpha(x, \mu)\|^2  \le K^2$, and 
	
	\item  $\displaystyle \|b(x, \mu) - b(y, \nu)\| + \|\sigma(x,\mu) - \sigma(y, \nu)\| + \|\alpha(x, \mu) - \alpha(y, \nu)\| \le K\left( \|x-y\| + d_{BL}(\mu,\nu)\right)$.
	
	\end{enumerate}

\end{condition}
For Theorem \ref{resulttype1} we can replace $\clm_+(\R^d)$ with $\clp(\R^d)$ in the above condition, however it is convenient to formulate the condition as above in order to have a common set of conditions for Theorems \ref{resulttype1} and \ref{resulttype2}.
For the LDP we will assume in addition that the  diffusion coefficient $\sigma$ depends on the state of the system only through the empirical measure:

\begin{condition}\label{sigmagamma} For  $x \in \R^d$ and $\mu \in \mathcal{M}_+(\R^d)$, $\sigma(x, \mu) = \sigma(\mu)$.
\end{condition}

Under Condition \ref{Lip} it follows by standard arguments  that for each $n$ there is a unique pathwise solution of \eqref{eq:interemp}. 
Abusing notation, let $\mu^n$ be  a random variable with values in $\clp(\clc([0,T]: \RR^d))$ defined as $\mu^n \doteq \frac{1}{n}\sum_{i=1}^n \delta_{X^n_i}$. Note that $\mu^n(s)$ is the (random) marginal distribution at time instant $s$ associated with $\mu^n$. 
We will occasionally denote the map $t \mapsto \mu^n(t)$, as $\mu^n(\cdot)$ which is viewed as a $\clp(\RR^d)$-valued stochastic process with continuous sample paths or, equivalently, a random variable with values in $\clc([0,T]:\clp(\RR^d))$.

Our first main result gives a large deviation principle for $\mu^n$ in $\clp(\clc([0,T]: \RR^d))$. We begin by introducing the associated rate function.
This function will be described in terms of solutions to certain controlled McKean-Vlasov equations which we now introduce. 
Recall the Polish spaces $\clr$ and $\clr_1$ of relaxed controls from the Introduction.

Given $\varphi \in L^2([0, T] : \R^k)$ and a continuous map $\nu: [0,T] \to \mathcal{P}(\R^d)$, consider the controlled nonlinear SDE $\mathcal{S}_1[\varphi, \nu]$, on some filtered probability space $(\bar\Om, \bar\clf, \bar P, \{\bar\clf_t\})$, equipped with an $m$-dimensional $\bar\clf_t$-Brownian motion $W$:
\begin{equation} \label{eq: controlled SDE}
	\sS_1[\varphi, \nu] \doteq \left\{
	\begin{aligned}
	d\bar{X}(t) &= b(\bar{X}(t), \nu(t))\,dt + \left(\int_{\R^m} \sigma(\bar{X}(t), \nu(t))y\,\rho_t(dy)\right)\,dt +\sigma(\bar X(t), \nu(t))\,dW(t) \\
			 &\quad + \alpha(\bar{X}(t), \nu(t)) \varphi(t)\,dt, \\
	\bar{X}(t) &\sim \nu(t), \quad t\in [0,T], \quad \nu(0) = \xi_0,   
	\end{aligned} \right. 
\end{equation}
where $\xi_0 \in \mathcal{P}(\R^d)$ is as in Condition \ref{empmzrinit}. In the above equation $\rho$ is an $\clr_1$-valued random variable such that $\rho([0,t]\times A)$ is  $ \bar \clf_t$-measurable for every $A \in \clb(\RR^m)$
and $t \in [0,T]$, and $\bar X$ is an $\bar \clf_t$-adapted stochastic process with sample paths in $\clc([0,T]: \RR^d)$. The notation $\bar X(t)\sim \nu(t)$ signifies that
$\bar X(t)$ has probability distribution $\nu(t)$, i.e. $\bar P \circ \bar X(t)^{-1} = \nu(t)$.
We note that $\sS_1[\varphi, \nu]$ is driven by two types of controls, the control $\varphi$ is a deterministic function whereas  $\rho$ 
represents a random control in the dynamics.

A triple $(\bar{X}, \rho, W)$ that solves $\sS_1[\varphi, \nu]$ for a given $\varphi$ and $\nu$ can be viewed as a $\mathcal{Z}_1$-valued random variable, where
\[
	\mathcal{Z}_1 \doteq \mathcal{X}\times \mathcal{R}_1\times \mathcal{W}, \;  \mathcal{X} \doteq \sC([0,T]:\R^d), \mbox{ and } \mathcal{W} \doteq \sC([0,T]:\R^m).
\]
The distribution of $(\bar{X}, \rho, W)$ on $\mathcal{Z}_1$ is an element of $\mathcal{P}(\mathcal{Z}_1)$ and is called a {\em weak solution} of the controlled SDE $\sS_1[\varphi, \nu]$. Define 
\[
	\mathcal{P}_2(\mathcal{Z}_1) \doteq \left\{ \Theta \in \mathcal{P}(\mathcal{Z}_1) : E_\Theta\left[ \int_{\R^m\times [0,T]} \|y\|^2\,\rho(dy\,dt) \right] < \infty \right\}, 
\]
where in the above display $E_\Theta$ denotes expectation on $(\mathcal{Z}_1, \sB(\mathcal{Z}_1), \Theta)$ and, abusing notation, $\rho$ is the second coordinate map on $(\mathcal{Z}_1, \sB(\mathcal{Z}_1))$, i.e. 
\[
	\rho(x, r, w) \doteq r, \quad (x, r, w) \in \mathcal{Z}_1.
\]
Note that, the above expectation can be written as
\[
	E_\Theta\left[ \int_{\R^m\times [0,T]} \|y\|^2\,\rho(dy\,dt) \right] = \int_{\mathcal{R}_1} \int_{\R^m\times[0,T]} \|y\|^2\,r(dy\,dt)\,[\Theta]_2(dr). 
\]

For  $\Theta\in \mathcal{P}(\mathcal{Z}_1)$, let $\nu_\Theta: [0, T] \to \mathcal{P}(\R^d)$ be defined as
\[
	\nu_\Theta(t)(B) \doteq \Theta\left\{(x, r, w)\in \mathcal{Z}_1 : x(t)\in B\right\}, \quad B\in \sB(\R^d). 
\]
Note that if $\Theta$ is a weak solution of $\sS_1[\varphi, \nu]$, then $\nu(t)= \nu_{\Theta}(t)$ for all $t \in [0,T]$. For a given $\varphi \in L^2([0,T]:\R^k)$, let $\mathcal{E}_1[\varphi]$ denote the subset of $\mathcal{P}_2(\mathcal{Z}_1)$ given as
\begin{equation}\label{eq:cle1phi}
	\mathcal{E}_1[\varphi] \doteq \left\{ \Theta \in \mathcal{P}_2(\mathcal{Z}_1) :  \Theta\;\text{is a weak solution to $\sS_1[\varphi, \nu_{\Theta}]$}\right\}.  \end{equation}
Then the candidate rate function for the LDP for $\mu^n$ is 
\begin{equation}\label{eq:eq1135}
	I_1(\nu) \doteq \inf_{\varphi \in L^2([0,T]:\R^k)} \left\{\inf_{\Theta\in \mathcal{E}_1[\varphi] : [\Theta]_1=\nu} E_{\Theta}\left[ \frac{1}{2} \int_{\R^m\times[0,T]} \|y\|^2\,\rho(dy\,dt) \right] + \frac{1}{2\la^2} \int_0^T \|\varphi(t)\|^2\,dt \right\}, 
\end{equation}
for $\nu \in \mathcal{P}(\clx)$, where $\la \in (0,\infty)$ is introduced below. 

The following is the first main result of this work.  It gives an LDP in the case $\kappa(n)$ is of the order $n^{-1/2}$. Later in Section \ref{sec:sizeofnoise} we will consider the large deviation behavior when $\kappa(n)$ is of smaller or higher order than $n^{-1/2}$.
Part 1 below gives a law of large numbers result while part 2 establishes a large deviation principle.
Denote the element $\delta_{\{0\}}(dy)\, dt$ of $\clr$ as $r^o$.
\begin{theorem}\label{resulttype1} 
	Suppose that Conditions \ref{empmzrinit}, \ref{Lip} hold and that $\kappa(n)\to 0$ as $n\to \infty$.
\begin{enumerate}
\item   There is a  $\mu^* \in \mathcal{P}(\clx)$ such that $\mu^n \to \mu^*$ in probability. Furthermore, $\mu^*$ can be characterized as the first marginal
$[\Theta]_1$ of $\Theta$, where $\Theta$ is the unique element in $\clp(\clz_1)$ that is a weak solution of $\cls_1[0, \nu_{\Theta}]$ and satisfies $[\Theta]_2 = \delta_{r^o}$.
\item Suppose in addition that   Condition \ref{sigmagamma} is satisfied and that $\sqrt{n}\kappa(n) \to \la \in (0,\infty)$. Then $\{\mu^n\}_{n\in\N}$ satisfies a large deviation principle on $\mathcal{P}(\clx)$ with speed $n$ and rate function $I_1$.
\end{enumerate}
\end{theorem}
Proof of Theorem \ref{resulttype1}  will be given in Section \ref{sec:proofthm1}.  
\begin{remark}\label{rem:contprinc}
	Since the map $\nu \mapsto \{t\mapsto \nu(t)\}$ from $\clp(\clx)$ to $\clc([0,T]:\clp(\R^d))$ is a continuous map, we have by the contraction principle that $\mu^n(\cdot)$ regarded as a sequence of random variables with values in $\clc([0,T]:\clp(\R^d))$ satisfies an LDP as well.
\end{remark}

\subsection{Interacting Particle Systems for Feynman-Kac Functionals}
\label{sec:wtedempmzr}
In this section we consider a setting where the interaction term is given in terms of a weighted empirical measure of the states of the particles and where the weights are governed by another system of stochastic equations.  Let $(\Om, \clf, P, \{\clf_t\})$, $\{W_i\}$, $B$ be as in Section \ref{sec:intempdis}. Consider for $n \in \NN$, a collection of
stochastic processes $\{(X^n_i, A^n_i)\}_{i=1}^n$ with sample paths in $\clc([0,T]: \RR^d\times \RR_+)$ given by the  system of equations
in \eqref{eq:interempwtd}.
Here 
$\theta: \R_+ \to \R_+$,
 $\kappa: \NN \to \RR_+$, and $b, \sigma, \alpha, c, \gamma$, and $\beta$ are suitable maps.
 Note that $\mu^n(t)$ in this set of equations can also be represented as on the right side of \eqref{eq:wtedempmzr}.
In addition to Condition \ref{Lip} on the coefficients, we will assume the following condition.
 \begin{condition}
\label{LipA} The maps $c, \gamma, \beta$ are bounded and Lipschitz from $\RR^d \times \mathcal{M}_+(\R^d)$ to $\RR^d$, $\RR^{m}$, and $\RR^{k}$ respectively. Namely,
there is a  $K \in (0,\infty)$ such that for each $x, y\in \R^d$ and $\mu, \nu \in \mathcal{M}_+(\R^d)$, 
\begin{enumerate}
	\item $\displaystyle  \|c(x, \mu)\|^2 + \|\gamma(x,\mu)\|^2 + \|\beta(x, \mu)\|^2 \le K^2$, and 
	
	\item  $\displaystyle  \|c(x, \mu) - c(y, \nu)\| + \|\gamma(x,\mu) - \gamma(y, \nu)\| + \|\beta(x, \mu) - \beta(y, \nu)\| \le K\left( \|x-y\| + d_{BL}(\mu,\nu)\right)$.
\end{enumerate}
\end{condition}
The weights in the random measure $\mu^n(t)$ are determined through the map $\theta$ on which we  make the following assumption.

\begin{condition}\label{theta} Either one of the following hold:
	\begin{enumerate}[(a)]
	\item $\theta \in \sC^2(\R_+)$ and 
	\begin{equation}\label{eq:thetderbds}
		\sup_{x\in \R_+} |\theta'(x) x|  +  \sup_{x\in \R_+} |\theta''(x) x^2| < \infty. 
	\end{equation} 
	
	\item There is a  $L \in (0,\infty)$ such that $|\theta(x) - \theta(y)| \le L|x-y|$ for all $x, y\in \R_+$. 
	\end{enumerate}
\end{condition}
Condition \ref{theta}(b) simply says that $\theta$ is a Lipschitz function. It is easily checked that under Condition \ref{theta}(a), $\theta$ is Lipschitz as well. The latter  condition,  in addition, implies  an (at most) logarithmic growth on $\theta$.

Under Conditions \ref{Lip}, \ref{LipA}, and \ref{theta}, there is a unique pathwise solution to the system of equations in \eqref{eq:interempwtd}.
Although the proof is standard, we provide a sketch in Appendix \ref{sec:llnwted}.
The object of interest is the stochastic process $\{\mu^n(t)\}_{t\in [0,T]}$ which is regarded as a  random variable with values in $\sC([0,T] : \mathcal{M}_+(\R^d))$.
Our second main result gives a large deviation principle for $\mu^n(\cdot)$ in this path space. 
We introduce two additional conditions that will be needed for this result.
For the initial values $\{(a^n_i, x^n_i)\}$ in \eqref{eq:interempwtd} we will assume the following in addition to Condition \ref{empmzrinit}:
\begin{condition}\label{wtdmzrinit} 
	There exists $\eta_0 \in \mathcal{P}(\R^d\times \R_+)$ such that for all $\eta_0$-integrable $g: \R^d \times \R_+ \to \R$, 
		\[
			\lim_{n\to\infty}\frac{1}{n}\sum_{i=1}^n g(x^n_i, a_i^n) = \il g, \eta_0 \ir. 
		\]
	Furthermore,
		\[
		\sup_{n\ge 1} \frac{1}{n} \sum_{i=1}^n (a_i^n)^2 < \infty \quad \mbox{and}\quad \sup_{n\ge 1} \frac{1}{n} \sum_{i=1}^n (\log a_i^n)^- < \infty.
	\]
\end{condition}
\noindent Note that when both Conditions \ref{empmzrinit} and \ref{wtdmzrinit} hold, we have $[\eta_0]_1 = \xi_0$, where $[\eta_0]_1$ is the marginal distribution of $\eta_0$ on $\R^d$.

Finally, for the large deviations result, in addition to Condition \ref{sigmagamma}, we will assume that the diffusion coefficient $\gamma$ depends on the state of the system only through the empirical measure, namely:

\begin{condition}\label{gammaconst}
	For  $x \in \R^d$ and $\mu \in \mathcal{M}_+(\R^d)$, $\gamma(x, \mu) = \gamma(\mu)$.
\end{condition}

We now present the rate function that will govern the LDP for $\{\mu^n(\cdot)\}$.
 Given $\varphi \in L^2([0, T] : \R^k)$ and  $\nu \in \clc([0,T]: \mathcal{M}_+(\R^d))$ as in Section \ref{sec:intempdis}, consider the controlled nonlinear SDE $\mathcal{S}_2[\varphi, \nu]$ given on some filtered probability space $(\bar\Om, \bar\clf, \bar P, \{\bar\clf_t\})$, equipped with an $m$-dimensional $\bar\clf_t$-Brownian motion $W$:
\begin{equation}\label{controlledSDEtype2}
	\sS_2[\varphi, \nu] \doteq \left\{
	\begin{aligned}
	d\bar{X}(t) &= b(\bar{X}(t), \nu(t))\,dt + \left(\int_{\R^m} \sigma(\bar{X}(t), \nu(t))y\,\rho_t(dy)\right)\,dt +\sigma(\bar{X}(t), \nu(t))\,dW(t) \\
			 &\quad + \alpha(\bar{X}(t), \nu(t))\varphi(t)\,dt, \\
	d\bar{A}(t) &= \bar{A}(t) c(\bar{X}(t), \nu(t))\,dt + \left(\int_{\R^m} \bar{A}(t)\gamma^T(\bar{X}(t), \nu(t))y\,\rho_t(dy) \right)\,dt\\
	   		&\quad +  \bar{A}(t)\gamma^T(\bar{X}(t), \nu(t))\,dW(t) 
			+ \bar{A}(t)\beta^T(\bar{X}(t), \nu(t))\varphi(t)\,dt, \\ 
	\il f, \nu(t) \ir&= \bar{E}[\theta(\bar{A}(t))f(\bar{X}(t))] \mbox{ for every $f\in \sC_b(\R^d)$}, \quad t\in [0,T], \quad (\bar X(0), \bar A(0)) \sim \eta_0, 
	\end{aligned} \right. 
\end{equation}
where $\bar E$ denotes expectation with respect to $\bar P$. Here $\rho$ is as in Section \ref{sec:intempdis}, and $\bar X$ and  $\bar A$ are $\bar \clf_t$-adapted stochastic processes with sample 
paths in $\clc([0,T]:\R^d)$ and $\clc([0,T]:\R_+)$, respectively, such that
\[
	\bar E \left[ \sup_{0\le t \le T} \theta(\bar A(t)) \right] < \infty. 
\]

A quadruple $(\bar{X}, \bar{A}, \rho, W)$ that solves $\sS_2[\varphi, \nu]$ is a $\mathcal{Z}_2$-valued random variable, where 
\[
	\mathcal{Z}_2 \doteq \mathcal{X} \times \mathcal{Y} \times \mathcal{R}_1 \times \mathcal{W}, \quad \mathcal{Y} \doteq \sC([0,T] : \R_+),
\]
and $\clx, \clw, \clr_1$ are as before.
The distribution of $(\bar{X}, \bar{A}, \rho, W)$ on $\clz_2$ is an element of $\mathcal{P}(\mathcal{Z}_2)$ and is called a weak solution of $\sS_2[\varphi, \nu]$. Let 
\[
	\mathcal{P}_2(\mathcal{Z}_2) \doteq \left\{ \Theta \in \mathcal{P}(\mathcal{Z}_2) : E_\Theta\left[ \int_{\R^m\times [0,T]} \|y\|^2\,\rho(dy\,dt) \right] < \infty, \;  E_\Theta \left[\sup_{0\le t \le T} \theta(\bar A(t))\right] <\infty \right\}.
\]
Note that if $\Theta \in \mathcal{P}_2(\mathcal{Z}_2)$ then $\nu_{\Theta} \in \clc([0,T]: \clm_+(\R^d))$, where
$\nu_{\Theta}$ is defined as
\begin{equation}\label{eq:eq503}	
	\il f, \nu_\Theta(t) \ir \doteq E_\Theta\left[ \theta(\bar{A}(t)) f(\bar{X}(t)) \right] \quad \mbox{for}\quad f\in \mathcal{C}_b(\R^d),\quad t\in [0,T], 
\end{equation} and if such a $\Theta$
 is a weak solution of $\sS_2[\varphi, \nu]$, then, for every $t \in [0,T]$,
 $\nu(t)= \nu_{\Theta}(t)$.
Given $\varphi \in L^2([0,T]:\R^d)$, let 
\[
	\mathcal{E}_2[\varphi] \doteq \{ \Theta \in \mathcal{P}_2(\mathcal{Z}_2) : \text{$\Theta$ is a weak solution to $\sS_2[\varphi, \nu_\Theta]$} \}. 
\]
The candidate rate function is given as
\begin{equation}\label{eq:ratefuni2}
	I_2(\nu) \doteq \inf_{\varphi \in L^2([0,T]:\R^k)}\left\{ \inf_{\Theta\in \mathcal{E}_2[\varphi] : \nu_\Theta=\nu}  E_{\Theta}\left[ \frac{1}{2} \int_{\R^m\times[0,T]} \|y\|^2\,\rho(dy\,dt) \right] + \frac{1}{2\la^2} \int_0^T \|\varphi(t)\|^2\,dt \right\},
\end{equation}
for $\nu \in \sC([0,T] : \mathcal{M}_+(\R^d))$.
The following is the second main result of this work.  As in Section \ref{sec:intempdis} here we only consider the case where $\kappa(n)$ is of order $n^{-1/2}$.
Values of $\kappa(n)$ of higher or lower order than $n^{-1/2}$ will be considered in Section \ref{sec:sizeofnoise}.

Once more, the first part of the theorem below gives a law of large numbers (LLN) and the second part establishes an LDP.
The proof is given in Section \ref{sec:proofthm2}.
\begin{theorem}\label{resulttype2} Suppose that Conditions  \ref{empmzrinit}, \ref{Lip}, \ref{LipA}, \ref{theta}, and \ref{wtdmzrinit}  hold and that $\kappa(n)\to 0$ as $n\to \infty$.
\begin{enumerate}
\item  There is a  $\mu^* \in \sC([0,T] : \mathcal{M}_+(\R^d))$ such that $\mu^n \to \mu^*$ in probability.
Furthermore, $\mu^*$ can be characterized as the map $t \mapsto \nu_{\Theta}(t)$,
 where $\Theta$ is the unique element in $\clp(\clz_2)$ that is a weak solution of $\cls_2[0, \nu_{\Theta}]$ and satisfies $[\Theta]_3 = \delta_{r^o}$.
	
\item Suppose that $\sigma$ and $\gamma$ satisfy Conditions \ref{sigmagamma} and \ref{gammaconst}, and either 
 (i) $\theta$ satisfies Condition \ref{theta}(a), or 
 (ii) $\theta$ satisfies Condition \ref{theta}(b) and $\gamma \equiv 0$. Also suppose that $\sqrt{n}\kappa(n) \to \la \in (0,\infty)$. 
Then $\{\mu^n\}_{n\in \N}$ satisfies a large deviation principle on $\sC([0,T]:\mathcal{M}_+(\R^d))$ with speed $n$ and rate function 
$I_2$.
	\end{enumerate}
\end{theorem}

\begin{remark} 
	\label{rem:remwted}
\begin{enumerate}[(a)]
\item Consider the empirical measure of $\{X^n_i(s), A^n_i(s)\}$ on $\RR^d \times \RR_+$, given as
$$\hat \mu^n(s) \doteq \frac{1}{n} \sum_{i=1}^n \delta_{\left(X^n_i(s), A^n_i(s)\right)}.$$
Then the system in equation \eqref{eq:interempwtd} can be written in form of a system as in \eqref{eq:interemp} in which $X^n_i$ is replaced by
the pair $(X^n_i, A^n_i)$. With such a rewriting, one may attempt to deduce Theorem \ref{resulttype2} as a corollary of Theorem \ref{resulttype1}. However, with this reformulation, the conditions needed for Theorem \ref{resulttype1} are too restrictive. In particular, conditions assumed in the statement of Theorem \ref{resulttype2} will, in general, not imply  the conditions of Theorem \ref{resulttype1} (with the new coefficients obtained through the reformulation). Specifically, requiring Conditions \ref{Lip} and \ref{sigmagamma}
for  the reformulated system will say that $\theta$ is bounded and $\gamma\equiv 0$.

\item  A minor modification of the proof of Theorem \ref{resulttype2}  shows in fact that the joint empirical measure $\frac{1}{n} \sum_{i=1}^n \delta_{(X^n_i, A^n_i)}$ satisfies an LDP
on $\clp(\sC([0,T]: \R^d\times \R_+))$. Note that since $\theta$ may be unbounded, the map $\Theta \mapsto \nu_{\Theta}$ is not continuous (in fact in general not even well defined) on all of $\clp(\sC([0,T]: \R^d\times \R_+))$ and so one cannot deduce an LDP for $\mu^n$ from that of the joint empirical measure by a direct application of the contraction principle.  In any case, the amount of work needed to establish the LDP for $\mu^n$ is about the same as that needed for the LDP for the joint empirical measure.

\item In Theorem \ref{resulttype2}, for the case where $\theta$ satisfies Condition \ref{theta}(b), we require that $\gamma\equiv 0$.
The reason for this restrictive requirement on $\gamma$ can be seen as follows.
Consider the simplest example of a $\theta$ satisfying Condition \ref{theta}(b), namely $\theta(x)=x$. Consider also the simplest form of a non-zero $\gamma$ in \eqref{eq:interempwtd}, namely $\gamma(x,\mu) = \gamma \in \RR^m \setminus \{0\}$. Also suppose that $\beta \equiv 0$, $c(x,\mu)\equiv c \in \RR$, and that $a^n_i=1$ for all $i,n$. Then the second set of equations in \eqref{eq:interempwtd} reduces to
\begin{align*}
	dA_i^n(t) = c A_i^n(t)\,dt + \gamma A_i^n(t)\,dW_i(t), \;\;
	A_i^n(0) = 1, \quad1\le i \le n. 
\end{align*}
Namely,
\[
	A_i^n(t) = \exp\left\{ \left(c - \frac{\gamma^2}{2} \right)t + \gamma W_i(t) \right\}.  
\]
In this case, an LDP for $\mu^n(\cdot)$ will in particular say (by the contraction principle) that the sequence $\{\mu^n(1)(\RR^d)\}$ satisfies an LDP. However the latter
is just an LDP for the  empirical mean of iid random variables, $\{A^n_i(1)\}$, namely 
$\frac{1}{n}\sum_{i=1}^n A_i^n(1)$, which is the subject of 
 Cram\'{e}r's theorem.
 However the key condition for this theorem, namely the finiteness of the moment generating function in a neighborhood of the origin, fails to hold in this case. 
\end{enumerate}
\end{remark}

\subsection{Intensity of the common noise}
\label{sec:sizeofnoise}
The LDP in Theorems \ref{resulttype1} and \ref{resulttype2} are established under the condition that
 the common noise intensity $\kappa(n)$ is $O(1/\sqrt{n})$. If this intensity approaches $0$ at a different rate, the form of the rate function is expected to be different.
In this section we discuss such results. We will consider two cases: Case I: $\sqrt{n}\kappa(n)\to 0$, and Case II: $\sqrt{n}\kappa(n)\to \infty$.
Let
$\mathcal{E}_1[\varphi]$ for a given $\varphi \in L^2([0,T]:\R^k)$ be as in Section \ref{sec:intempdis}. In order to define the rate function in the second case, we consider, for a $\varphi$ as above and a  $\nu \in \clc([0,T] : \mathcal{P}(\R^d))$,  the controlled nonlinear SDE $\tilde{\mathcal{S}}_1[\varphi, \nu]$, on some filtered probability space $(\bar\Om, \bar\clf, \bar P, \{\bar\clf_t\})$, equipped with a $m$-dimensional $\bar\clf_t$-Brownian motion $W$:
\begin{equation}\label{newlimitSDE}
	\tilde{\sS}_1[\varphi, \nu] \doteq \left\{
	\begin{aligned}
	d\bar{X}(t) &= b(\bar{X}(t), \nu(t))\,dt + \sigma(\bar{X}(t), \nu(t))\,dW(t) + \alpha(\bar{X}(t), \nu(t))\varphi(t)\,dt, \\ 
	\bar{X}(t) &\sim \nu(t), \quad t\in [0,T],  \quad  \nu(0) = \xi_0.
	\end{aligned} \right. 
\end{equation}
The difference between the above equation and the equation in \eqref{eq: controlled SDE} is the absence of the control term $\rho_t$ on the right side of \eqref{newlimitSDE}.
The distribution, on $\clx\times\clw$, of a pair $(\bar{X}, W)$ that solves \eqref{newlimitSDE} for a given $\varphi$ and $\nu$ will be called a weak solution of $\tilde{\sS}_1[\varphi, \nu]$.

For a  $\varphi \in L^2([0,T]:\R^k)$, let 
\begin{equation}\label{eq:cle1phitil}
	\tilde{\mathcal{E}}_1[\varphi] \doteq \left\{ \Theta \in \clp(\clx\times\clw) :  
	\Theta\;\text{is a weak solution to $\tilde{\sS}_1[\varphi, \nu_{\Theta}]$}\right\}.  
\end{equation}
For $\nu \in \mathcal{P}(\sC([0,T]:\R^d))$, we denote the map $t \mapsto \nu(t)$, once more as $\nu$. 
The following result gives an LDP when $\kappa(n)$ is different from $O(1/\sqrt{n})$. 
Recall that we assume $\kappa(n)\to 0$ as $n\to \infty$. Also recall the space $\clx = \sC([0,T]:\R^d)$.

\begin{theorem} \label{thm:diffspeed} Let $\{\mu^n\}_{n\in\N}$ be as in Section \ref{sec:intempdis}.
	Suppose that Conditions  \ref{empmzrinit} and \ref{Lip} hold.
\begin{enumerate}[(i)]
\item Suppose in addition that Condition \ref{sigmagamma} is satisfied. If $\sqrt{n}\kappa(n) \to 0$ as $n\to \infty$, then $\{\mu^n\}$ satisfies an LDP on $\mathcal{P}(\clx)$ with speed $n$ and
rate function $\tilde I_{1,0}$ given as
\begin{equation}\label{eq:itillar}
	\tilde{I}_{1,0}(\nu) \doteq \inf_{\Theta \in \mathcal{E}_1[0] : [\Theta]_1 = \nu} E_\Theta\left[ \frac{1}{2} \int_{\R^m \times [0,T]} \|y\|^2\,\rho(dy\,dt) \right],\quad \nu \in  \mathcal{P}(\clx).
\end{equation}

\item If $\sqrt{n}\kappa(n) \to \infty$ as $n\to \infty$, then $\{\mu^n\}$ satisfies an LDP on $\mathcal{P}(\clx)$ with speed $\kappa(n)^{-2}$ and
rate function $\tilde I_{1,\infty}$ given as 

\begin{equation}\label{eq:itilsma}
	\tilde{I}_{1,\infty}(\nu) \doteq \inf_{\varphi \in L^2([0,T]:\R^k)}\left\{ \inf_{\Theta \in \tilde{\mathcal{E}}_1[\varphi] : [\Theta]_1 = \nu}\; \frac{1}{2} \int_0^T \|\varphi(t)\|^2\,dt \right\}, \quad \nu \in \mathcal{P}(\mathcal{X}). 
\end{equation}
\end{enumerate}
\end{theorem}
The proof of Theorem \ref{thm:diffspeed} is very similar to that of Theorem \ref{resulttype1} and therefore we will only provide a sketch and leave the details to the reader. This sketch is given in Section \ref{sec:proofsketchthm3}.
 \begin{remark}
 	\label{rem:bmcase} Consider the special case discussed in the Introduction (see \eqref{eq:smcoup}) in which the interaction only comes through the common Brownian motion. For this special case
	the results in Theorems \ref{resulttype1} and \ref{thm:diffspeed} (by some minor proof modifications) say the following.  Suppose that the coefficients $b, \sigma$ and $\alpha$ in \eqref{eq:smcoup} are Lipschitz.
	Also, suppose first that $\sqrt{n}\kappa(n)\to \lambda \in (0,\infty)$. Then $\{\mu^n\}$ as introduced in \eqref{eq:empsmdep} satisfies an LDP in $\clp(\clx)$ with speed $n$ and rate function $I$ defined in \eqref{eq:eq1135b}. If $\sqrt{n}\kappa(n)\to 0$, then $\{\mu^n\}$  satisfies an LDP  with speed $n$ and rate function $\tilde I_{1,0}$ as in \eqref{eq:itillar} and where $\cle_1[\cdot]$ is as introduced below 
	\eqref{eq:empsmdep}.  Finally, when $\sqrt{n}\kappa(n)\to \infty$, then $\{\mu^n\}$ satisfies an LDP  with speed $\kappa(n)^{-2}$ and rate function $\tilde I_{1,\infty}$ given simply as
	$$\tilde I_{1,\infty}(\nu) = \inf_{\varphi} \left\{\frac{1}{2} \int_0^T \|\varphi(t)\|^2\,dt\right\},$$
	where the infimum is taken over all $\varphi \in L^2([0,T]:\R^k)$ such that the solution $\{X\}$ of the controlled SDE
	$$X(t) = x + \int_0^t b(X(s))\, ds + \int_0^t \sigma(X(s)) \,dW(s)  
	+ \int_0^t \alpha(X(s)) \varphi(s) \,ds, \quad t \in [0,T],$$
	has probability law $\nu$.
 \end{remark}
One can also give an analogue of Theorem \ref{thm:diffspeed} for Feynman-Kac weighted measures of the form in Section \ref{sec:wtedempmzr}. We state such a result and leave proof details to the reader.

Consider, for a $\varphi \in L^2([0,T]:\R^k)$  and a  $\nu \in \clc([0,T] : \mathcal{M}_+(\R^d))$,  the controlled nonlinear SDE 
$\tilde{\mathcal{S}}_2[\varphi, \nu]$, on some filtered probability space $(\bar\Om, \bar\clf, \bar P, \{\bar\clf_t\})$, equipped with a $m$-dimensional $\bar\clf_t$-Brownian motion $W$:
\begin{equation}\label{controlledSDEtype2tilde}
	\tilde\sS_2[\varphi, \nu] \doteq \left\{
	\begin{aligned}
	d\bar{X}(t) &= b(\bar{X}(t), \nu(t))\,dt +\sigma(\bar{X}(t), \nu(t))\,dW(t)  + \alpha(\bar{X}(t), \nu(t))\varphi(t)\,dt, \\
	d\bar{A}(t) &= \bar{A}(t) c(\bar{X}(t), \nu(t))\,dt 
	   		 +  \bar{A}(t)\gamma^T(\bar{X}(t), \nu(t))\,dW(t) 
			+ \bar{A}(t)\beta^T(\bar{X}(t), \nu(t))\varphi(t)\,dt, \\ 
	\il f, \nu(t) \ir&= \bar{E}[\theta(\bar{A}(t))f(\bar{X}(t))] \mbox{ for every $f\in \sC_b(\R^d)$}, \quad t\in [0,T], \quad (\bar X(0), \bar A(0)) \sim \eta_0, 
	\end{aligned} \right. 
\end{equation}
where $\bar E$ denotes expectation with respect to $\bar P$.
The distribution, on $\clx\times\cly\times\clw$, of  $(\bar{X}, \bar{A}, W)$ that solves \eqref{controlledSDEtype2tilde} for a given $\varphi$ and $\nu$ will be called a weak solution of $\tilde{\sS}_2[\varphi, \nu]$.
For a  $\varphi \in L^2([0,T]:\R^k)$, let 
\begin{equation*}
	\tilde{\mathcal{E}}_2[\varphi] \doteq \left\{ \Theta \in \clp(\clx\times\cly\times\clw) :  
	\Theta\;\text{is a weak solution to $\tilde{\sS}_2[\varphi, \nu_{\Theta}]$}\right\}.  
\end{equation*}
\begin{theorem}
	\label{thm:FKdiffrat}  Let $\{\mu^n\}_{n\in \NN}$ be as in section \ref{sec:wtedempmzr}. Suppose that Conditions  \ref{empmzrinit}, \ref{Lip}, \ref{sigmagamma}, \ref{LipA},  \ref{wtdmzrinit}, and \ref{gammaconst}  hold.
	Also suppose that, either
	 (i) $\theta$ satisfies Condition \ref{theta}(a), or 
	 (ii) $\theta$ satisfies Condition \ref{theta}(b) and $\gamma \equiv 0$.
	 \begin{enumerate}[(i)]
	 \item  If $\sqrt{n}\kappa(n) \to 0$ as $n\to \infty$, then 
	 $\{\mu^n\}$ 
	 satisfies an LDP on $\clc([0,T]: \clm_+(\RR^d))$ with speed $n$ and
	 rate function $\tilde I_{2,0}$ given as
	 \begin{equation}\label{eq:itillar2}
	 	\tilde{I}_{2,0}(\nu) \doteq \inf_{\Theta \in \mathcal{E}_2[0] : \nu_{\Theta} = \nu} E_\Theta\left[ \frac{1}{2} \int_{\R^m \times [0,T]} \|y\|^2\,\rho(dy\,dt) \right],\quad \nu \in  \clc([0,T]: \clm_+(\RR^d)).
	 \end{equation}
	 \item If $\sqrt{n}\kappa(n) \to \infty$ as $n\to \infty$, then $\{\mu^n\}$ satisfies an LDP on $\clc([0,T]: \clm_+(\RR^d))$  with speed $\kappa(n)^{-2}$ and
	 rate function $\tilde I_{2,\infty}$ given as 

	 \begin{equation}\label{eq:itilsma2}
	 	\tilde{I}_{2,\infty}(\nu) \doteq \inf_{\varphi \in L^2([0,T]:\R^k)}\left\{ \inf_{\Theta \in \tilde{\mathcal{E}}_2[\varphi] : \nu_{\Theta} = \nu}\; \frac{1}{2} \int_0^T \|\varphi(t)\|^2\,dt \right\}, \quad 
		\nu \in \clc([0,T]: \clm_+(\RR^d)). 
	 \end{equation}
	 	 \end{enumerate}
	 \end{theorem}

\section{Proof of Theorem \ref{resulttype1}. }
\label{sec:proofthm1}
Part 1 follows by a standard argument (cf. \cite{szn91}), however for completeness we give a sketch in the Appendix.
We now consider part 2.

From the well known equivalence between an LDP and a Laplace principle (cf. \cite{dupell4}) it suffices to show that the function $I_1$ introduced in \eqref{eq:eq1135}
is a rate function and for every $F \in \clc_b(\clp(\clx))$ the following upper and lower bounds are satisfied.\\

\noindent {\bf Laplace Upper Bound}
\begin{equation}\label{eq:eq1137u}
	\liminf_{n\to\infty} -\frac{1}{n} \log E\left[e^{-n F(\mu^n)}\right] \ge \inf_{\nu \in \mathcal{P}(\mathcal{X})} \left[ F(\nu) + {I}_1(\nu) \right].
\end{equation}
\noindent {\bf Laplace Lower Bound}
\begin{equation}\label{eq:eq1137l}
	\limsup_{n\to\infty} -\frac{1}{n} \log E\left[e^{-n F(\mu^n)}\right] \le \inf_{\nu \in \mathcal{P}(\mathcal{X})} \left[ F(\nu) + {I}_1(\nu) \right].
\end{equation}

The upper bound is shown in Section \ref{sec:lapupp2.1} and the lower bound is treated in Section \ref{sec:laplow2.1}. 
The upper bound proof does not require Condition \ref{sigmagamma} and we present an argument assuming only Conditions \ref{empmzrinit} and \ref{Lip}.
The proof of the statement that $I_1$
is a rate function is very similar to that of the upper bound and thus we only give a brief sketch which appears in Section \ref{sec:ratefun2.1}.
Proofs rely on a certain stochastic control representation for the Laplace functional on the left side of \eqref{eq:eq1137u} and \eqref{eq:eq1137l} which we now present.

Given some filtered probability space 
 $(\bar\Om, \bar\clf, \bar P, \{\bar\clf_t\})$ that supports iid $m$-dimensional Brownian motions  $\{W_i\}_{i=1}^{\infty}$ and   a $k$-dimensional Brownian motion $B$ that is independent of the collection $\{W_i\}_{i=1}^{\infty}$ and such that for every $s$, $\{W_i(t)-W_i(s), B(t)-B(s), i \ge 1, t \ge s\}$ is independent of $\bar\clf_s$, 
denote by $\cla^{1,n}$ the class of $\bar \clf_t$-progressively measurable processes $u: [0, T]\times \Omega \to \R^{n m}$ such that
\[
	\bar{E}\left[\int_0^T \|u(s)\|^2 \,ds\right] < \infty. 
\] 
For $u \in \mathcal{A}^{1,n}$, we will write $u = (u_1, \ldots, u_n)$, where $u_i$ is the $i$th component of $u$ and is $m$-dimensional. For $M \in (0, \infty)$, let 
\[
	S_M \doteq \left\{ v \in L^2([0, T] : \R^k) : \int_0^T \|v(s)\|^2\,ds \le M \right\}.
\]
This space will be equipped with the weak topology under which it is a compact  space. Note that 
\[
	\bigcup_{M\in \N} S_M = L^2([0, T] : \R^k). 
\]
Also let 
\[
	\mathcal{A}^2_M \doteq \left\{ \text{Progressively measurable $\R^k$-valued processes $v$ such that}\; v \in S_M\;\text{$\bar{P}$-a.s.} \right\}, 
\]
and 
\[
	\mathcal{A}^2 \doteq \left\{ \text{Progressively measurable $\R^k$-valued processes $v$ such that}\; \bar{E}\left[\int_0^T \|v(s)\|^2 \,ds\right] < \infty\right\}.
\]

For  $(u, v) \in \mathcal{A}^{1,n}\times \mathcal{A}^2$,  consider the  controlled analogue of the system in \eqref{eq:interemp}, driven by controls  $(u, v)$:
\begin{equation}\label{controlledn}
\begin{aligned}
	d\bar{X}^n_i(t) &= b(\bar{X}^n_i(t), \bar{\mu}^n(t))\,dt +  \sigma(\bar{X}^n_i(t), \bar{\mu}^n(t))u_i(t)\,dt + \alpha(\bar{X}^n_i(t), \bar{\mu}^n(t))v(t)\,dt \\ 
			       &\quad  +\sigma(\bar{X}^n_i(t),\bar{\mu}^n(t))\,dW_i(t) + \kappa(n) \alpha(\bar{X}^n_i(t), \bar{\mu}^n(t))\,dB(t), \\ 
	\bar{X}^n_i(0) &= x^n_i, \quad 1\le i \le n, 
\end{aligned}
\end{equation}
where $\bar{\mu}^n(t) = \frac{1}{n}\sum_{i=1}^n \delta_{\bar{X}_i^n(t)}$. Using the Lipschitz and boundedness conditions on the coefficients it is easy to check that the above system of equations has a unique solution.
We also consider the empirical measure $\bar{\mu}^n = \frac{1}{n}\sum_{i=1}^n \delta_{\bar{X}_i^n}$ 
which is a $\clp(\clx)$-valued random variable.
A form of the following representation was first shown in \cite{boudup}. The representation given below, that allows for an arbitrary filtered probability space on the right side was given in 
\cite{buddup3} (see also \cite{buddupfis}). All expectations will be denoted by $E$ unless specified otherwise. 

\begin{theorem}\label{reptype1} For any $F \in \sC_b(\mathcal{P}(\mathcal{X}))$ and for each $n\in \N$, 
\begin{equation}\label{eq:seq}
	-\frac{1}{n}\log E\left[e^{-n F(\mu^n)}\right] =\inf_{(u, v) \in \mathcal{A}^{1,n}\times \mathcal{A}^2} E\left[ \frac{1}{2n} \sum_{i=1}^n \int_0^T \|u_i(t)\|^2\,dt + \frac{1}{2n\kappa(n)^2}\int_0^T \|v(t)\|^2\,dt + F(\bar{\mu}^n) \right].
\end{equation}
Furthermore, for every $\delta > 0$, there is an $M < \infty$ such that for each $n\in \N$, 
\begin{equation}\label{eq:seqb}
	-\frac{1}{n}\log E\left[e^{-n F(\mu^n)}\right] \ge\inf_{(u, v) \in \mathcal{A}^{1,n}\times \mathcal{A}_M^2} E\left[ \frac{1}{2n} \sum_{i=1}^n \int_0^T \|u_i(t)\|^2\,dt + \frac{1}{2n\kappa(n)^2}\int_0^T \|v(t)\|^2\,dt + F(\bar{\mu}^n) \right] - \delta. 
\end{equation}
\end{theorem}

\noindent We now use the above result to complete the proof of \eqref{eq:eq1137u} and \eqref{eq:eq1137l}.

\subsection{Laplace Upper Bound}
\label{sec:lapupp2.1}
Throughout this  section we assume that  Conditions \ref{empmzrinit} and \ref{Lip} are satisfied.
As noted previously, the upper bound proof does not require Condition \ref{sigmagamma} and so this condition will not be used in this section.

Fix $F \in \sC_b(\mathcal{P}(\mathcal{X}))$ and  $\delta \in (0,1)$. From Theorem \ref{reptype1} there is an $M < \infty$ such that for each $n\in \N$, one can find $(u^n, v^n) \in \mathcal{A}^{1,n} \times \mathcal{A}^2_M$ such that 
\begin{equation}\label{eq:seqc}
	-\frac{1}{n}\log E\left[e^{-n F(\mu^n)}\right]  \ge E\left[ \frac{1}{2n} \sum_{i=1}^n \int_0^T \|u_i^n(t)\|^2\,dt + \frac{1}{2n\kappa(n)^2}\int_0^T \|v^n(t)\|^2\,dt + F(\bar{\mu}^n) \right] - \delta,
\end{equation}
where $\bar{\mu}^n = \frac{1}{n}\sum_{i=1}^n \delta_{\bar{X}_i^n}$ and $\bar{X}_i^n$ are given by \eqref{controlledn} (repalcing 
$(u,v)$ with $(u^n, v^n)$). We will next show that 
\begin{equation}\label{eq:varlb}
\begin{aligned}
	&\liminf_{n\to\infty} E\left[ \frac{1}{2n}\sum_{i=1}^n \int_0^T \|u_i^n(t)\|^2\,dt + \frac{1}{2n\kappa(n)^2}\int_0^T \|v^n(t)\|^2\,dt + F(\bar{\mu}^n) \right] \\ 
	&\ge  \inf_{\varphi \in L^2([0,T]:\R^k)} \inf_{\Theta\in \mathcal{E}_1[\varphi]} \left( E_\Theta\left[ \frac{1}{2} \int_{\R^m\times[0,T]} \|y\|^2\,\rho(dy\,dt) \right] + \frac{1}{2\la^2} \int_0^T \|\varphi(t)\|^2\,dt + F([\Theta]_1) \right). 
\end{aligned}
\end{equation}
Since $\delta\in (0,1)$ is arbitrary, the inequality in \eqref{eq:eq1137u} is immediate from \eqref{eq:varlb} on using the definition of $I_1$ in \eqref{eq:eq1135}.

We now prove \eqref{eq:varlb}. From \eqref{eq:seqc} it follows that
\begin{equation}
	\label{eq:eq128}
	\sup_{n \in \NN} E\left[ \frac{1}{2n}\sum_{i=1}^n \int_0^T \|u_i^n(t)\|^2\,dt + \frac{1}{2n\kappa(n)^2}\int_0^T \|v^n(t)\|^2\,dt\right] \le 2\|F\|_{\infty} +1.
\end{equation}
The following lemma shows that under such a uniform boundedness property, one has the tightness of certain key occupation measures.

\begin{lemma}\label{tightness} 
	Suppose for some $M \in (0,\infty)$, $\{(u^n, v^n)\}_{n\in \N}$ is a sequence with $(u^n, v^n) \in \sA^{1,n}\times \sA^2_M$ for each $n$, and suppose $\{u^n\}_{n\in \N}$ satisfies, for some $L\in (0,\infty)$,
	\begin{equation}\label{eq:eq135}
		\sup_{n \in \NN} E\left[ \frac{1}{n} \sum_{i=1}^n \int_0^T \|u_i^n(t)\|^2\,dt \right] \le L.
	\end{equation}
	 Define  $\mathcal{P}(\mathcal{Z}_1)$-valued random variables
	\begin{equation}\label{eq:bigempmzr}
		Q^n(A\times R\times C) \doteq \frac{1}{n}\sum_{i=1}^n \delta_{\bar{X}_i^n}(A)\delta_{\rho_i^n}(R)\delta_{W_i}(C),\; A\times R\times C \in \sB(\mathcal{Z}_1),
	\end{equation}
	 where $\bar{X}^n_i$ is defined as in \eqref{controlledn} (repalcing $(u,v)$ with $(u^n, v^n)$), and
	\begin{equation}\label{eq:eq1027}
		\rho_i^n(E\times B) \doteq \int_B\delta_{u_i^n(t)}(E)\,dt,\; E\in \sB(\R^m), B\in \sB([0,T]).
	\end{equation}
Then $\{(Q^n, v^n)\}_{n\in \N}$ is tight as a sequence of $\mathcal{P}(\mathcal{Z}_1) \times S_M$-valued random variables.
\end{lemma}

\begin{proof} Since $S_M$ is compact, tightness of $\{v^n\}$ is immediate. The third marginals of $Q^n$ are clearly tight since $W_i$ are iid. The first marginal of $Q^n$, namely  $[Q^n]_1$, equals $\bar{\mu}^n$. For each $n$ let $\gamma^n = E[\bar{\mu}^n]$. For tightness of $\{\bar{\mu}^n\}_{n\in \N}$, it suffices to prove that the family $\{\gamma^n\}_{n\in \N}$ of measures on $\mathcal{X}$ is relatively compact.
	
By using the growth properties on the coefficients it follows that, for some $c_1\in (0,\infty)$ and all $n\in \NN$,
\begin{equation}\label{eq:eq1048}
	E\left[ \sup_{0\le s\le T} \|\bar{X}_i^n(s)\|^2 \right] \le c_1\left( 1 + \|x_i^n\|^2 + E\left[ \int_0^T \|u_i^n(s)\|^2\,ds\right] \right). 
\end{equation}
Thus, 
\begin{align}\label{eq:finiteness} \nonumber
	\sup_{n\in \N} \int_{\mathcal{X}} \sup_{0\le t\le T} \|\psi(t)\|^2 \,d\gamma^n(\psi) 
				&= \sup_{n\in \N} \frac{1}{n} \sum_{i=1}^n E\left[ \sup_{0\le t\le T} \|\bar{X}_i^n(t)\|^2 \right]   \nonumber \\
				&\le c_1 \sup_{n\in \N} \left( 1 + \frac{1}{n}\sum_{i=1}^n\|x_i^n\|^2 + E\left[\frac{1}{n} \sum_{i=1}^n \int_0^T \|u_i^n(s)\|^2\,ds\right]  \right) 
				< \infty,
\end{align}
where the last inequality is from \eqref{eq:eq135} and Condition \ref{empmzrinit}. 

Next note that for any $\eps \in (0,1)$ and $t \in [0, T-\eps]$, 
\begin{align*}
	\|\bar{X}_i^n(t+\eps) - \bar{X}_i^n(t)\|^2 &\le c_2\left( \left\| \int_t^{t +\eps}b(\bar{X}_i^n(s), \bar{\mu}^n(s))\,ds \right\|^2 +\left\| \int_t^{t+\eps}\sigma(\bar{X}_i^n(s), \bar{\mu}^n(s))u_i^n(s)\,ds\right\|^2  \right. \\
	&\quad + \left\|\int_t^{t+\eps} \alpha(\bar{X}_i^n(s), \bar{\mu}^n(s))v^n(s)\,ds\right\|^2 +\left\| \int_t^{t+\eps} \sigma(\bar{X}_i^n(s), \bar{\mu}^n(s))\,dW_i(s) \right\|^2 \\
	&\quad \left. + \kappa(n)^2 \left\|\int_t^{t+\eps} \alpha(\bar{X}_i^n(s), \bar{\mu}^n(s))\,dB(s) \right\|^2 \right).
\end{align*}
Thus for any stopping time $\tau$ taking values in $[0, T-\eps]$, using the Cauchy-Schwarz inequality, the linear growth of $b$, and the boundedness of $\alpha$ and $\sigma$, 
\[
	E \left[ \left\| \bar{X}_i^n(\tau+\eps) - \bar{X}_i^n(\tau) \right\|^2 \right] \le c_3 \eps \left( 1 + E\left[ \sup_{0\le s\le T} \|\bar{X}_i^n(s)\|^2\right] + E\left[ \int_0^T \|u_i^n(s)\|^2\,ds\right] \right), 
\]
where the constant $c_3$  does not depend on $n$, $\eps$, or the stopping time $\tau$. 
Denoting by $\mathcal{T}_\eps$  the collection of all  stopping times $\tau$,  with respect to the canonical filtration generated by the coordinate process on $\clx$, taking values in $[0, T-\eps]$, we now have
\begin{align*}
	&\sup_{\tau \in \mathcal{T}_\eps} \int_\mathcal{X} \|\varphi(\tau+\eps) - \varphi(\tau)\|^2\,d\gamma^n(\varphi) \\
	&\quad \le c_3\eps \left( 1 + \frac{1}{n}\sum_{i=1}^nE\left[ \sup_{0\le s\le T} \|\bar{X}_i^n(s)\|^2\right] + E\left[ \frac{1}{n}\sum_{i=1}^n \int_0^T \|u_i^n(s)\|^2\,ds\right] \right)\\
	&\quad \le c_3\eps \left( 1 + \int_{\mathcal{X}} \sup_{0\le t\le T} \|\psi(t)\|^2 \,d\gamma^n(\psi)  + L \right).
\end{align*}
Using \eqref{eq:finiteness} in the above display,
$$\limsup_{\eps \to 0}\sup_{n\in \NN}\sup_{\tau \in \mathcal{T}_\eps} \int_\mathcal{X} \|\varphi(\tau+\eps) - \varphi(\tau)\|^2\,d\gamma^n(\varphi) = 0.$$
Thus from the Aldous-Kurtz tightness criterion, we have that the collection $\gamma^n$ is relatively compact which, as noted previously, gives the tightness of the collection
$\{\bar \mu^n\} = \{ [Q^n]_1\}$.

Finally we consider the second marginals of $Q^n$. Define 
\[
	g(r) \doteq \int_{\R^m\times [0,T]} \|y\|^2\, r(dy\,dt), \; r\in \mathcal{R}_1.
\]
We note that $g$ has compact level sets. Indeed, for $c \in \R_+$, let $L_c = \{r\in \mathcal{R}_1 : g(r) \le c\}$ denote the corresponding level set. By Chebyshev's inequality, 
\[
	\sup_{r \in L_c} r \left( \left\{ y\in \R^d : \|y\| > M \right\} \times [0,T] \right) \le \sup_{r\in L_c} \frac{g(r)}{M^2} 
							\le \frac{c}{M^2}  
							\to 0
\]
as $M \to \infty$. This shows that $L_c$ is relatively compact in $\clr$. Let $\{r_n\} \subset L_c$ be a sequence that converges
in $\clr$ to some $r^*$. By Fatou's lemma, $g(r^*) \le c$, and so $r^* \in L_c$. 
Also, by the uniform integrability that follows from
\[
	\sup_{n\ge 1} \int_{\R^m\times [0,T]} \|y\|^2\,r_n(dy\,dt) = \sup_{n\ge 1} g(r_n) \le c, 
\]
the moments of $r_n$ also converge to the moments of $r^*$. Thus  $r_n \to r^*$ in $\clr_1$, establishing compactness of $L_c$ in $\clr_1$.
Let $G: \mathcal{P}(\mathcal{R}_1) \to [0,\infty]$ be given as
\[
	G(\theta) \doteq \int_{\mathcal{R}_1} g(r)\,\theta(dr). 
\]
Then $G$ is a tightness function on $\mathcal{P}(\mathcal{R}_1)$ (namely it has relatively compact level sets), and thus to establish the tightness of the second marginals $\{[Q^n]_2\}$, it suffices to show that 
\begin{equation}\label{tightfunc}
	\sup_{n\ge 1} E[G([Q^n]_2)] < \infty. 
\end{equation}
For each $n \in \N$, 
\begin{align*}
	E[G([Q^n]_2)] &= E\left[ \int_{\mathcal{R}_1} g(r)\,[Q^n]_2(dr) \right] = E\left[ \frac{1}{n} \sum_{i=1}^n \int_{\R^m\times [0,T]} \|y\|^2\, \rho_i^n(dy\,dt) \right] \\
			      &= E\left[ \frac{1}{n} \sum_{i=1}^n \int_0^T \|u_i^n(t)\|^2\,dt \right] 
			      \le L.
\end{align*}
This proves \eqref{tightfunc} and completes the proof of the tightness of $\{[Q^n]_2\}$.
The result follows.
\end{proof}

The next lemma characterizes the weak limit points of the  sequence $(Q^n, v^n)$. Recall the collection $\cle_1[\varphi]$ from \eqref{eq:cle1phi}.

\begin{lemma}\label{limit1}  Suppose, for some $M \in (0,\infty)$, $\{(u^n, v^n)\}_{n\in \N}$ is a sequence with $(u^n, v^n) \in \sA^{1,n}\times \sA^2_M$ for each $n$, and such that $\{u^n\}_{n\in \N}$ satisfies \eqref{eq:eq135} with some $L\in (0,\infty)$.
Let $Q^n$ be defined as in Lemma \ref{tightness}.
If $(Q^n, v^n)$ converges in distribution, along some subsequence, to $(Q, v)$, then $Q\in \sE_1[v]$ a.s.
\end{lemma}

\begin{proof} 
	Let $(Q, v)$ be a weak limit point of $(Q^n, v^n)$ given on some probability space $(\Om^*, \clf^*, P^*)$.
Note that by Fatou's lemma, 
\begin{align}\label{eq:fatlemapp}
	E^*\left[ \int_{\sR_1}\int_{\R^m\times [0,T]} \|y\|^2\,r(dy\,dt)\,[Q]_2(dr) \right] \le \liminf_{n\to\infty} E\left[ \frac{1}{n} \sum_{i=1}^n \int_0^T \|u_i^n(t)\|^2\,dt \right] 
			\le L, 
\end{align}
Thus $Q\in \mathcal{P}_2(\mathcal{Z}_1)$, $P^*$-a.s. Also, since $\int_0^T \|v(s)\|^2 ds \le M$, $v \in L^2([0,T]:\RR^k)$ $P^*$-a.s.
To complete the proof, we need to argue that 
for $P^*$-a.e. $\omega \in \Omega^*$, $Q(\omega)$ is a weak solution to $\sS_1[v(\omega), \nu_{Q(\omega)}]$.

Denote the canonical coordinate variables on $\clz_1$ by $(z,r,w)$. By Condition \ref{empmzrinit},
$[Q^n]_1 \circ (z(0))^{-1} \to \xi_0$ weakly, which shows that, for $P^*$-a.e. $\om$, under $Q(\omega)$, $z(0)$ has distribution $\xi_0$.
Denote by $\{\sH_t\}_{0\le t\le T}$ the canonical filtration on $(\clz_1, \clb(\clz_1))$, namely
\begin{equation}\label{eq:eq428}
	 \sH_t \doteq \sigma \{ z(s), w(s), r (A \times [0,s]), \, A \in \clb(\RR^m),\, s \le t\}.
\end{equation}
For $f\in \sC_c^2(\R^d\times \R^m)$,  $\varphi \in L^2([0,T]:\R^k)$, and $\Theta \in \mathcal{P}(\mathcal{Z}_1)$, consider the process $\{M_{f,\varphi}^{\Theta}(t)\}_{0\le t\le T}$  defined on the probability space $(\mathcal{Z}_1, \sB(\mathcal{Z}_1), \Theta)$ by 
\begin{equation}\label{eq:eq359}
\begin{aligned}
	M_{f,\varphi}^\Theta(t, (z, r, w)) &\doteq  f(z(t), w(t)) - f(z(0), 0) - \int_0^t \int_{\R^m} \mathcal{L}^\Theta_s (f)(z(s), y, w(s))\,r_s(dy)\,ds \\
		&\quad - \int_0^t \left\il \alpha(z(s), \nu_{\Theta}(s)) \varphi(s), \nabla_x f(z(s), w(s)) \right\ir \,ds,
\end{aligned}
\end{equation}
where 
\begin{equation}
	\label{eq:eq358}
\begin{aligned}
	\mathcal{L}_s^\Theta(f)(x,y,w) &\doteq \il b(x,\nu_\Theta(s)) + \sigma(x,\nu_\Theta(s))y, \nabla_x f(x,w) \ir   \\
		&\quad + \frac{1}{2}\sum_{j,j'=1}^d (\sigma\sigma^T)_{jj'}(x,\nu_\Theta(s)) \frac{\partial^2f}{\partial x_j \partial x_{j'}}(x,w)  + \frac{1}{2}\sum_{j=1}^m \frac{\partial^2f}{\partial w_j^2}(x,w)  \\
		&\quad + \sum_{j=1}^d \sum_{j'=1}^m \sigma_{jj'}(x,\nu_\Theta(s))\frac{\partial^2 f}{\partial x_j \partial w_{j'}}(x,w)
\end{aligned}
\end{equation}
for $x \in \RR^d$ and  $y, w \in \R^m$. 
Let, for  $B \in (0,\infty)$, $\zeta_B: \R^m \to \R^m$ be such that $\zeta_B$ is a continuous function with compact support satisfying $\zeta_B(y)=y$ for $\|y\|\le B$ and $\|\zeta_B(y)\| \le \|y\| + 1$ for every $y\in \R^m$.
It will be convenient to also consider, 	along with $\mathcal{L}_s^{\Theta}$, the operator  $\mathcal{L}_s^{\Theta,B}$ which is defined by replacing $y$ on the right side 
of \eqref{eq:eq358}  with 
$\zeta_B(y)$. Similarly, define $M_{f,\varphi}^{\Theta,B}$ by replacing $\mathcal{L}^\Theta_s$ in \eqref{eq:eq359} with $\mathcal{L}^{\Theta,B}_s$.

It suffices to show that for each $f \in \sC^2_c(\R^d\times \R^m)$,  any   time instants $0\le t_0<t_1\le T$,  and any  $\Psi \in \sC_b(\mathcal{Z}_1)$ that is measurable with respect to the sigma field $\sH_{t_0}$, we have, 
\begin{equation}
	\label{eq:eq823}
	E_{Q(\omega)}\left[ \Psi \left( M_{f, v(\omega)}^{Q(\omega)}(t_1) - M_{f, v(\omega)}^{Q(\omega)}(t_0) \right) \right] = 0 \;
	\mbox{ for }
	P^* \mbox{-a.e. } \om \in \Om^* .
\end{equation}
In the rest of the proof we suppress $\omega$ from the notation.
Fix a choice of $(t_0, t_1, \Psi, f)$ and define 
$\Phi  : \mathcal{P}(\mathcal{Z}_1)\times S_M \to \R$
by 
\begin{equation}\label{Phidef}
	\Phi(\Theta, \varphi) = E_{\Theta}\left[ \Psi \left( M_{f,\varphi}^{\Theta}(t_1) - M_{f, \varphi}^{\Theta}(t_0) \right) \right]. 
\end{equation}
Also, for every $B \in (0,\infty)$, define $\Phi_B$ by replacing $M_{f,\varphi}^{\Theta}$ with $M_{f,\varphi}^{\Theta, B}$ in the definition of $\Phi$. 
  We will now show that (a) for every $B \in (0,\infty)$, $\Phi_B$ is a bounded and continuous map on $\clp(\clz_1)\times S_M,$  (b)
$\sup_n E^*|\Phi_B(Q^n, v^n) - \Phi(Q^n, v^n)| \to 0$ and $ E^*|\Phi_B(Q, v) - \Phi(Q, v)| \to 0$ as $B\to \infty$, and (c) $\Phi(Q^n, v^n) \to 0$ in probability as $n \to \infty$.
The statement in  \eqref{eq:eq823} is an immediate consequence of (a)-(c).

We first show (a). Let $(\Theta_n, \varphi_n) \to (\Theta, \varphi)$ in $\clp(\clz_1) \times S_M$ as $n\to \infty$.
Note that this means $\int_0^T \il \varphi_n(s) - \varphi(s), h(s) \ir\,ds \to 0$ for all $h \in L^2([0,T]:\R^k)$. Thus, 
\begin{equation}
	\label{eq:eq449}
\begin{aligned}
	\left| \Phi_B(\Theta, \varphi_n) - \Phi_B(\Theta, \varphi)\right| &\le \|\Psi\|_{\infty}E_\Theta \left| \int_{t_0}^{t_1} \left\il \alpha(z(s), \nu_{\Theta}(s))(\varphi_n(s)-\varphi(s)), \nabla_x f(z(s), w(s)) \right\ir \,ds \right|\\
&=	\|\Psi\|_{\infty} E_\Theta \left| \int_{0}^{T} 1_{[t_0, t_1]}(s)\left\il (\varphi_n(s) - \varphi(s)), \alpha^T(z(s), \nu_{\Theta}(s))\nabla_x f(z(s), w(s)) \right\ir \,ds \right|\\
	 &\to 0
\end{aligned} 
\end{equation}
as $n \to \infty$, where the last convergence follows from the dominated convergence theorem upon observing that 
$h(\cdot) = \alpha^T(z(\cdot), \nu_{\Theta}(\cdot))\nabla_x f(z(\cdot), w(\cdot)) 1_{[t_0, t_1]}(\cdot)$ is in $L^2([0,T]:\R^k)$.
Next note that
\begin{equation}\label{eq:eq446}
	\sup_{\bar \varphi \in S_M} \left|E_{\Theta_n}\left[ \Psi\cdot \left( M^{\Theta,B}_{f, \bar\varphi}(t_1) - M^{\Theta,B}_{f, \bar\varphi}(t_0) \right) \right] - E_{\Theta}\left[ \Psi\cdot \left( M^{\Theta, B}_{f, \bar\varphi}(t_1) - M^{\Theta, B}_{f, \bar\varphi}(t_0) \right) \right]  \right| \to 0
\end{equation}
as $n \to \infty$. 
This convergence is a consequence of the following facts: (i) Continuity and boundedness of the map
$(z,r,w) \mapsto f(z(t), w(t)) - \int_0^t \int_{\R^m} \mathcal{L}^{\Theta,B}_s (f)(z(s), y, w(s))\,r_s(dy)\,ds$, (ii) the continuity and boundedness of the map
$(z,w) \mapsto \alpha^T(z(s), \nu_{\Theta}(s))\nabla_x f(z(s), w(s))$, (iii) the property that $\sup_{\bar \varphi \in S_M} \int_0^T \|\bar \varphi(s)\|^2 ds \le M$, and Cauchy-Schwarz inequality.
Next, for some $c_1\in (0, \infty)$ (possibly depending on $B$), and all $t\in [0,T]$, $\bar\varphi \in S_M$
\begin{align*}
	\left| M^{\Theta_n, B}_{f, \bar\varphi}(t) - M^{\Theta, B}_{f, \bar\varphi}(t) \right| &\le \int_0^T \int_{\R^m} \left| \mathcal{L}_s^{\Theta_n, B}(f)(z(s), y, w(s)) - \mathcal{L}_s^{\Theta, B}(f)(z(s), y, w(s)) \right| \,r_s(dy)\,ds \\
						&\quad + \int_0^T \left| \left \il (\alpha(z(s), \nu_{\Theta_n}(s)) - \alpha(z(s), \nu_{\Theta}(s))\bar\varphi(s), \nabla_x f(z(s), w(s)) \right \ir \right| \,ds \\
						&\le c_1 \left(\int_0^T d_{BL}(\nu_{\Theta_n}(s), \nu_{\Theta}(s))^2\,ds\right)^{1/2}. 
\end{align*}
Since, for every $s \in [0,T]$, $\nu_{\Theta_n}(s) \to \nu_{\Theta}(s)$, we now have 
\begin{equation}\label{eq:eq447}
	\sup_{\bar \varphi \in S_M}\left|E_{\Theta_n}\left[ \Psi\cdot \left( M^{\Theta_n,B}_{f, \bar\varphi}(t_1) - M^{\Theta_n,B}_{f, \bar\varphi}(t_0) \right) \right] - E_{\Theta_n}\left[ \Psi\cdot \left( M^{\Theta,B}_{f, \bar\varphi}(t_1) - M^{\Theta,B}_{f, \bar\varphi}(t_0) \right) \right]\right| \to 0
\end{equation}
as $n \to \infty$. Combining \eqref{eq:eq446} and \eqref{eq:eq447}
\[
	\sup_{\bar \varphi \in S_M}\left| \Phi_B(\Theta_n, \bar\varphi) - \Phi_B(\Theta, \bar\varphi) \right| \to 0
\]
as $n \to \infty$. Together with \eqref{eq:eq449}, the above display completes the proof of (a).

In order to see (b), note that, for some $c_2 \in (0, \infty)$, and every $n \in \NN$,
\begin{align}
	E|\Phi_B(Q^n, v^n) - \Phi(Q^n, v^n)| &\le c_2 E \left[ E_{Q^n} \left[\int_0^T \left\|\int_{\R^m} (y-\zeta_B(y)) \, r_s(dy)\right\| ds \right] \right]\nonumber\\
	&= c_2 E \left[ \frac{1}{n} \sum_{i=1}^n \int_0^T \|u^n_i(s) - \zeta_B(u^n_i(s))\| ds \right] \nonumber\\
	&\le \frac{c_2}{B} E \left[\frac{1}{n} \sum_{i=1}^n \int_0^T 2(\|u^n_i(s)\| + 1)\|u^n_i(s)\| ds \right] \le \frac{4c_2(L+T)}{B}.\label{eq:4lt}
\end{align}
The first statement in (b) is now immediate. The second statement in (b) is shown similarly by using \eqref{eq:fatlemapp}.

Finally we consider (c).
By the definition of $Q^n$ and since $\nu_{Q^n}(s) = \bar{\mu}^n(s)$,  
\begin{align*}
	&\Phi(Q^n, v^n) \\
	&= E_{Q^n}\left[ \Psi \left( M_{f,v^n}^{Q^n}(t_1) - M_{f,v^n}^{Q^n}(t_0) \right) \right] \\
				    &= \frac{1}{n} \sum_{i=1}^n \Psi(\bar{X}_i^n, \rho_i^n, W_i) \cdot \left( M_{f,v^n}^{Q^n}(t_1, (\bar{X}_i^n, \rho_i^n, W_i)) - M_{f,v^n}^{Q^n}(t_0, (\bar{X}_i^n, \rho_i^n, 
					W_i)) \right) \\
				    &=  \frac{1}{n} \sum_{i=1}^n \Psi(\bar{X}_i^n, \rho_i^n, W_i) 
					\cdot \left( \phantom\int\hspace{-0.35cm} f(\bar{X}_i^n(t_1), W_i(t_1))  - f(\bar{X}_i^n(t_0), W_i(t_0))       \right. \\
				    & \quad \left.- \int_{t_0}^{t_1} \mathcal{L}_s^{Q^n}(f)(\bar{X}_i^n(s), u_i^n(s), W_i(s))\,ds - \int_{t_0}^{t_1}  
					\left\il \alpha(\bar{X}_i^n(s), \bar{\mu}^n(s)) v^n(s), \nabla_x f(\bar{X}_i^n(s), W_i(s)) \right\ir \,ds \right),
\end{align*}
By It\^{o}'s formula, for each $i$, a.s.
\begin{align*}
	&f(\bar{X}_i^n(t_1), W_i(t_1))  - f(\bar{X}_i^n(t_0), W_i(t_0)) \\
	& = \int_{t_0}^{t_1} \mathcal{L}_s^{Q^n}(f)(\bar{X}_i^n(s), u_i^n(s), W_i(s))\,ds +  \int_{t_0}^{t_1} \left\il \alpha(\bar{X}_i^n(s), \bar{\mu}^n(s)) v^n(s), \nabla_x f(\bar{X}_i^n(s), W_i(s)) \right\ir  \,ds\\
			& \quad + \int_{t_0}^{t_1} \left[ \nabla_x f(\bar{X}_i^n(s), W_i(s))\right]^T \sigma(\bar{X}_i^n(s), \bar{\mu}^n(s))\,dW_i(s) + \int_{t_0}^{t_1} \left[ \nabla_w f(\bar{X}_i^n(s), W_i(s)) \right]^T \,dW_i(s)\\
			& \quad + \kappa(n)\int_{t_0}^{t_1} \left[ \nabla_x f(\bar{X}_i^n(s), W_i(s))\right]^T \alpha(\bar{X}_i^n(s), \bar{\mu}^n(s))\,dB(s) \\
			& \quad + \frac{\kappa(n)^2}{2} \int_{t_0}^{t_1} \tr\left( (\alpha\alpha^T)(\bar{X}_i^n(s), \bar{\mu}^n(s)) D^2_x f(\bar{X}_i^n(s), W_i(s)) \right) \,ds. 
\end{align*}
Writing $\Psi_i^n = \Psi(\bar{X}_i^n, \rho_i^n, W_i)$,  we then have
\begin{align*}
	\Phi(Q^n, v^n) &=  \frac{1}{n}\sum_{i=1}^n \Psi_i^n   \int_{t_0}^{t_1} \left[ \nabla_x f(\bar{X}_i^n(s), W_i(s))\right]^T \sigma(\bar{X}_i^n(s),\bar{\mu}^n(s))\,dW_i(s) \\
	&\quad+ \frac{1}{n}\sum_{i=1}^n \Psi_i^n  \int_{t_0}^{t_1} \left[ \nabla_w f(\bar{X}_i^n(s), W_i(s)) \right]^T \,dW_i(s)  + \clt^n_1,
\end{align*}	
where using the fact that   $\kappa(n)\to 0$ as $n\to \infty$, we have that
$\clt^n_1\to 0$ in probability as $n\to \infty$.

Denote the first two terms on the right side of above display as $J^1_n$ and $J^2_n$ respectively.
Using the boundedness of $\Psi_i^n$, $\sigma$, $\nabla_x f$, the independence of the $W_i$, the fact that
$\Psi_i^n$ are $\clh_{t_0}$ measurable, and It\^{o}'s isometry, 
$E[(J^1_n)^2] \le c_3/n$
for some $c_3\in (0,\infty)$ and all $n \in \NN$. Thus $J^1_n \to 0$ in probability as $n\to \infty$. Similarly,
$J^2_n \to 0$ in probability as $n\to \infty$. Combining the above observations we have that $\Phi(Q^n, v^n)\to 0$ in probability, completing the proof of (c) and therefore of the lemma.
\end{proof}

Finally we complete the proof of the Laplace upper bound \eqref{eq:eq1137u} by proving \eqref{eq:varlb}.
By the definition of $Q^n$, 
\begin{align*}
	& E\left[ \frac{1}{2n} \sum_{i=1}^n \int_0^T \|u_i^n(t)\|^2\,dt + \frac{1}{2n\kappa(n)^2}\int_0^T \|v^n(t)\|^2\,dt + F(\bar{\mu}^n) \right] \\
			&= E\left[ \int_{\sR_1}\left(\frac{1}{2}\int_{\R^m\times [0,T]}\|y\|^2\,r(dy\,dt) \right)\, [Q^n]_2(dr) + 
			\frac{1}{2n\kappa(n)^2}\int_0^T \|v^n(t)\|^2\,dt + F([Q^n]_1) \right]. 
\end{align*}
Recall the uniform bound \eqref{eq:eq128}. Then from Lemmas \ref{tightness} and \ref{limit1},
$(Q^n, v^n)$ is tight and if $(Q,v)$ is a weak limit point then $Q \in \cle_1[v]$ a.s. Assume without loss of generality that
$(Q^n, v^n) \to (Q,v)$ along the full sequence.
 Then by Fatou's lemma and since $\sqrt{n}\kappa(n)\to \la$,
\begin{align*}
	&\liminf_{n\to\infty} E\left[ \frac{1}{2n} \sum_{i=1}^n \int_0^T \|u_i^n(t)\|^2\,dt + \frac{1}{2n \kappa(n)^2}\int_0^T \|v^n(t)\|^2\,dt + F(\bar{\mu}^n) \right] \\ 
		&\ge E\left[ \int_{\sR_1}\left(\frac{1}{2}\int_{\R^m\times [0,T]}\|y\|^2\,r(dy\,dt) \right)\, [Q]_2(dr) + \frac{1}{2\la^2}\int_0^T \|v(t)\|^2\,dt + F([Q]_1) \right] \\
		&\ge \inf_{\varphi \in L^2([0,T]:\R^k)} \inf_{\Theta\in \mathcal{E}_1[\varphi]} \left( E_{\Theta}\left[ \frac{1}{2} \int_{\R^m\times[0,T]} \|y\|^2\,\rho(dy\,dt) \right] + \frac{1}{2\la^2} \int_0^T \|\varphi(t)\|^2\,dt + F([\Theta]_1) \right), 
\end{align*}
where the last inequality uses the fact that $Q \in \cle_1[v]$ a.s.
This completes the proof of the Laplace upper bound.  
\hfill \qed

\subsection{Laplace Lower Bound} 
\label{sec:laplow2.1}

Throughout this section we assume that Conditions  \ref{empmzrinit}, \ref{Lip}, and \ref{sigmagamma} are satisfied. 
Fix $\veps>0$ and $F \in \clc_b(\clp(\clx))$. Choose a $\varphi \in L^2([0,T]:\R^k)$ and a $\Theta \in \mathcal{E}_1[\varphi]$ such that
\begin{align}\label{eq:chosphithe}
	\frac{1}{2}E_\Theta\left[ \int_{\R^m\times [0,T]}\|y\|^2\rho(dy\,dt) \right] + \frac{1}{2\la^2}\int_0^T \|\varphi(t)\|^2\,dt + F([\Theta]_1) \le \inf_{\nu \in \mathcal{P}(\mathcal{X})} \left[ F(\nu) + {I}_1(\nu) \right] + \veps. 
\end{align}
We will show that there is an $M\in (0,\infty)$ and a sequence $(u^n, v^n)$ with $u^n \in \mathcal{A}^{1,n}$ and $v^n \in \mathcal{A}^2_M$ constructed on some filtered probability space such that 
\begin{equation}\label{eq:eq1242}
\begin{aligned}
	&\limsup_{n \to \infty} E\left[ \frac{1}{2n} \sum_{i=1}^n \int_0^T \|u_i^n(t)\|^2\,dt + \frac{1}{2n \kappa(n)^2}\int_0^T \|v^n(t)\|^2\,dt + F(\bar{\mu}^n) \right] \\
	&\le \frac{1}{2}E_\Theta\left[ \int_{\R^m\times [0,T]}\|y\|^2\rho(dy\,dt) \right] + \frac{1}{2\la^2}\int_0^T \|\varphi(t)\|^2\,dt + F([\Theta]_1). 
\end{aligned}
\end{equation}
The Laplace lower bound \eqref{eq:eq1137l} is then immediate from Theorem  \ref{reptype1} on noting that $\veps>0$ is arbitrary.
The key ingredient in the proof of \eqref{eq:eq1242} is the following uniqueness result.
Define the map $\vartheta: \clz_1 \to \clz_1^{\vartheta}\doteq \R^d\times \clr_1\times \clw$ as 
$\vartheta(z,r,w) \doteq (z(0), r,w)$. For $\Theta \in \clp(\clz_1)$, let $\Theta_{\vartheta} \doteq \Theta \circ \vartheta^{-1}$
be the probability measure on $\clz_1^{\vartheta}$ induced by $\Theta$ under $\vartheta$.

We will say that {\em weak uniqueness} holds for \eqref{eq: controlled SDE} if, for any given $\varphi \in L^2([0,T]:\R^k)$ and $\Theta^{(1)}, \Theta^{(2)} \in \cle_1[\varphi]$, whenever
$\Theta^{(1)}_{\vartheta} = \Theta^{(2)}_{\vartheta}$, we have that $\Theta^{(1)} = \Theta^{(2)}$.

\begin{lemma}\label{pathunique} Weak uniqueness holds for \eqref{eq: controlled SDE}.
\end{lemma}

\begin{proof} Fix $\varphi \in L^2([0,T]:\R^k)$ and $\Theta^{(1)}, \Theta^{(2)} \in \cle_1[\varphi]$. Suppose that $\Theta^{(1)}_{\vartheta} = \Theta^{(2)}_{\vartheta}\doteq \Lambda$.
	 Note that $\Theta^{(i)}$, $i=1,2$ can be disintegrated as
	$$\Theta^{(i)}(dx,\, dr,\, dw) = \tilde \Theta^{(i)}(r,w,x_0, dx)\Lambda(dx_0,\, dr,\, dw).$$
	Consider $\hat \clz_1 = \clx \times \clx \times \R^d \times \clr_1 \times \clw$.
	Define $\hat \Theta \in \clp(\hat \clz_1)$ as
	$$\hat \Theta (dx^{(1)},\, dx^{(2)},\, dx_0,\, dr,\, dw) \doteq \tilde \Theta^{(1)}(r,\,w,\, x_0,\, dx^{(1)}) \tilde \Theta^{(2)}(r,\,w,\, x_0,\, dx^{(2)}) \Lambda(dx_0,\, dr,\, dw)$$
and denote the coordinate maps on $\hat \clz_1$ as $(X^{(1)}, X^{(2)}, X_0, \rho, W)$. Note that the process $W$ is a Brownian motion with respect to the canonical filtration
$$\hat \clh_t \doteq \sigma\left\{X^{(1)}(s), X^{(2)}(s), \rho(A\times [0,s]), W(s),\; A \in \clb(\R^m), s \in [0,t]\right\}, \; t \in [0,T],$$
and for $i=1,2$, $X^{(i)}$  satisfy \eqref{eq: controlled SDE} 	with $\bar X$ replaced with $X^{(i)}$ and
$\nu(t)$ replaced with $\nu^{(i)}(t) \doteq \nu_{\Theta^{(i)}}(t)$. Also, $X^{(i)}(0)=X_0$ for $i=1,2$.
In order to prove the lemma it suffices to show that $X^{(1)} = X^{(2)}$ a.s.
Let $u(t) \doteq \int_{\RR^m} y \rho_t(dy)$, $t \in [0,T]$. Then $E_{\hat \Theta} \int_0^T \|u(t)\|^2 dt <\infty$.
By the Lipschitz properties of $b$, $\alpha$, and $\sigma$, the property that $\sigma(x,v)\equiv \sigma(v)$, and since $\varphi \in L^2([0,T]:\R^k)$,   we have that, for some  $c_1 \in (0,\infty)$ and for any $t\in [0,T]$, 
\begin{equation}\label{eq:simplethird}
\begin{aligned}
	E_{\hat \Theta}\left[ \sup_{0\le s \le t} \|X^{(1)}(s) - X^{(2)}(s)\|^2 \right] &\le  
			c_1  \int_0^{t} \left( E_{\hat \Theta}\|X^{(1)}(s) - X^{(2)}(s)\|^2 + d_{BL}(\nu^{(1)}(s), \nu^{(2)}(s))^2 \right) \,ds  \\
			&\quad + c_1E_{\hat \Theta}\left(\int_0^{T} d_{BL}(\nu^{(1)}(s), \nu^{(2)}(s)) \cdot \|u(s)\|\,ds \right)^2. 
\end{aligned}
\end{equation}
Since
$$
E_{\hat \Theta}\left(\int_0^{t} d_{BL}(\nu^{(1)}(s), \nu^{(2)}(s)) \cdot \|u(s)\|\,ds \right)^2 \le \int_0^{t} d_{BL}(\nu^{(1)}(s), \nu^{(2)}(s))^2 \,ds \cdot 
E_{\hat \Theta}\left[\int_0^{T}  \|u(s)\|^2\,ds\right]$$
and $E_{\hat \Theta}\int_0^{T}  \|u(s)\|^2\,ds<\infty$, we have, for all $t\in [0,T]$,
\begin{align*}
	E_{\hat \Theta}\left[\sup_{0\le s \le t} \|X^{(1)}(s) - X^{(2)}(s)\|^2\right] &\le  
			c_2  \int_0^{t} \left( E_{\hat \Theta}\|X^{(1)}(s) - X^{(2)}(s)\|^2 + d_{BL}(\nu^{(1)}(s), \nu^{(2)}(s))^2 \right) \,ds .
\end{align*}
Furthermore, for each $t$,  
\begin{align*}
	d_{BL}(\nu^{(1)}(t), \nu^{(2)}(t)) 
						       &= \sup_{f\in BL(\R^d)} \left| \int_{\hat{\mathcal{Z}_1}} f(X^{(1)}(t))\,d\hat \Theta - \int_{\hat{\mathcal{Z}_1}} f(X^{(2)}(t))\,d\hat \Theta \right| \le 
						       E_{\hat \Theta}\| X^{(1)}(t) - X^{(2)}(t) \|.
\end{align*}
Thus, for some $c_3 \in (0,\infty)$, we have, for all $t\in [0,T]$,
\[
	E_{\hat \Theta}\left[ \sup_{0\le s\le t}\|X^{(1)}(s) - X^{(2)}(s)\|^2\right] \le c_3 \int_0^t E_{\hat \Theta}\left[\sup_{0\le s \le  \tau}\|X^{(1)}(s) - X^{(2)}(s)\|^2\right]\,d\tau.
\]
By Gronwall's inequality, this shows that  $X^{(1)}$ and $X^{(2)}$ are indistinguishable on $[0,T]$ and completes the proof of the lemma.
\end{proof}

Now we return to the construction of $(u^n, v^n)$ that satisfy \eqref{eq:eq1242}, where recall that $\Theta$ and $\varphi$ are chosen to satisfy \eqref{eq:chosphithe}.
Let $(\bar X, \rho, W)$ be the coordinate maps on the space $(\mathcal{Z}_1, \sB(\mathcal{Z}_1), \Theta)$ equipped with the canonical filtration 
$\clh_t$, defined in \eqref{eq:eq428}, namely 
\[
	\mathcal{H}_t = \sigma\left\{X(s), \rho(A \times [0,s]), W(s) :  A \in \mathcal{B}(\R^m), s\le t\right\}.
\]
Since $\Theta \in \sE_1[\varphi]$, equation \eqref{eq: controlled SDE} is satisfied with $\nu(t) = \nu_{\Theta}(t)$ and $\nu_{\Theta}(0) = \xi_0$.

Disintegrate $\Theta_{\vartheta}$  as 
\[
	\Theta_{\vartheta}(dx, dr, dw) = \xi_0(dx)\,[\Theta]_3(dw)\,\hat{\Lambda}_0(x,  w, dr).
\]
Let $\mathcal{V} \doteq \sC([0,T]:\R^k)$ and define 
\[
	\Omega' \doteq (\mathcal{R}_1\times \mathcal{W})^{\infty} \times \mathcal{V}, \quad \sF' \doteq \sB(\Omega').
\]
Elements of $\Omega'$ are of the form $(r, w, \beta)$, where $\beta \in \clv$, $r = (r_1, r_2, \ldots)$, $w = (w_1, w_2, \ldots)$,  $r_i \in \mathcal{R}_1$ and $w_i \in \mathcal{W}$ for each $i \in \N$. On the measurable space $(\Omega', \sF')$ define the random variables 
\[
	W_i(t, (r, w, \beta)) \doteq w_i(t), \qquad B(t, (r, w, \beta)) \doteq \beta(t), \qquad \rho_i(r, w, \beta) \doteq r_i,
\]
for each $t \in [0,T]$ and $i\in \NN$. Let $\Gamma$ be the standard Wiener measure on $\mathcal{V}$.
Recall the initial values $\{x^n_i\}$ introduced in Section \ref{sec:intempdis}.
For each $n\in \N$, define the probability measure $P^n$ on $(\Omega', \sF')$ by 
\[
	dP^n(r, w, \beta) = \left[\bigotimes_{i=1}^n \,[\Theta]_3(dw_i)\,\hat{\Lambda}_0(x_i^n, w_i, dr_i)\, \bigotimes_{i=n+1}^\infty \,[\Theta]_{(2,3)}(dr_i, dw_i)\right]\otimes\,\Gamma(d\beta). 
\]
Under $P^n$,  $\{W_i\}_{1\le i\le n}$ and $B$ are mutually independent Brownian motions. Define the sequence $\{\Lambda^n\}_{n\in \N}$ of $\mathcal{P}(\R^d\times \mathcal{R}_1 \times \mathcal{W})$-valued random variables on $(\Omega', \sF')$ by
\[
	\Lambda^n(A\times R\times C) = \frac{1}{n} \sum_{i=1}^n \delta_{x_i^n}(A)\delta_{\rho_i}(R)\delta_{W_i}(C), \; A\times R\times C \in \sB(\R^d\times \sR_1 \times \mathcal{W}).
\]
Then by Condition \ref{empmzrinit}, 
\begin{equation}\label{weakPns}
	P^n \circ (\Lambda^n)^{-1} \to \delta_{\Theta_{\vartheta}}. 
\end{equation}
Let, for $n \in \NN$, $v^n \doteq \varphi$. Denoting $\int_0^T \|\varphi(s)\|^2 ds \doteq M$, we have that $v^n \in S_M$ for every $n$.
Next, for each $i\in \N$, let 
\begin{equation}\label{eq:equtorho}
	u_i(t) \doteq \int_{\R^m} y\,(\rho_i)_t(dy), \; t\in [0,T],
\end{equation}
 where $(\rho_i)_t(dy)\,dt = \rho_i(dy\,dt)$, and for each $n\in \N$, let $(\bar{X}^n_1, \ldots, \bar{X}_n^n)$ be the solution on $(\Omega', \sF', P^n)$ of the system \eqref{controlledn},
where $\bar{\mu}^n(t) = \frac{1}{n} \sum_{i=1}^n \delta_{\bar{X}^n_i(t)}$  for each $t\in [0,T]$. Unique solvability of the above equation is a consequence of our assumptions on the coefficients, namely Condition \ref{Lip}.

For each $n$, define the occupation measure $Q^n$ by \eqref{eq:bigempmzr}, replacing $\rho^n_i$ with $\rho_i$. That is, 
\[
	{Q}^n(B\times R\times D) \doteq \frac{1}{n}\sum_{i=1}^n \delta_{\bar{X}_i^n}(B)\delta_{\rho_i}(R)\delta_{W_i}(D),\; B\times R\times D \in \sB(\mathcal{Z}_1).
\]
 Let $E^n$ denote expectation over the probability measure $P^n$. Then
\begin{align}
	&\limsup_{n\to\infty} E^n \left[ \frac{1}{n} \sum_{i=1}^n \int_0^T \|u_i(t)\|^2\,dt \right] \nonumber\\
				&= \limsup_{n\to\infty} \frac{1}{n}\sum_{i=1}^n \int_{\sR_1\times \mathcal{W}} \int_0^T \left\| \int_{\R^m} y\,r_t(dy) \right\|^2 \,dt\,\hat{\Lambda}(x_i^n, w, dr)\,[\Theta]_3(dw) \nonumber\\
				&= E_\Theta\left[ \int_0^T  \left\| \int_{\R^m} y\,\rho_t(dy) \right\|^2 \,dt \right] \le E_\Theta\left[ \int_{\R^m\times [0,T]} \|y\|^2\,\rho(dy\,dt) \right] < \infty, \label{eq:eq645}
\end{align}
where the second equality is from Condition \ref{empmzrinit}.
It follows from Lemma \ref{tightness} that $\{({Q}^n, v^n)\}_{n\in \N}$ is tight.
 If $({Q}, {v})$ is a limit point of this sequence defined on some probability space $(\tilde{\Omega}, \tilde{\sF}, \tilde{P})$, then ${v} = \varphi$ $\tilde{P}$-a.s., and, by Lemma \ref{limit1}, ${Q} \in \sE_1[{v}] = \sE_1[\varphi]$ $\tilde{P}$-a.s.
Recall that $\Theta \in  \sE_1[\varphi]$ as well.
By \eqref{weakPns}, for $\tilde{P}$-a.e. $\omega \in \tilde{\Omega}$, $Q_{\vartheta}(\omega) = \Theta_{\vartheta}$.
 Thus
by the  weak uniqueness established in Lemma \ref{pathunique}, ${Q} = \Theta$ $\tilde{P}$-a.s., and so $Q^n\to \Theta$ in probability. Finally, 
\begin{equation}\label{eq:eq1078}
\begin{aligned}
	&\limsup_{n\to\infty} E^n\left[ \frac{1}{2n} \sum_{i=1}^n \int_0^T \|u_i(t)\|^2\,dt + \frac{1}{2n\kappa(n)^2}\int_0^T\|v^n(t)\|^2\,dt + F(\bar{\mu}^n) \right] \\
				&= \limsup_{n\to\infty} E^n\left[ \frac{1}{2n} \sum_{i=1}^n \int_0^T \|u_i(t)\|^2\,dt + \frac{1}{2n\kappa(n)^2}\int_0^T\|v^n(t)\|^2\,dt + F([{Q}^n]_1) \right] \\
				&\le \frac{1}{2}E_\Theta\left[ \int_{\R^m\times [0,T]}\|y\|^2\rho(dy\,dt) \right]
				 + \frac{1}{2\la^2} \int_0^T \|\varphi(t)\|^2\,dt + F([\Theta]_1),
\end{aligned}
\end{equation}
which follows from \eqref{eq:eq645}, the equality $v^n = \varphi$, the weak convergence $Q^n \to \Theta$, and the assumption that $\sqrt{n}\kappa(n) \to \lambda$. 
This proves \eqref{eq:eq1242} and completes the proof of the lower bound. \hfill \qed
\subsection{Rate Function Property}
\label{sec:ratefun2.1}
In this section we sketch the proof of the fact that $I_1$ defined in \eqref{eq:eq1135} is a rate function.
The proof is very similar to the Laplace upper bound and so some details are left to the reader.
We will assume   Conditions \ref{empmzrinit} and \ref{Lip} are satisfied.
Like with the proof of the upper bound, Condition \ref{sigmagamma} is not needed.

Fix $L \in (0,\infty)$, let $\Gamma_L \doteq \{ \nu \in \clp(\clx): I_1(\nu) \le L\}$, and let $\{\nu_n\}$ be a sequence in $\Gamma_L$. We need to show that the sequence has a limit point that lies in $\Gamma_L$.
From the definition of $I_1$, we can find, for each $n$, a $\varphi^n \in L^2([0,T]: \R^k)$ and a $\Theta^n \in \cle_1[\varphi^n]$ with $[\Theta^n]_1 = \nu^n$ such that
 \begin{equation}\label{eq:unifbdlevset}
	 E_{\Theta^n}\left[ \frac{1}{2} \int_{\R^m\times[0,T]} \|y\|^2\,\rho(dy\,dt) \right] + \frac{1}{2\la^2} \int_0^T \|\varphi^n(t)\|^2\,dt  \le L + \frac 1 n.
	 \end{equation}
 In particular, $\{\varphi^n\} \subset S_M$ where $M = 2(L+1)\la^2$.
 An argument similar to the proof of Lemma \ref{tightness} shows that the sequence $(\Theta^n, \varphi^n)$ is relatively compact in $\clp(\clz_1)\times S_M$.
 Suppose that $(\Theta^n, \varphi^n) \to (\Theta, \varphi)$ along some subsequence. Then (along the subsequence) $\nu^n \to \nu \doteq [\Theta]_1$.
 Sending $n\to \infty$ and using lower semicontinuity,
 $$E_{\Theta}\left[ \frac{1}{2} \int_{\R^m\times[0,T]} \|y\|^2\,\rho(dy\,dt) \right] + \frac{1}{2\la^2} \int_0^T \|\varphi(t)\|^2\,dt  \le L.$$
 Furthermore, since $\Theta^n \in \cle_1[\varphi^n]$,
  $\Phi(\Theta^n, \varphi^n)=0$ for each $n$, where $\Phi$ is as in \eqref{Phidef}.
  As shown in Lemma \ref{limit1}, for each $B<\infty$, $\Phi_B(\Theta^n, \varphi^n)\to \Phi_B(\Theta, \varphi)$. Also a similar argument as in \eqref{eq:4lt} shows that, as $B\to \infty$,
  $$\sup_{n \in \NN} |\Phi_B(\Theta^n, \varphi^n) - \Phi(\Theta^n, \varphi^n)| \to 0, \; |\Phi_B(\Theta, \varphi) - \Phi(\Theta, \varphi)|\to 0.$$
  It then follows that $\Phi(\Theta, \varphi)=0$, proving that 
  $\Theta \in \cle_1[\varphi]$.
  Thus, since $\nu = [\Theta]_1$,
  $$I_1(\nu) \le E_{\Theta}\left[ \frac{1}{2} \int_{\R^m\times[0,T]} \|y\|^2\,\rho(dy\,dt) \right] + \frac{1}{2\la^2} \int_0^T \|\varphi(t)\|^2\,dt  \le L.$$
  The result follows. \hfill \qed
  
\section{Proof of Theorem \ref{resulttype2}. }
\label{sec:proofthm2}
In this section we prove Theorem \ref{resulttype2}. 
Proof of part 1 follows by standard arguments and is therefore left to the Appendix. 
Proof of part 2 follows similar steps as that for Theorem \ref{resulttype1}. Namely, we prove the Laplace upper and lower bounds and show that the function $I_2$ introduced in \eqref{eq:ratefuni2} is a rate function.
The upper bound is established in Section \ref{sec:lapuppbdwted} while the lower bound is given in
Section \ref{sec:laplowbd2}. The rate function property is verified in Section \ref{sec:ratefunctype2}.

For $(u^n, v^n) \in \mathcal{A}^{1,n}\times \mathcal{A}^2_M$, we consider the following system of  controlled SDEs:
\begin{equation}\label{controlledn2}
\begin{aligned}
	d\bar{X}_i^n(t) &= b(\bar{X}_i^n(t), \bar{\mu}^n(t))\,dt +  \sigma(\bar{X}_i^n(t), \bar{\mu}^n(t))u^n_i(t)\,dt + \alpha(\bar{X}_i^n(t), \bar{\mu}^n(t))v^n(t)\,dt \\ 
				&\quad  +\sigma(\bar{X}_i^n(t), \bar{\mu}^n(t))\,dW_i(t) + \kappa(n) \alpha(\bar{X}_i^n(t), \bar{\mu}^n(t))\,dB(t), \\ 
	d\bar{A}_i^n(t) &= \bar{A}_i^n(t)c(\bar{X}_i^n(t), \bar{\mu}^n(t))\,dt + \bar{A}_i^n(t)\gamma^T(\bar{X}_i^n(t), \bar{\mu}^n(t))u_i^n(t)\,dt + \bar{A}_i^n(t)\beta^T(\bar{X}_i^n(t), \bar{\mu}^n(t))v^n(t)\,dt \\ 
				&\quad + \bar{A}_i^n(t)\gamma^T(\bar{X}_i^n(t), \bar{\mu}^n(t))\,dW_i(t) + \kappa(n) \bar{A}_i^n(t)\beta^T(\bar{X}_i^n(t), \bar{\mu}^n(t))\,dB(t), \\ 
	\bar{X}_i^n(0) &= x_i^n, \quad \bar{A}_i^n(0) = a_i^n,\quad 1\le i \le n,
\end{aligned}
\end{equation}
where $\bar{\mu}^n(t)$ is the weighted empirical measure 
\begin{equation}\label{eq:eq124}
	\bar{\mu}^n(t) = \frac{1}{n}\sum_{i=1}^n \theta(\bar{A}_i^n(t))\delta_{\bar{X}_i^n(t)}. 
\end{equation}
The existence and uniqueness of strong solutions of the above system of equations is argued in the same way as for the uncontrolled system in \eqref{eq:interempwtd} (see Appendix \ref{sec:llnwted}).

The following representation follows along the lines of Theorem \ref{reptype1}. 
Let $\clk \doteq \sC([0,T]:\mathcal{M}_+(\R^d))$.

\begin{theorem}\label{reptype2} For any $F \in \sC_b(\clk)$ and for each $n\in \N$, 
\begin{equation}\label{eq:seq2}
	-\frac{1}{n}\log E\left[e^{-n  F(\mu^n)}\right] =\inf_{(u, v) \in \mathcal{A}^{1,n}\times \mathcal{A}^2} E\left[ \frac{1}{2n} \sum_{i=1}^n \int_0^T \|u_i(t)\|^2\,dt + \frac{1}{2}\int_0^T \|v(t)\|^2\,dt + F(\bar{\mu}^n) \right].
\end{equation}
Furthermore, for every $\delta > 0$, there is an $M < \infty$ such that for each $n\in \N$, 
\begin{equation}\label{eq:seq2b}
	-\frac{1}{n}\log E\left[e^{-n  F(\mu^n)}\right] \ge \inf_{(u, v) \in \mathcal{A}^{1,n}\times \mathcal{A}_M^2} E\left[ \frac{1}{2n} \sum_{i=1}^n \int_0^T \|u_i(t)\|^2\,dt + \frac{1}{2}\int_0^T \|v(t)\|^2\,dt + F(\bar{\mu}^n) \right] - \delta. 
\end{equation}
\end{theorem}

\subsection{Laplace Upper Bound}
\label{sec:lapuppbdwted}
In this section we show that for every $F \in \clc_b(\clk)$
\begin{equation}\label{eq:eq1137u2}
	\liminf_{n\to\infty} -\frac{1}{n} \log E\left[e^{-n F(\mu^n)}\right] \ge \inf_{\nu \in \clk} \left[ F(\nu) + {I}_2(\nu) \right],
\end{equation}
where $I_2$ is as in \eqref{eq:ratefuni2}.
Throughout the section we assume that Conditions  \ref{empmzrinit}, \ref{Lip}, \ref{LipA}, \ref{theta}, and \ref{wtdmzrinit} are satisfied. We will not make use of Conditions \ref{sigmagamma} and \ref{gammaconst} for the upper bound proof.

Fix $F \in \sC_b(\clk)$ and $\delta\in (0,1)$. From Theorem \ref{reptype2}, there is an $M < \infty$ and, for each $n\in \N$,  $(u^n, v^n)\in \mathcal{A}^{1,n} \times \mathcal{A}^2_M$ such that 
\begin{equation}\label{eq:seq2c}
	-\frac{1}{n}\log E\left[e^{-n  F(\mu^n)}\right] \ge E\left[ \frac{1}{2n} \sum_{i=1}^n \int_0^T \|u_i^n(t)\|^2\,dt + \frac{1}{2n \kappa(n)^2}\int_0^T \|v^n(t)\|^2\,dt + F(\bar{\mu}^n) \right] - \delta.
\end{equation}
We will next show that 
\begin{equation}\label{eq:varlb2}
\begin{aligned}
	&\liminf_{n\to\infty} E\left[ \frac{1}{2n} \sum_{i=1}^n \int_0^T \|u_i^n(t)\|^2\,dt + \frac{1}{2n \kappa(n)^2}\int_0^T \|v^n(t)\|^2\,dt + F(\bar{\mu}^n) \right] \\ 
	&\ge  \inf_{\varphi \in L^2([0,T]:\R^k)} \inf_{\Theta\in \mathcal{E}_2[\varphi]} \left( E_\Theta\left[ \frac{1}{2} \int_{\R^m\times[0,T]} \|y\|^2\,\rho(dy\,dt) \right] + \frac{1}{2\la^2} \int_0^T \|\varphi(t)\|^2\,dt + F(\nu_\Theta) \right), 
\end{aligned}
\end{equation}
where $\nu_\Theta$ is as in \eqref{eq:eq503}.
Since $\delta\in(0,1)$ is arbitrary, the desired bound in \eqref{eq:eq1137u2} is immediate from the above inequality on recalling the definition of $I_2$ in \eqref{eq:ratefuni2}.
In the rest of this section we prove \eqref{eq:varlb2}.  

We begin by observing that from \eqref{eq:seq2c} we have, as in Section \ref{sec:lapupp2.1}, that \eqref{eq:eq128} is satisfied.
The next two lemmas are  analogues of Lemmas \ref{tightness} and \ref{limit1}. In Lemma \ref{tightness2}  below, the result under Condition (ii) in \eqref{eq:1202} will be used for the proof of the LLN sketched in the Appendix.
\begin{lemma}\label{tightness2} 
		Suppose for some $M \in (0,\infty)$, $\{(u^n, v^n)\}_{n\in \N}$ is a sequence with $(u^n, v^n) \in \sA^{1,n}\times \sA^2_M$ for each $n$, and suppose $\{u^n\}_{n\in \N}$ satisfies, for some $L\in (0,\infty)$,
		\begin{equation}\label{eq:eq135b}
			\sup_{n \in \NN} E\left[ \frac{1}{n} \sum_{i=1}^n \int_0^T \|u_i^n(t)\|^2\,dt \right] \le L.
		\end{equation}
		 Define the $\mathcal{P}(\mathcal{Z}_2)$-valued random variable $Q^n$ as
		 \begin{equation}\label{eq:eq906}
		 	Q^n(A\times R\times C) = \frac{1}{n}\sum_{i=1}^n \delta_{(\bar{X}_i^n, \bar{A}_i^n)}(A)\delta_{\rho_i^n}(R)\delta_{W_i}(C), \;  A\times R\times C \in \sB(\mathcal{Z}_2),
		 \end{equation}
		where $\rho_i^n$ is as in \eqref{eq:eq1027}.
		Suppose that
		\begin{equation}\label{eq:1202}
			\mbox{ either  (i) } \gamma \equiv 0, \mbox{ or  (ii) } u^n_i\equiv 0 \mbox{ for all } i,n, \mbox{ or (iii) Condition \ref{theta}(a) holds. } 
		\end{equation}
		Then $\{(Q^n, v^n)\}_{n\in \N}$ is tight as a sequence of $\mathcal{P}(\mathcal{Z}_2) \times S_M$-valued random variables.
\end{lemma}

\begin{proof} 
	Tightness of $\{v^n\}$ is immediate from the compactness of $S_M$. The tightness of $[Q^n]_3$ and $[Q^n]_4$ follows  as in the proof of Lemma \ref{tightness}.
	Finally we show the tightness of $[Q^n]_{1,2}$.
	If (i) or (ii) in \eqref{eq:1202} hold, this tightness
	 follows  as the proof of the tightness of $[Q^n]_1$ in Lemma \ref{tightness} on recalling Condition \ref{wtdmzrinit}, the linear growth property of $\theta$, and
	using the following estimate instead of \eqref{eq:eq1048}:
	\begin{equation}\label{eq:eq1048b}
		E\left[ \sup_{0\le s\le T} \left(\|\bar{X}_i^n(s)\|^2 +  (\bar{A}_i^n(s))^2\right)\right] \le c_1\left( 1 + \|x_i^n\|^2 + (a_i^n)^2+ E\left[ \int_0^T \|u_i^n(s)\|^2\,ds\right] \right). 
	\end{equation}
	For case (iii) in \eqref{eq:1202}, we cannot ensure the above square integrability property. However, one can proceed as follows.
	By It\^{o}'s formula, 
	\begin{align*}
		\theta(\bar A_i^n(t)) &= \theta(a_i^n) 		 + \int_0^t \theta'(\bar A_i^n(s))\bar A_i^n(s)\,dD_i^n(s) \\
					& \quad + \frac{1}{2}\int_0^t \theta''(\bar A_i^n(s))\bar A_i^n(s)^2\left(\|\gamma(\bar X_i^n(s), \bar \mu^n(s)) \|^2 + \kappa(n)^2 \|\beta(\bar X_i^n(s), \bar\mu^n(s)\|^2 \right)\,ds, 
	\end{align*}
	where 
	\begin{align*}
		D^n_i(t) &= \int_0^t c(\bar X_i^n(s), \bar \mu^n(s))\,ds + \int_0^t \gamma^T(\bar X_i^n(s), \bar \mu^n(s))\,dW_i(s) + \kappa(n) \int_0^t \beta^T(\bar X_i^n(s), \bar\mu^n(s))\,dB(s)\\
		&\quad + \int_0^t \gamma^T(\bar X_i^n(s), \bar \mu^n(s)) u^n_i(s) ds +  \int_0^t \beta^T(\bar X_i^n(s), \bar\mu^n(s)) v^n(s) ds.
	\end{align*}
	By the boundedness of the coefficients and using \eqref{eq:thetderbds}, i.e.
	$\sup_x |\theta'(x)x| + \sup_x |\theta''(x)x^2| < \infty$, we then have, for some $c_2 \in (0,\infty)$,
	\begin{equation}
		E\left[ \sup_{0\le s\le T} \left(\|\bar{X}_i^n(s)\|^2 +  (\theta(\bar{A}_i^n(s)))^2\right)\right] \le c_2\left( 1 + \|x_i^n\|^2 + (a_i^n)^2+ E\left[ \int_0^T \|u_i^n(s)\|^2\,ds\right] \right). 
		\label{eq:eq217}
	\end{equation}
	Using the above integrability, the tightness of $[Q^n]_1$ follows as in the proof of Lemma \ref{tightness}.
	In order to show the tightness of $[Q^n]_2$  we will use the fact that the map $\phi(\cdot) \mapsto e^{\phi(\cdot)}$ is a continuous map from
	$\clc([0,T]: \R)$ to $\clc([0,T]:\R_+)$. With this fact, it suffices to show that the collection $\{\frac{1}{n}\sum_{i=1}^n \delta_{\log \bar A^n_i(\cdot)}, n \in \NN\}$ is tight as a sequence of  $\clp(\clc([0,T]: \R))$-valued random variables. This tightness follows, once again as in the proof of Lemma \ref{tightness}, from Condition \ref{wtdmzrinit} and  the estimates
	\begin{equation}\label{eq:eq1048c}
		E\left[ \sup_{0\le s\le T} \left| \log(\bar A_i^n(s))\right| \right] \le c_3\left( 1 + |\log a_i^n| + E\left[ \int_0^T \|u_i^n(s)\|^2\,ds\right] \right)
	\end{equation}
	and
	\[
		E \left[ \left| \log\bar{A}_i^n(\tau+\eps) - \log \bar{A}_i^n(\tau) \right|^2 \right] \le c_3 \eps \left( 1  + E\left[ \int_0^T \|u_i^n(s)\|^2\,ds\right] \right), 
	\]
	where  $\tau$ is a stopping time taking values in $[0, T-\eps]$, and the constant $c_3$  does not depend on $n$, $i$, $\eps$, or the stopping time $\tau$.
	\end{proof}

\begin{lemma}\label{limit} Let  $\{(u^n, v^n)\}_{n\in \N}$ be as in Lemma \ref{tightness2}.  Suppose that one of the conditions in \eqref{eq:1202} is satisfied.
	Also suppose that $(Q^n, v^n)$ converges, in distribution, along a subsequence to 
	a $\clp(\clz_2) \times S_M$-valued random variable
	$(Q,v)$. Then $Q\in \sE_2[v]$ a.s.
\end{lemma}

\begin{proof} 
	Suppose that $(Q, v)$ is given on the probability space $(\Omega^*, \mathcal{F}^*, P^*)$. 
	In a similar manner as in the proof of Lemma \ref{limit1} (in particular using \eqref{eq:eq1048b} and \eqref{eq:eq217}) we see that 
	 $Q \in \mathcal{P}_2(\mathcal{Z}_2)$ $P^*$-a.s.
We need to show that $Q(\omega)$ is a weak solution to $\sS_2[v(\omega), \nu_{Q}(\omega)]$ for $P^*$-a.e. $\omega \in \Omega^*$.
Note that
$[Q^n]_{1,2} \circ (z(0), \vs(0))^{-1} \to \eta_0$ weakly, which shows that, for $P^*$-a.e. $\om$, under $Q(\omega)$, $(z(0), \vs(0))$ has distribution $\eta_0$,
where
$(z,\vs, r,w)$
denote the canonical coordinate variables on $\clz_2$.

Thus to prove the result
 it suffices to show that for every $f\in \sC_c^2(\R^m\times \R_+ \times \R^d)$, for a.e. $\omega$, $M^{Q(\omega)}_{f, v(\omega)}$ is a martingale under $Q(\omega)$ with respect to the canonical filtration $\tilde \clh_t \doteq \sigma \{ z(s), \vs(s), w(s), r (A \times [0,s]), \, A \in \clb(\RR^m),\, s \le t\}$, $t\in [0,T]$,
where for each $\varphi \in L^2([0,T]:\R^k)$ and $\Theta \in \mathcal{P}_2(\mathcal{Z}_2)$, the process $\{M_{f,\varphi}^\Theta(t), 0\le t\le T\}$ is defined on $(\mathcal{Z}_2, \sB(\mathcal{Z}_2), \Theta)$ by 
\begin{equation}\label{eq:mfphithet}
\begin{aligned}
	M_{f, \varphi}^\Theta(t, (\zz, \vs, r, w)) &= f(\zz(t), \vs(t), w(t)) - f(\zz(0), \vs(0), 0) - \int_0^t \int_{\R^{m}} \mathcal{L}^\Theta_s(f)(\zz(s), \vs(s), y, w(s))\,r_s(dy)\,ds \\
					&\quad - \int_0^t \left\il \alpha(\zz(s), \nu_{\Theta}(s))\varphi(s), \nabla_x f(\zz(s), \vs(s), w(s)) \right \ir\,ds \\
					&\quad - \int_0^t \vs(s)\beta^T(\zz(s), \nu_{\Theta}(s))\varphi(s) \frac{\partial f}{\partial a}(\zz(s), \vs(s), w(s))\,ds,  
\end{aligned}
\end{equation}
and where
\begin{equation}\label{eq:lthetasf} 
\begin{aligned}
	&\mathcal{L}_s^\Theta(f)(x, a, y, w) \\
	&\quad = \left \il b(x, \nu_\Theta(s)) + \sigma(x,\nu_\Theta(s))y, \nabla_x f(x, a, w) \right\ir + \left( a c(x, \nu_\Theta(s)) + a \gamma^T(x,\nu_\Theta(s))y\right)\frac{\partial f}{\partial a}(x, a, w)  \\
			&\qquad + \frac{1}{2} \sum_{j,j'=1}^d (\sigma\sigma^T)_{jj'}(\nu_{\Theta}(s))\frac{\partial^2f}{\partial x_j \partial x_{j'}}(x, a, w) + \frac{1}{2}a^2\|\gamma(x,\nu_{\Theta}(s))\|^2 \frac{\partial^2 f}{\partial a^2}(x, a, w) \\
			&\qquad + \frac{1}{2} \sum_{j=1}^d a(\sigma\gamma)_{j}(x,\nu_\Theta(s))\frac{\partial^2f}{\partial x_j \partial a}(x, a, w) + \frac{1}{2}\sum_{j=1}^m \frac{\partial^2 f}{\partial w_j^2}(x, a, w) \\
			&\qquad  + \sum_{j=1}^d \sum_{j'=1}^m \sigma_{jj'}(x,\nu_\Theta(s))\frac{\partial^2 f}{\partial x_j \partial w_{j'}} (x, a, w) + \sum_{j=1}^m a\gamma_j(x,\nu_{\Theta}(s)) \frac{\partial^2 f}{\partial a \partial w_j}(x, a, w), 
\end{aligned}
\end{equation}
for $(x, a, y, w) \in \R^d\times \R_+ \times \R^m\times \R^m$. 

In order to prove the martingale property, as previously, it suffices to show that for
any   time instants $0\le t_0<t_1\le T$,  and any  $\Psi \in \sC_b(\mathcal{Z}_2)$ that is measurable with respect to the sigma field $\tilde \sH_{t_0}$, we have, 
\begin{equation}
	\label{eq:eq823b}
	E_{Q(\omega)}\left[ \Psi \left( M_{f, v(\omega)}^{Q(\omega)}(t_1) - M_{f, v(\omega)}^{Q(\omega)}(t_0) \right) \right] = 0, \;
	\mbox{ for }
	P^* \mbox{-a.e. } \om \in \Om^*.
\end{equation}
We suppress $\omega$ in the notation of the remaining proof.
Fix a choice of $(t_0, t_1, \Psi, f)$ and define 
$\Phi  : \mathcal{P}_2(\mathcal{Z}_2)\times S_M \to \R$ 
by 
\begin{equation}\label{eq:phithetvar}
	\Phi(\Theta, \varphi) = E_{\Theta}\left[ \Psi \left( M_{f,\varphi}^{\Theta}(t_1) - M_{f, \varphi}^{\Theta}(t_0) \right) \right]. 
\end{equation}
Fix $B\in (0,\infty)$. For $\Theta \in \clp_2(\clz_2)$, define $\mathcal{L}_s^{\Theta,B}$ by replacing $y$ on the right side of \eqref{eq:lthetasf} by $\zeta_B(y)$ and 
$\nu_{\Theta}$ by $\nu_{\Theta}^B$, where $\zeta_B$ is as in the proof of Lemma \ref{limit1} and $\nu_{\Theta}^B \in \mathcal{K}$ is defined as
\begin{equation}\label{eq:eqnuthetab}
\il f, \nu_{\Theta}^B(t)\ir \doteq E_{\Theta}\left[ (\theta(\vs(t))\wedge B) f(z(t))\right], \; t \in [0,T],\; f \in \clc_b(\R^d).\end{equation}
Similarly define $M_{f,\varphi}^{\Theta,B}$ by replacing $\mathcal{L}_s^{\Theta}$ with $\mathcal{L}_s^{\Theta,B}$ and $\nu_{\Theta}$ with $\nu_{\Theta}^B$ in \eqref{eq:mfphithet}.
Finally, define $\Phi_B$ by replacing $M_{f,\varphi}^{\Theta}$ with $M_{f,\varphi}^{\Theta,B}$ on the right side of \eqref{eq:phithetvar}.
Then, as before, we will argue (a) for every $B \in (0,\infty)$, $\Phi_B$ is a bounded and continuous map on $\clp_2(\clz_2)\times S_M,$  (b)
$\sup_n E[|\Phi_B(Q^n, v^n) - \Phi(Q^n, v^n)|\wedge 1] \to 0$ and $ E^*[|\Phi_B(Q, v) - \Phi(Q, v)|\wedge 1] \to 0$ as $B\to \infty$,  
(c) $\Phi(Q^n, v^n) \to 0$ in probability as $n \to \infty$. The statement in \eqref{eq:eq823b} is immediate from (a)-(c).

Part (a) is shown exactly as in the proof of Lemma \ref{limit1}. 
Next consider (b). Using the Lipschitz property of the coefficients, for some $c_1\in (0,\infty)$ and all $n\in \NN$,
\begin{equation}\label{eq:phibphi}
\begin{aligned}
	|\Phi_B(Q^n, v^n) - \Phi(Q^n, v^n)| &\le c_1 \sup_{0\le t \le T} d_{BL}\left(\nu_{Q^n}(t), \nu_{Q^n}^B(t)\right) \left ( 1 + \frac{1}{n}\sum_{i=1}^n \int_0^T \|u^n_i(s)\| \, ds\right)\\
	&\quad + \frac{c_1}{n}\sum_{i=1}^n \int_0^T \|u^n_i(s) - \zeta_B(u^n_i(s))\| \, ds.
\end{aligned}
\end{equation}
Also, 
\begin{equation}
\sup_{0\le t \le T} d_{BL}\left(\nu_{Q^n}(t), \nu_{Q^n}^B(t)\right) \le \frac{1}{n}\sum_{i=1}^n \sup_{0\le t \le T} \theta(\bar A^n_i(t)) 1_{\left\{\sup_{0\le t \le T} \theta(\bar A^n_i(t))> B \right\}}
\le \frac{1}{nB}\sum_{i=1}^n \sup_{0\le t \le T} \left[\theta(\bar A^n_i(t))\right]^2.\label{eq:eqdblunif}\end{equation}
Combining this with the bounds in \eqref{eq:eq1048b}, \eqref{eq:eq217}, we have, for some $c_2 \in (0,\infty)$ and every $B <\infty$,
$$\sup_{n\in \NN} E \left[ \sup_{0\le t \le T} d_{BL}\left(\nu_{Q^n}(t), \nu_{Q^n}^B(t)\right) \right] \le \frac{c_2}{B}.$$
Fix $\eps \in (0,1)$ and using \eqref{eq:eq135b} choose $m_1 \in (0,\infty)$ such that 
%$m_1 > L/\eps$ so that
$$\sup_{n \in \NN} P\left(\frac{1}{n}\sum_{i=1}^n \int_0^T \|u^n_i(s)\| ds > m_1\right) <\eps.$$
Then using the inequality $E[(UV)\wedge 1] \le P(V > m_1+1) + (m_1+1) E[U]$
for non-negative random variables $U$ and $V$, we have
$$E\left[\left\{c_1 \sup_{0\le t \le T} d_{BL}\left(\nu_{Q^n}(t), \nu_{Q^n}^B(t)\right) \left ( 1 + \frac{1}{n}\sum_{i=1}^n \int_0^T \|u^n_i(s)\| ds\right)\right\}\wedge 1\right] \le \eps + \frac{(m_1+1) c_1 c_2}{B}.$$
Using this estimate in \eqref{eq:phibphi}, for some $c_3 \in (0,\infty)$, 
$$\sup_{n \in \NN} E\left[|\Phi_B(Q^n, v^n) - \Phi(Q^n, v^n)|\wedge 1\right] \le \frac{c_3(1+ m_1)}{B} + \eps.$$
Sending $B\to \infty$ and since $\eps$ is arbitrary, we have the first statement in (b). The second statement in (b) follows in a similar manner on noting the properties
$$ E^* \left[E_{Q}\left[\sup_{0\le t \le T} \theta(\vs(t))^2\right] \right]< \infty, \quad E^*\left[ E_{Q}\left[\int_0^T \left\| \int_{\R^m} y \, r_s(dy)\right\|^2 \,ds \right]\right] <\infty, $$
which follow from analogous (uniform in $n$) bounds when $Q$ is replaced by $Q^n$ and $E^*$ by $E$.

Finally we consider (c). For each $n \in \NN$,
\begin{align*}
	&\Phi(Q^n, v^n) \\
	&= E_{Q^n}\left[ \Psi \left( M^{Q^n}_{f, v^n}(t_1) - M^{Q^n}_{f, v^n}(t_0) \right) \right] \\
						&= \frac{1}{n} \sum_{i=1}^n \Psi(\bar{X}_i^n, \bar{A}_i^n, \rho_i^n, W_i) 
						 \left( M^{Q^n}_{f, v^n}(t_1, (\bar{X}_i^n, \bar{A}_i^n, \rho_i^n, W_i))  -   M^{Q^n}_{f, v^n}(t_0, (\bar{X}_i^n, \bar{A}_i^n, \rho_i^n, W_i)) \right) \\
						&= \frac{1}{n}\sum_{i=1}^n  \Psi(\bar{X}_i^n, \bar{A}_i^n, \rho_i^n, W_i)  \left( f(\bar{X}_i^n(t_1), \bar{A}_i^n(t_1), W_i(t_1)) -  f(\bar{X}_i^n(t_0), \bar{A}_i^n(t_0), W_i(t_0))  
						 - \int_{t_0}^{t_1} \clu^n(s) \, ds \right), 
\end{align*}
where, noting that $\nu_{Q^n}(s) = \bar{\mu}^n(s)$,
\begin{align*}
\clu^n(s)	&\doteq  \mathcal{L}_s^{Q^n}(f)(\bar{X}_i^n(s), \bar{A}_i^n(s), u_i^n(s), W_i(s))
	+  [\nabla_x f(\bar{X}_i^n(s), \bar{A}_i^n(s), W_i(s))]^T \alpha(\Bar{X}_i^n(s), \bar{\mu}^n(s))v^n(s) \\
	&\quad +   \bar{A}_i^n(s) \beta^T(\bar{X}_i^n(s), \bar{\mu}^n(s))v^n(s) \frac{\partial f}{\partial a}(\bar{X}_i^n(s), \bar{A}_i^n(s), W_i(s)). 
\end{align*}
By It\^{o}'s formula, for each $i$ and $n$, we have a.s. that 
\begin{align*}
	&f(\bar{X}_i^n(t_1), \bar{A}_i^n(t_1),  W_i(t_1)) - f(\bar{X}_i^n(t_0), \bar{A}_i^n(t_0), W_i(t_0))\\
	 &= \int_{t_0}^{t_1} \clu^n(s) \,ds 
		+ \int_{t_0}^{t_1} [\nabla_x f(\bar{X}_i^n(s), \bar{A}_i^n(s), W_i(s)) ]^T\sigma(\bar X^n_i(s), \bar{\mu}^n(s))\,dW_i(s) \\
		&\quad + \kappa(n)\int_{t_0}^{t_1} [\nabla_x f(\bar{X}_i^n(s), \bar{A}_i^n(s), W_i(s)) ]^T \alpha(\bar{X}_i^n(s), \bar{\mu}^n(s))\,dB(s) \\
		&\quad+ \int_{t_0}^{t_1} \frac{\partial f}{\partial a}(\bar{X}_i^n(s), \bar{A}_i^n(s), W_i(s)) \bar{A}_i^n(s)\gamma^T(\bar{X}_i^n(s), \bar{\mu}^n(s))\,dW_i(s) \\
		&\quad+ \kappa(n)\int_{t_0}^{t_1} \frac{\partial f}{\partial a}(\bar{X}_i^n(s), \bar{A}_i^n(s), W_i(s)) \bar{A}_i^n(s) \beta^T(\bar{X}_i^n(s), \bar{\mu}^n(s)\,dB(s) \\
		&\quad+ \int_{t_0}^{t_1} [\nabla_w f(\bar{X}_i^n(s), \bar{A}_i^n(s), W_i(s))]^T \,dW_i(s) 
		+ \clt^{n}_i,
\end{align*}
where, for some $c_1 \in (0,\infty)$, $|\clt^{n}_i|\le c_1\kappa(n)^2$ for all $n,i$.
Letting $\Psi^n_i = \Psi(\bar{X}_i^n, \bar{A}_i^n, \rho_i^n, W_i)$,
$f_i^n(s) = f(\bar{X}_i^n(s), \bar{A}_i^n(s), W_i(s))$, and using similar notation for derivatives of $f$, we  have a.s. that
\begin{align*}
	&\Phi(Q^n, v^n) \\
	&\quad=\frac{1}{n}\sum_{i=1}^n \Psi_i^n \Bigg[ \int_{t_0}^{t_1} [\nabla_x f_i^n(s)]^T \sigma(\bar{X}_i^n(s),\bar{\mu}^n(s))\,dW_i(s) + \int_{t_0}^{t_1} \frac{\partial f_i^n(s)}{\partial a} \bar{A}_i^n(s) \gamma^T(\bar{X}_i^n(s), \bar{\mu}^n(s))\,dW_i(s)  \\
	&\qquad +  \int_{t_0}^{t_1}[\nabla_w f^n_i(s)]^T \,dW_i(s)\Bigg] +\clt^n, 
\end{align*}
where, as in the proof of Lemma \ref{limit1}, $\clt^n \to 0$ in probability.
Now by the same argument as in Lemma \ref{limit1}, $\Phi(Q^n, v^n) \to 0$ in probability, proving (c). Thus we have $\Phi(Q, v) = 0$ a.s., which proves \eqref{eq:eq823b} and completes the proof. 
\end{proof}

We now complete the proof of \eqref{eq:varlb2}. 
In addition to the standing assumptions of this section (namely Conditions   \ref{empmzrinit}, \ref{Lip}, \ref{LipA}, \ref{theta} and \ref{wtdmzrinit} )
suppose that if Condition \ref{theta}(a) is not satisfied then $\gamma = 0$.

Since $\bar{\mu}^n = \nu_{Q^n}$, we have
\begin{align*}
	&E\left[ \frac{1}{2n} \sum_{i=1}^n \int_0^T \|u_i^n(t)\|^2\,dt + \frac{1}{2n\kappa(n)^2}\int_0^T \|v^n(t)\|^2\,dt + F(\bar{\mu}^n) \right] \\
			&= E\left[ \int_{\sR_1}\left(\frac{1}{2}\int_{\R^m\times [0,T]}\|y\|^2\,r(dy\,dt) \right)\, [Q^n]_2(dr) + \frac{1}{2n\kappa(n)^2}\int_0^T \|v^n(t)\|^2\,dt + F(\nu_{Q^n}) \right], 
\end{align*}
where $[Q^n]_2$ denotes the second marginal of $Q^n$.
Recalling the bound \eqref{eq:eq128}, we have from Lemmas \ref{tightness2} and \ref{limit} that,
$(Q^n, v^n)$ is tight and if $(Q,v)$ is a weak limit point then $Q \in \cle_2[v]$ a.s. Assume without loss of generality that
$(Q^n, v^n) \to (Q,v)$ along the full sequence.
We claim that $(Q^n, v^n, \nu_{Q^n}) \to(Q, v, \nu_Q)$, in distribution, in $\clp(\clz_2)\times S_M \times \clk$.
For $\Theta \in \clp(\clz_2)$ and $B\in (0,\infty)$, define $\nu_{\Theta}^B \in \clk$ as in \eqref{eq:eqnuthetab}, i.e.
$$\nu_{\Theta}^B(t)(C) \doteq E_{\Theta}\left[\left(\theta(\vs(t)) \wedge B \right) 1_C(z(t))\right], \quad C \in \clb(\R^d).$$
Then it is easy to check that, since $\theta(\cdot)\wedge B$ is a bounded Lipschitz function, $\Theta \mapsto \nu_{\Theta}^B$ is a continuous map from
$\clp(\clz_2)$ to $\clk$ for every $B$. Also, from \eqref{eq:eqdblunif},
for some $c_1 \in (0,\infty)$,
\begin{align*}
	\sup_{n\in \NN} E \left[ \sup_{0\le t \le T} d_{BL}\left(\nu_{Q^n}(t), \nu_{Q^n}^B(t) \right)  \right]
	&\le \frac{c_1}{B} \sup_{n\in \NN}E \left[ \frac{1}{n}\sum_{i=1}^n \sup_{0\le t \le T} \left(\theta(\bar A^n_i(t))\right)^2 \right]  \to 0
\end{align*}
as $B\to \infty$, since 
 $\sup_{n\in \NN}E [\frac{1}{n}\sum_{i=1}^n \sup_{0\le t \le T} (\theta(\bar A^n_i(t)))^2]< \infty$, which follows from \eqref{eq:eq1048b} and linear growth of $\theta$ when property (i)  of \eqref{eq:1202} holds and from \eqref{eq:eq217} when property (iii) in \eqref{eq:1202} is satisfied.
Combining the above uniform convergence with
the fact that $(Q^n, v^n, \nu_{Q^n}^B) \Rightarrow (Q, v, \nu_Q^B)$ for every $B$ proves the claim.

Finally by Fatou's lemma and since $\sqrt{n}\kappa(n)\to \la$,
\begin{align*}
	&\liminf_{n\to\infty} E\left[ \frac{1}{2n} \sum_{i=1}^n \int_0^T \|u_i^n(t)\|^2\,dt + \frac{1}{2n\kappa(n)^2}\int_0^T \|v^n(t)\|^2\,dt + F(\bar{\mu}^n) \right] \\ 
		&\ge E\left[ \int_{\sR_1}\left(\frac{1}{2}\int_{\R^m\times [0,T]}\|y\|^2\,r(dy\,dt) \right)\, [Q]_2(dr) + \frac{1}{2\la^2}\int_0^T \|v(t)\|^2\,dt + F(\nu_Q) \right] \\
		&\ge \inf_{\varphi \in L^2([0,T]:\R^k)} \inf_{\Theta\in \mathcal{E}_2[\varphi]} \left( E_{\Theta}\left[ \frac{1}{2} \int_{\R^m\times[0,T]} \|y\|^2\,\rho(dy\,dt) \right] + \frac{1}{2\la^2} \int_0^T \|\varphi(t)\|^2\,dt + F(\nu_{\Theta}) \right). 
\end{align*}
This proves \eqref{eq:varlb2} and completes the proof of the Laplace upper bound. \hfill \qed

\subsection{Laplace Lower Bound}
\label{sec:laplowbd2}
Throughout this section we assume that Conditions \ref{empmzrinit}-\ref{gammaconst} are satisfied. Additionally we assume that if Condition \ref{theta}(a) does not hold then $\gamma=0$.
We will proceed as in Section
\ref{sec:laplow2.1}.

Fix $\veps>0$ and $F \in \clc_b(\clk)$. Choose a $\varphi \in L^2([0,T]:\R^k)$ and a $\Theta \in \mathcal{E}_2[\varphi]$ such that
\begin{align*}
	\frac{1}{2}E_\Theta\left[ \int_{\R^m\times [0,T]}\|y\|^2\rho(dy\,dt) \right] + \frac{1}{2\la^2}\int_0^T \|\varphi(t)\|^2\,dt + F(\nu_{\Theta}) \le \inf_{\nu \in \clk} \left[ F(\nu) + {I}_2(\nu) \right] + \veps. 
\end{align*}
We will show that there is a $M\in (0,\infty)$ and a sequence $(u^n, v^n)$ with $u^n \in \mathcal{A}^{1,n}$ and $v^n \in \mathcal{A}^2_M$ constructed on some filtered probability space such that 
\begin{equation}\label{eq:eq1242b}
\begin{aligned}
	&\limsup_{n \to \infty} E\left[ \frac{1}{2n} \sum_{i=1}^n \int_0^T \|u_i^n(t)\|^2\,dt + \frac{1}{2n \kappa(n)^2}\int_0^T \|v^n(t)\|^2\,dt + F(\bar{\mu}^n) \right] \\
	&\le \frac{1}{2}E_\Theta\left[ \int_{\R^m\times [0,T]}\|y\|^2\rho(dy\,dt) \right] + \frac{1}{2\la^2}\int_0^T \|\varphi(t)\|^2\,dt + F(\nu_{\Theta}). 
\end{aligned}
\end{equation}
The Laplace lower bound 
\begin{equation*}
	\limsup_{n\to\infty} -\frac{1}{n} \log E\left[e^{-n F(\mu^n)}\right] \le \inf_{\nu \in \clk} \left[ F(\nu) + {I}_2(\nu) \right]
\end{equation*}
is then immediate from Theorem  \ref{reptype2} on noting that $\veps>0$ is arbitrary.
We begin with the following uniqueness result.
Analogous to Section \ref{sec:laplow2.1}, define the map $\vartheta: \clz_2 \to \clz_2^{\vartheta}\doteq \R^d\times \R_+\times \clr_1\times \clw$ as 
$\vartheta(z,\vs, r,w) \doteq (z(0), \vs(0), r,w)$. For $\Theta \in \clp(\clz_2)$, let $\Theta_{\vartheta} \doteq \Theta \circ \vartheta^{-1}$
be the probability measure on $\clz_2^{\vartheta}$ induced by $\Theta$ under $\vartheta$.
We will say that {\em weak uniqueness} holds for \eqref{controlledSDEtype2} if, for any given $\varphi \in L^2([0,T]:\R^k)$ and $\Theta^{(1)}, \Theta^{(2)} \in \cle_2[\varphi]$, whenever
$\Theta^{(1)}_{\vartheta} = \Theta^{(2)}_{\vartheta}$, we have that $\Theta^{(1)} = \Theta^{(2)}$.

\begin{lemma}\label{pathunique2} Weak uniqueness holds for \eqref{controlledSDEtype2}.
\end{lemma}

\begin{proof} Fix $\varphi \in L^2([0,T]:\R^k)$ and $\Theta^{(1)}, \Theta^{(2)} \in \cle_2[\varphi]$. Suppose that $\Theta^{(1)}_{\vartheta} = \Theta^{(2)}_{\vartheta}\doteq \Lambda$.
	  Note that $\Theta^{(i)}$, $i=1,2$ can be disintegrated as
	$$\Theta^{(i)}(dx,\, da,\, dr,\, dw) = \tilde \Theta^{(i)}(dx,\,da,\,x_0,\, a_0, \,r, \,w)\Lambda(dx_0,\, da_0,\, dr,\, dw).$$
Consider $\hat \clz_2 = \clx \times \cly \times \clx \times \cly \times \mathcal{Z}_2^\vartheta$, and define $\hat \Theta \in \clp(\hat \clz_2)$ as
	\begin{align*}
		&\hat \Theta (dx^{(1)},\, da^{(1)},\, dx^{(2)},\, da^{(2)},\, dx_0,\, da_0,\, dr,\, dw) \\
		&\doteq \tilde \Theta^{(1)}(dx^{(1)},\,da^{(1)},\, x_0,\, a_0,\,r, \,w ) \,
	\tilde \Theta^{(2)}(dx^{(2)},\,da^{(2)},\, x_0,\, a_0,\,r, \,w )\, \Lambda(dx_0,\, da_0,\, dr,\, dw),
	\end{align*}
and denote the coordinate maps on $\hat \clz_2$ as $(X^{(1)}, A^{(1)}, X^{(2)}, A^{(2)}, X_0, A_0, \rho, W)$. Note that $\{W(t), t\in[0,T]\}$ is a Brownian motion  with respect to the canonical filtration
$$\hat \clh_t \doteq \sigma\left\{X^{(i)}(s), A^{(i)}(s), \rho(A\times [0,s]), W(s),\; i= 1,2, \;A \in \clb(\R^m), \;s \in [0,t]\right\}, \quad t \in [0,T],$$
and for $i=1,2$, $(X^{(i)}, A^{(i)})$  satisfy \eqref{controlledSDEtype2} 	with $(\bar X, \bar A)$ replaced with $(X^{(i)}, A^{(i)})$ and $\nu(t)$ replaced with $\nu^{(i)}(t) \doteq \nu_{\Theta^{(i)}}(t)$.
In order to prove the lemma it suffices to show that $(X^{(1)}, A^{(1)}) = (X^{(2)}, A^{(2)})$ a.s.
Let $u(t) \doteq \int_{\RR^m} y \, \rho_t(dy)$, $t \in [0,T]$. Then, $E_{\hat \Theta} \int_0^T \|u(t)\|^2 \,dt <\infty$.
By similar estimates as in the proof of Lemma \ref{tightness2} we see that 
\begin{equation}
\mbox{ when Condition \ref{theta}(a) is satisfied, }	E_{\hat \Theta}\left[ \sup_{0\le s\le T} \left(\|X^{(i)}(s)\|^2 +  (\theta({A}^{(i)}(s)))^2\right)\right] < \infty \mbox{ for } i=1,2, 
\label{eq:eq340}	
\end{equation}
and
\begin{equation}
\mbox{ when $\gamma=0$, }	E_{\hat \Theta}\left[ \sup_{0\le s\le T} \left(\|X^{(i)}(s)\|^2 +  ({A}^{(i)}(s))^2\right)\right] < \infty \mbox{ for } i=1,2.
\label{eq:eq340b}	
\end{equation}

Consider first the case $\gamma=0$.
For $t\in [0,T]$, define 
\[
	g(t) = E_{\hat \Theta} \left[\sup_{0\le s\le t} \|X^{(1)}(s) - X^{(2)}(s)\|^2\right], \quad h(t) = \left( E_{\hat \Theta} \left[\sup_{0\le s\le t} |A^{(1)}(s) - A^{(2)}(s)| \right]\right)^2.
\]
Since $\theta$ is a Lipschitz function under Condition \ref{theta}, we have
\begin{align*}
	d_{BL}\left( \nu^{(1)}(s), \nu^{(2)}(s)\right) 
							&\le \sup_{f\in BL(\R^d)} E_{\hat \Theta} \left| \theta(A^{(1)}(s)) f(X^{(1)}(s)) - \theta(A^{(2)}(s)) f(X^{(2)}(s)) \right| \\
							&\le E_{\hat \Theta}\left[ \theta(A^{(1)}(s)) \|X^{(1)}(s) - X^{(2)}(s)\| \right] + L E_{\hat \Theta} \left|A^{(1)}(s) - A^{(2)}(s)\right|, 
\end{align*}
where $L$ is the Lipschitz constant for $\theta$. Then by the Cauchy-Schwarz inequality and \eqref{eq:eq340b}, for some $c_1 \in (0,\infty)$, 
\begin{equation}\label{eq:eqdblnu1nu2}
	\sup_{0\le s\le t} d_{BL}\left( \nu^{(1)}(s), \nu^{(2)}(s)\right)^2 \le c_1( g(t) + h(t))\; \mbox{ for all } t \in [0,T].
\end{equation}
By the Lipschitz properties of $b$, $\sigma$, and $\alpha$, the property $\sigma(x,\nu) = \sigma(\nu)$, the Burkholder-Davis-Gundy and Cauchy-Schwarz inequalities, and the fact that $\int_0^T \|\varphi(s)\|^2\,ds < \infty$, there are  $c_2, c_3 \in (0, \infty)$ such that, for all $t \in [0,T]$, 
\begin{align}
	g(t) &\le c_2 E_{\hat \Theta}\left[ \int_0^t \sup_{0\le \tau\le s} \left( \|X^{(1)}(\tau) - X^{(2)}(\tau)\|^2 + d_{BL}\left(\nu^{(1)}(\tau), \nu^{(2)}(\tau)\right)^2 \right)\,ds\right] \nonumber\\
		&\quad+ c_2 E_{\hat \Theta} \left[\left( \int_0^t \sup_{0\le \tau \le s} d_{BL}\left(\nu^{(1)}(\tau), \nu^{(2)}(\tau)\right) \cdot \|u(s)\|\,ds \right)^2\right] \nonumber\\
		&\le c_3 \int_0^t (g(s) + h(s))\,ds. \label{eq:eq504}
\end{align}
Furthermore, since $\gamma=0$, for $j = 1, 2$, $A^{(j)}(t) = e^{Y^{(j)}(t)}$, where 
\[
	Y^{(j)}(t) = Y^{(j)}(0) + \int_0^t c({X}^{(j)}(s), \nu^{(j)}(s))\,ds + \int_0^t\beta^T({X}^{(j)}(s), \nu^{(j)}(s))\varphi(s)\,ds.
\]
Using the inequality $|e^x - e^y| \le (e^x\vee e^y)|x-y|$, the Lipschitz property of $c$ and $\beta$, \eqref{eq:eqdblnu1nu2} and \eqref{eq:eq340b},	and the Cauchy-Schwarz inequality, there is $c_4 \in (0,\infty)$ such that 
\begin{equation}\label{eq:eq434}
\begin{aligned}
	h(t) &\le \left( E_{\hat \Theta} \left[ \sup_{0\le s\le t} (A^{(1)}(s) \vee A^{(2)}(s)) |Y^{(1)}(s) - Y^{(2)}(s)|\right] \right)^2 \\
		&\le E_{\hat \Theta} \left[ \sup_{0\le s\le t} (A^{(1)}(s) \vee A^{(2)}(s))^2 \right] E_{\hat \Theta} \left[\sup_{0\le s\le t}|Y^{(1)}(s) - Y^{(2)}(s)|^2 \right] \\
		&\le c_4 \int_0^t (g(s) + h(s))\,ds. 
\end{aligned}
\end{equation}
Thus, 
\[
	g(t) + h(t) \le (c_3+c_4) \int_0^t (g(s) + h(s))\,ds \; \mbox{ for every } t \in [0,T], 
\]
and hence by Gronwall's inequality, $g(T) + h(T) = 0$, from which it follows that $(X^{(1)}, A^{(1)})$ and $(X^{(2)}, A^{(2)})$ are indistinguishable on $[0,T]$. 

Consider now the case where Condition \ref{theta}(a) is satisfied. Define
\[
	\tilde h(t) = \left( E_{\hat \Theta}  \left[\sup_{0\le s\le t} |\log A^{(1)}(s) - \log A^{(2)}(s)| \right] \right)^2. 
\]
Since $c_5 \doteq \sup_{x\in \R_+} |\theta'(x) x| < \infty$, we have
\[
	|\theta(e^x) - \theta(e^y)| \le \sup_{z\in \R_+} |\theta'(z) z| \cdot |x-y| = c_5 |x-y|\; \mbox{ for all } x,y \in \R.
\] 
Thus, 
\begin{align*}
	d_{BL}\left( \nu^{(1)}(s), \nu^{(2)}(s)\right) 
							&\le  \sup_{f\in BL(\R^d)}E_{\hat \Theta} \left| \theta(A^{(1)}(s)) f(X^{(1)}(s)) - \theta(A^{(2)}(s))f(X^{(2)}(s)) \right| \\
							&\le E_{\hat \Theta}\left[ \theta(A^{(1)}(s)) \|X^{(1)}(s) - X^{(2)}(s)\| \right] + c_5 E_{\hat \Theta} \left|\log A^{(1)}(s) - \log A^{(2)}(s) \right|. 
\end{align*}
Hence, using \eqref{eq:eq340}, for some $c_6 \in (0, \infty)$, 
\[
	\sup_{0\le s\le t} d_{BL}\left(\nu^{(1)}(s), \nu^{(2)}(s)\right)^2 \le c_6(g(t) + \tilde h(t) ) \; \mbox{ for all } t \in [0,T]. 
\]
Now exactly as in \eqref{eq:eq504}, we have that for some $c_7 \in (0,\infty)$, 
\[
	g(t) \le c_7 \int_0^t (g(s) + \tilde h(s))\,ds \; \mbox{ for every } t \in [0,T].
\]
Note that, for some $c_8 \in (0,\infty)$, 
$$\left(E_{\hat \Theta}\left[ \int_0^t  \left(\gamma(\nu^{(1)}(s)) - \gamma(\nu^{(2)}(s))\right)^T  u(s) \,ds\right]\right)^2
\le c_8 \int_0^t d_{BL}\left(\nu^{(1)}(s), \nu^{(2)}(s)\right)^2 \,ds.$$
Using this estimate and Lipschitz properties of $c, \gamma$, and $\beta$, we now have that, for some $c_9 \in (0,\infty)$
\[
	\tilde h(t) \le c_9 \int_0^t (g(s) + \tilde h(s))\,ds \mbox{ for every } t \in [0,T].
\]
Thus 
\[
	g(t) + \tilde h(t) \le (c_7+c_9) \int_0^t (g(s) + \tilde h(s))\,ds \mbox{ for every } t \in [0,T], 
\]
which shows the indistinguishability of $(X^{(1)}, \log A^{(1)})$ and $(X^{(2)}, \log A^{(2)})$ and hence the indistinguishability of $(X^{(1)},  A^{(1)})$ and $(X^{(1)},  A^{(1)})$ on $[0,T]$. 
\end{proof}

We now complete the proof of the lower bound by constructing $(u^n, v^n)$ that satisfy \eqref{eq:eq1242b}.
Let $(\bar X, \bar A, \rho, W)$ be the coordinate maps on the space $(\mathcal{Z}_2, \sB(\mathcal{Z}_2), \Theta)$ equipped with the canonical filtration 
$\tilde \clh_t$ defined  as in the proof of Lemma \ref{limit}.
Since $\Theta \in \sE_2[\varphi]$, equation \eqref{controlledSDEtype2} is satisfied with $\nu(t) = \nu_{\Theta}(t)$.
Disintegrate $\Theta_{\vartheta}$ as 
\[
	\Theta_{\vartheta}(dx, da, dr, dw) = \eta_0(dx, da)\,[\Theta]_4(dw)\,\hat{\Lambda}_0(x,a, w, dr).
\]
Let $\mathcal{V}, \Omega', \sF'$  and coordinate processses $W_i, B, \rho_i$ be as introduced in Section \ref{sec:laplow2.1}. 
As before, let $\Gamma$ be the standard Wiener measure on $\mathcal{V}$.
Next, for each $n\in \N$, define the probability measure $P^n$ on $(\Omega', \sF')$ by 
\[
	dP^n(r, w, \beta) = \left[\bigotimes_{i=1}^n \,[\Theta]_4(dw_i)\,\hat{\Lambda}_0(x_i^n, a^n_i, w_i, dr_i)\, \bigotimes_{i=n+1}^\infty \,[\Theta]_{(3,4)}(dr_i, dw_i)\right]\otimes\,\Gamma(d\beta). 
\]
Under $P^n$,  $\{W_i\}_{1\le i\le n}$ and $B$ are mutually independent Brownian motions. Define the sequence $\{\Lambda^n\}_{n\in \N}$ of $\mathcal{P}(\R^d\times \R_+\times \mathcal{R}_1 \times \mathcal{W})$-valued random variables by
\[
	\Lambda^n(A\times B \times R\times C) = \frac{1}{n} \sum_{i=1}^n \delta_{x_i^n}(A)\delta_{a_i^n}(B)\delta_{\rho_i}(R)\delta_{W_i}(C), \; A\times B\times R\times C \in \sB(\R^d\times \R_+\times \sR_1 \times \mathcal{W}).
\]
Then by Condition \ref{wtdmzrinit}, 
\begin{equation}\label{weakPns2}
	P^n \circ (\Lambda^n)^{-1} \to \delta_{\Theta_{\vartheta}}. 
\end{equation}
Let, for $n \in \NN$, $v^n \doteq \varphi$. Then $v^n \in S_M$ for every $n$, where  $M \doteq \int_0^T \|\varphi(s)\|^2 ds$.
Next, define $u_i$ by \eqref{eq:equtorho}
and for each $n\in \N$ let $(\bar{X}^n_1, \bar{A}^n_1 \ldots, \bar{X}_n^n, \bar{A}^n_n)$ be the solution on $(\Omega', \sF', P^n)$ of the system in \eqref{controlledn2}
(with $u^n_i$ replaced with $u_i$).

Define the occupation measure $Q^n$ by the right side of \eqref{eq:eq906}, replacing $\rho^n_i$ with $\rho_i$.
 Let $E^n$ denote expectation over the probability measure $P^n$. Then, as in \eqref{eq:eq645} (using Condition \ref{wtdmzrinit} instead of \ref{empmzrinit}), we see that
\begin{align}
	\limsup_{n\to\infty} E^n \left[ \frac{1}{n} \sum_{i=1}^n \int_0^T \|u_i(t)\|^2\,dt \right] 
				&= E_\Theta\left[ \int_0^T  \left\| \int_{\R^m} y\,\rho_t(dy) \right\|^2 \,dt \right]
				< \infty. \label{eq:costcgcea}
\end{align}
It now follows from Lemma \ref{tightness2} that $\{({Q}^n, v^n)\}_{n\in \N}$ is tight.
 If $({Q}, {v})$ is a limit point of this sequence defined on some probability space $(\tilde{\Omega}, \tilde{\sF}, \tilde{P})$, then ${v} = \varphi$ $\tilde{P}$-a.s., and, by Lemma \ref{limit}, 
 ${Q} \in \sE_2[{v}] = \sE_2[\varphi]$ $\tilde{P}$-a.s.
 Also, $\Theta \in  \sE_2[\varphi]$.
  By \eqref{weakPns2}, for $\tilde{P}$-a.e. $\omega \in \tilde{\Omega}$, 
${Q}_{\vartheta}(\omega) = \Theta_{\vartheta}$.
Thus
by the  weak uniqueness established in Lemma \ref{pathunique2}, ${Q} = \Theta$ $\tilde{P}$-a.s. 
Thus we have $Q^n \to \Theta$ in probability. A similar argument as in Lemma \ref{limit} now shows that 
$(Q^n, \nu_{Q^n}) \to (\Theta, \nu_{\Theta})$ in probability.
Finally,
\begin{align*}
	&\limsup_{n\to\infty} E^n\left[ \frac{1}{2n} \sum_{i=1}^n \int_0^T \|u_i(t)\|^2\,dt + \frac{1}{2n\kappa(n)^2}\int_0^T\|v^n(t)\|^2\,dt + F(\bar{\mu}^n) \right] \\
				& = \limsup_{n\to\infty} E^n\left[ \frac{1}{2n} \sum_{i=1}^n \int_0^T \|u_i(t)\|^2\,dt + \frac{1}{2n\kappa(n)^2}\int_0^T\|v^n(t)\|^2\,dt + F(\nu_{{Q}^n}) \right] \\
				& \le \frac{1}{2}E_\Theta\left[ \int_{\R^m\times [0,T]}\|y\|^2\rho(dy\,dt) \right]
				 + \frac{1}{2\la^2} \int_0^T \|\varphi(t)\|^2\,dt + F(\nu_{\Theta}),
\end{align*}
which follows from \eqref{eq:costcgcea}, the equality $v^n = \varphi$, the convergence $(Q^n, \nu_{Q^n}) \to (\Theta, \nu_{\Theta})$, and the assumption that $\sqrt{n} \kappa(n) \to \lambda$.
This proves \eqref{eq:eq1242b} and completes the proof of the lower bound. \hfill \qed

\subsection{Rate Function Property of $I_2$}
\label{sec:ratefunctype2}
The proof is very similar to the argument in Section \ref{sec:ratefun2.1} and so we omit the details and note only that we use the argument in Lemma \ref{limit} to show that if for  $\Theta^n, \Theta \in \clp_2(\clz_2)$, $\Theta^n \to \Theta$, and a bound as in \eqref{eq:unifbdlevset} is satisfied for every $n$,
then under the conditions of Theorem \ref{resulttype2}, $\nu_{\Theta^n}\to \nu_{\Theta}$ in $\mathcal{K}$. \hfill \qed

\section{Proof Sketch of Theorem \ref{thm:diffspeed}. }
\label{sec:proofsketchthm3}
In Section \ref{sec:diffspeed1} we sketch the proof of part (i) of the theorem while part 
(ii) is sketched in Section \ref{sec:diffspeed2}.

\subsection{Case I:  $\sqrt{n}\kappa(n) \to 0$}
\label{sec:diffspeed1}
Recall that we assume Conditions   \ref{empmzrinit}, \ref{Lip}, and \ref{sigmagamma} hold.
For the Laplace upper bound we start with the inequality in \eqref{eq:seqc}
for some $(u^n, v^n) \in \mathcal{A}^{1,n} \times \mathcal{A}^2_M$.
This inequality gives the uniform bound in \eqref{eq:eq128}. With this uniform bound, the tightness of the sequence of
$\clp(\clz_1)$-valued random variables $Q^n$ defined in \eqref{eq:bigempmzr} is shown  as in Lemma \ref{tightness}.

Furthermore, the inequality in \eqref{eq:eq128} also shows that
$$E\left[\int_0^T \|v^n(t)\|^2 dt \right] \le 2n \kappa(n)^2 (2\|F\|_{\infty}+1) \to 0\; \mbox{ as } n\to \infty,$$
since $n\kappa(n)^2 \to 0$. Thus $v^n\to 0$ in $L^2([0,T]:\R^k)$, in probability.   

Now a similar argument as in Lemma \ref{limit1} shows that if $Q$ is a weak limit point of $Q^n$, then $Q \in \cle_1[0]$ a.s.
Finally, with $(u^n, v^n)$ as above and $\bar \mu^n$ defined as below \eqref{controlledn}, taking the limit as $n \to \infty$ along any  convergent subsequence of $\{(Q^n, v^n)\}$,
\begin{align*}
	&\liminf_{n\to\infty} E\left[ \frac{1}{2n} \sum_{i=1}^n \int_0^T \|u_i^n(t)\|^2\,dt + \frac{1}{2n\kappa(n)^2}\int_0^T \|v^n(t)\|^2\,dt + F(\bar{\mu}^n) \right] \\ 
	&\ge \liminf_{n\to\infty} E\left[ \frac{1}{2n} \sum_{i=1}^n \int_0^T \|u_i^n(t)\|^2\,dt  + F(\bar{\mu}^n) \right]\\
		&\ge E\left[ \int_{\sR_1}\left(\frac{1}{2}\int_{\R^m\times [0,T]}\|y\|^2\,r(dy\,dt) \right)\, [Q]_2(dr) + F([Q]_1) \right] \\
		&\ge  \inf_{\Theta\in \mathcal{E}_1[0]} \left( E_{\Theta}\left[ \frac{1}{2} \int_{\R^m\times[0,T]} \|y\|^2\,\rho(dy\,dt) \right] + F([\Theta]_1) \right), 
\end{align*}
where the last inequality uses the fact that $Q \in \cle_1[0]$ a.s.
Since $\delta\in(0,1)$ in \eqref{eq:seqc} is arbitrary, recalling the definition of $\tilde I_{1,0}$ in \eqref{eq:itillar}, the above inequality completes the proof of the Laplace upper bound.

For the proof of the lower bound we proceed as follows. Fix $\veps>0$ and $F \in \clc_b(\clp(\clx))$. Choose $\Theta \in \mathcal{E}_1[0]$ such that
\begin{align}\label{eq:chosphithezero}
	\frac{1}{2}E_\Theta\left[ \int_{\R^m\times [0,T]}\|y\|^2\rho(dy\,dt) \right]  + F([\Theta]_1) \le \inf_{\nu \in \mathcal{P}(\mathcal{X})} \left[ F(\nu) + \tilde{I}_{1,0}(\nu) \right] + \veps. 
\end{align}
Using this $\Theta$, define $(\Om', \clf', P^n)$,  as in Section \ref{sec:laplow2.1}. Also, take $v^n=0$ for every $n$. Then with $u_i$
defined as in \eqref{eq:equtorho} and $\bar \mu^n$ and $Q^n$ constructed as below \eqref{eq:equtorho}, we have exactly as in \eqref{eq:eq1078} that
\begin{align*}
	&\limsup_{n\to\infty} -\frac{1}{n} \log E\left[e^{-n F(\mu^n)}\right]\\
	&\le \limsup_{n\to\infty} E^n\left[ \frac{1}{2n} \sum_{i=1}^n \int_0^T \|u_i(t)\|^2\,dt + \frac{1}{2n\kappa(n)^2}\int_0^T\|v^n(t)\|^2\,dt + F(\bar{\mu}^n) \right] \\
	&= \limsup_{n\to\infty} E^n\left[ \frac{1}{2n} \sum_{i=1}^n \int_0^T \|u_i(t)\|^2\,dt  + F([{Q}^n]_1) \right] \\
	&\le \frac{1}{2}E_\Theta\left[ \int_{\R^m\times [0,T]}\|y\|^2\rho(dy\,dt) \right]
				  + F([\Theta]_1).
\end{align*}
In particular, in obtaining the last equality we have used the uniqueness result in Lemma \ref{pathunique} (applied to the case where $\varphi=0$).
Combining the above inequality with \eqref{eq:chosphithezero} and since $\veps>0$ is arbitrary, we have the desired lower bound.

Finally, the proof that $\tilde I_{1,0}$ is a rate function can be carried out  as in Section \ref{sec:ratefun2.1}. We omit the details.
\hfill \qed
\subsection{Case II: $\sqrt{n}\kappa(n) \to \infty$}
\label{sec:diffspeed2}
For this case we assume Conditions \ref{empmzrinit} and \ref{Lip}. Condition  \ref{sigmagamma} is not needed.
In a similar manner to Theorem \ref{reptype1} it can be shown that for any $F \in \sC_b(\mathcal{P}(\mathcal{X}))$ and for each $n\in \N$, 
\begin{equation}\label{eq:seqrsm}
	-\kappa(n)^2\log E\left[e^{-\frac{1}{\kappa(n)^2} F(\mu^n)}\right] =\inf_{(u, v) \in \mathcal{A}^{1,n}\times \mathcal{A}^2} E\left[ \frac{\kappa(n)^2}{2} \sum_{i=1}^n \int_0^T \|u_i(t)\|^2\,dt + \frac{1}{2}\int_0^T \|v(t)\|^2\,dt + F(\bar{\mu}^n) \right],
\end{equation}
where $\bar \mu^n$ is as introduced below \eqref{controlledn}.
Furthermore, for every $\delta > 0$, there is an $M < \infty$ such that for each $n\in \N$,
\begin{equation}\label{eq:seqbrsm}
\begin{aligned}
	&-\kappa(n)^2\log E\left[e^{-\frac{1}{\kappa(n)^2}  F(\mu^n)}\right] \\
	&\qquad\ge
	\inf_{(u, v) \in \mathcal{A}^{1,n}\times \mathcal{A}_M^2} 
	E\left[\frac{\kappa(n)^2}{2} \sum_{i=1}^n \int_0^T \|u_i(t)\|^2\,dt + \frac{1}{2}\int_0^T \|v(t)\|^2\,dt + F(\bar{\mu}^n) \right] - \delta.
\end{aligned}
\end{equation}
Fix $F \in \sC_b(\mathcal{P}(\mathcal{X}))$ and  $\delta \in (0,1)$. Select, for each $n\in \N$, $(u^n, v^n) \in \mathcal{A}^{1,n} \times \mathcal{A}^2_M$ such that 
\begin{equation}\label{eq:seqcn}
	-\kappa(n)^2\log E\left[e^{-\frac{1}{\kappa(n)^2} F(\mu^n)}\right]  \ge E\left[ \frac{\kappa(n)^2}{2} \sum_{i=1}^n \int_0^T \|u_i^n(t)\|^2\,dt + \frac{1}{2}\int_0^T \|v^n(t)\|^2\,dt + F(\bar{\mu}^n) \right] - 2\delta,
\end{equation}
where $\bar{\mu}^n = \frac{1}{n}\sum_{i=1}^n \delta_{\bar{X}_i^n}$ and $\bar{X}_i^n$ is given by \eqref{controlledn} (repalcing 
$(u,v)$ with $(u^n, v^n)$).
The uniform bound in \eqref{eq:eq128} now gets replaced by
\begin{equation}
	\label{eq:eq128new}
	\sup_{n \in \NN} E\left[ \frac{\kappa(n)^2}{2}\sum_{i=1}^n \int_0^T \|u_i^n(t)\|^2\,dt + \frac{1}{2}\int_0^T \|v^n(t)\|^2\,dt\right] \le 2(\|F\|_{\infty} +1).
\end{equation}
This in particular says that
\begin{equation}
	\label{eq:costtozero}
	E\left[ \frac{1}{n}\sum_{i=1}^n \int_0^T \|u_i^n(t)\|^2\,dt\right] \le \frac{4(\|F\|_{\infty} +1)}{n\kappa(n)^2} \to 0 \mbox{ as } n \to \infty,
\end{equation}
since $n\kappa(n)^2\to \infty$.
Define $Q^n$ by \eqref{eq:bigempmzr}, where $\rho^n_i$ are as in \eqref{eq:eq1027}. The tightness of $(Q^n, v^n)$ is shown  as in Lemma
\ref{tightness}. Let $(Q,v)$ be a weak limit point (along some subsequence) given on some probability space $(\Om^*, \clf^*, P^*)$.
Then using \eqref{eq:costtozero} we see that
\begin{align*}
	E^*\left[ \int_{\sR_1}\int_{\R^m\times [0,T]} \|y\|^2\,r(dy\,dt)\,[Q]_2(dr) \right] \le \liminf_{n\to\infty} E\left[ \frac{1}{n} \sum_{i=1}^n \int_0^T \|u_i^n(t)\|^2\,dt \right] 
			=0.
\end{align*}
Thus we have that
$[Q]_2 = \delta_{r^o}$, where we recall that $r^o(dy\, dt) = \delta_{\{0\}}(dy) \,dt$. Also, as in Lemma \ref{limit1}, it can be seen that
$Q\in \cle_1[v]$ a.s. Combining this fact with $[Q]_2 = \delta_{r^o}$ and recalling the definition of $\tilde \cle_1$ given in Section \ref{sec:sizeofnoise}, we now see that $[Q]_{(1,3)}\in \tilde\cle_1[v]$ $P^*$-a.s. Taking the limit as $n \to \infty$ along a convergent subsequence
\begin{align*}
	&\liminf_{n\to\infty} E\left[ \frac{\kappa(n)^2}{2} \sum_{i=1}^n \int_0^T \|u_i^n(t)\|^2\,dt + \frac{1}{2}\int_0^T \|v^n(t)\|^2\,dt + F(\bar{\mu}^n) \right] \\ 
		 &\ge E^*\left[  \frac{1}{2}\int_0^T \|v(t)\|^2\,dt + F([Q]_1) \right] \\
		&\ge \inf_{\varphi \in L^2([0,T]:\R^k)} \inf_{\Theta\in \tilde \cle_1[\varphi]} \left(  \frac{1}{2} \int_0^T \|\varphi(t)\|^2\,dt + F([\Theta]_1) \right), 
\end{align*}
where the last inequality uses the fact that $[Q]_{(1,3)} \in \tilde\cle_1[v]$ $P^*$-a.s.
Combining this with \eqref{eq:seqcn} and recalling that $\delta\in(0,1)$ is arbitrary completes the proof of the Laplace upper bound.

Now we consider the lower bound.
Fix $\veps>0$ and $F \in \clc_b(\clp(\clx))$. Choose a $\varphi \in L^2([0,T]:\R^k)$ and a $\Theta^o \in \tilde\cle_1[\varphi]$ such that
\begin{align}\label{eq:chosphitheb}
	 \frac{1}{2}\int_0^T \|\varphi(t)\|^2\,dt + F([\Theta^o]_1) \le \inf_{\nu \in \mathcal{P}(\mathcal{X})} \left[ F(\nu) + \tilde{I}_{1,\infty}(\nu) \right] + \veps, 
\end{align}
where $\tilde I_{1,\infty}$ is as in \eqref{eq:itilsma}.
Define $\Theta$ on $(\clz_1, \clb(\clz_1))$ as $\Theta(dz,\, dr,\, dw) = \Theta^o(dz,\,  dw)\, \delta_{r^o}(dr)$.
Using this $\Theta$, 
define $(\Om', \clf', P^n)$  as in Section \ref{sec:laplow2.1}. Also, let $v^n=\varphi$ for every $n$. 
Note that  $u_i$
defined through \eqref{eq:equtorho} satisfies $u_i=0$ $P^n$-a.s. 
Now with $\bar \mu^n$ and $Q^n$ constructed as below \eqref{eq:equtorho}, we have as in \eqref{eq:eq1078} that
\begin{align*}
	&\limsup_{n\to\infty} -\kappa(n)^2\log E\left[e^{-\frac{1}{\kappa(n)^2} F(\mu^n)}\right]\\
	&\le \limsup_{n\to\infty} E^n\left[ \frac{\kappa(n)^2}{2} \sum_{i=1}^n \int_0^T \|u_i(t)\|^2\,dt + \frac{1}{2}\int_0^T\|v^n(t)\|^2\,dt + F(\bar{\mu}^n) \right] \\
				&= \limsup_{n\to\infty}\left(  \frac{1}{2}\int_0^T\|\varphi(t)\|^2\,dt  + E^n\left[F([{Q}^n]_1) \right]  \right)\\
				&= \frac{1}{2}\int_0^T\|\varphi(t)\|^2\,dt
				  + F([\Theta]_1).
\end{align*}
The last equality uses a uniqueness result of the type in Lemma \ref{pathunique} which is shown in the same manner. In particular, since $[\Theta]_2= \delta_{r^o}$,
the proof does not require Condition \ref{sigmagamma} since the analogue of the last term on the right side of \eqref{eq:simplethird}, namely
$$E_{\hat \Theta}\left[\left(\int_0^{T} \left( \|X^{(1)}(s) - X^{(2)}(s)\| + d_{BL}(\nu^{(1)}(s), \nu^{(2)}(s))\right)  \|u(s)\|\,ds \right)^2\right],$$
is simply zero.
Combining the above inequality with \eqref{eq:chosphitheb}
and since $\veps>0$ is arbitrary, we have the desired lower bound.

Finally, the proof that $\tilde I_{1,\infty}$ in \eqref{eq:itilsma} is a rate function is carried out  as before and is omitted.
\hfill \qed

\section{Proof Sketch of  Theorem \ref{resulttype1}(1)}
Let $\rho^n_i = r^o$ for all $i= 1, \ldots n$ and $n \in \NN$. With this choice of $\rho^n_i$, define $Q^n$ by
\eqref{eq:bigempmzr} by replacing $\bar X^n_i$ with $X^n_i$. By Lemmas \ref{tightness} and \ref{limit1}, $\{Q^n\}$ is tight 
and any weak limit point $Q$ satisfies $Q\in \cle_1[0]$. This in particular shows that the 
 nonlinear SDE
\begin{equation}\label{tildeXsde}
\begin{aligned}
	d\tilde{X}(t) &=  b(\tilde{X}(t), \tilde{\mu}(t))\,dt +  \sigma(\tilde{X}(t), \tilde{\mu}(t))\,dW(t), \\ 
	\tilde{X}(t) &\sim \tilde \mu(t),\quad t \ge 0, \quad  \tilde \mu(0)= \xi_0,
\end{aligned}
\end{equation}
has a weak solution, namely on some filtered probability space $(\bar\Om, \bar\clf, \bar P, \{\bar\clf_t\})$ equipped with an $m$-dimensional $\bar\clf_t$-Brownian motion $W$, there is an $\bar\clf_t$-adapted process $\tilde X$ with sample paths in $\clc([0,T]:\RR^d)$ which satisfies
the above equation. Furthermore, using standard Lipschitz estimates, martingale inequalities, and Gronwall's lemma, we see that
pathwise uniqueness holds for \eqref{tildeXsde}. Thus, by the Yamada-Watanabe results (cf. \cite[Chapter IV]{ikewat}) there is a unique weak solution to \eqref{tildeXsde}. 
Denote this weak solution (namely the probability law of $(\tilde X, W)$)) as $\Theta^*_{(1,3)}$.
Let
$\Theta^* \in \clp(\clz_1)$ be defined as $\Theta^*(dx, dr, dw ) = \Theta^*_{(1,3)}(dx, dw)\, \delta_{r^o}(dr)$.  Then any weak limit point
$Q$ of $Q^n$ must equal $\Theta^*$ a.s. As argued above, $\Theta^*$ is the unique element in $\clp(\clz_1)$ that is a weak solution of $\cls_1[0, \nu_{\Theta}]$ and satisfies $[\Theta^*]_2 = \delta_{r^o}$. The result follows. \hfill \qed

\section{Proof Sketch of Theorem \ref{resulttype2}(1)} 
\label{sec:llnwted}
It was noted in Section \ref{sec:wtedempmzr}, that the system of equations in \eqref{eq:interempwtd} has a unique strong solution under Conditions \ref{empmzrinit}, \ref{Lip}, \ref{LipA},  \ref{theta}, and
\ref{wtdmzrinit}.
This can be seen as follows.  
Note that, with $\zeta = b,\sigma,\alpha$, the maps
$$(x,a) = (x_1, \ldots , x_n, a_1, \ldots, a_n) \mapsto (\zeta(x_1, \mu(x,a)), \ldots,  \zeta(x_n, \mu(x,a))),$$
and with $\vs = c, \gamma^T, \beta^T$, the maps
$$(x,a) \mapsto (a_1\vs(x_1, \mu(x,a)), \ldots,  a_n\vs(x_n, \mu(x,a))),$$
where $\mu(x,a) = \frac{1}{n}\sum_{i=1}^n \theta(a_i) \delta_{x_i}$, are locally Lipschitz functions with (at most) linear growth from $\RR^{nd}\times \RR_+^n$ to appropriate Euclidean spaces.
For example, for $(x,a), (\tilde x, \tilde a) \in \RR^{nd}\times \RR_+^n$,
\begin{align*}
	|b(x_i,\mu(x,a)) - b(\tilde x_i,\mu(\tilde x,\tilde a))|
&\le K (\|x_i- \tilde x_i\| + d_{BL}(\mu(x,a), \mu(\tilde x,\tilde a)))\\
&\le K \left(\|x_i- \tilde x_i\| + \frac{1}{n} \sum_{i=1}^n\left(|\theta(a_i) - \theta(\tilde a_i)| + \theta(a_i) \|x_i- \tilde x_i\|\right)\right).
\end{align*}
The local Lipschitz property of  $(x,a) \mapsto b(x_1, \mu(x,a))$ is immediate from the above estimate on recalling that under Conditon \ref{theta}, $\theta$ is a Lipschitz function. Properties on other coefficients can be verified in a similar manner. Existence and uniqueness of strong solutions of \eqref{eq:interempwtd} follows from this.

We are interested in the asymptotic behavior of $t\mapsto \mu^n(t)$ regarded as a sequence of $\clc([0,T]:\clm_+(\R^d))$-valued random variables, where $\mu^n(t)$ is defined as in \eqref{eq:interempwtd}. In order to characterize the limit of $\mu^n$, we consider the  nonlinear SDE $\cls_2[0, \tilde \mu]$ in \eqref{controlledSDEtype2}  with $\rho(dy\, dt) = r^o(dy)\, dt$, namely the following equation given
on some filtered probability space $(\bar\Om, \bar\clf, \bar P, \{\bar\clf_t\})$, equipped with an $m$-dimensional $\bar\clf_t$-Brownian motion $W$:
\begin{equation}\label{XAtildeSDE}
\begin{aligned}
d\tilde{X}(t) &=  b(\tilde{X}(t), \tilde{\mu}(t))\,dt +  \sigma(\tilde{X}(t),\tilde{\mu}(t))\,dW(t) \\
d\tilde{A}(t) &=  \tilde{A}(t) c(\tilde{X}(t), \tilde{\mu}(t))\,dt + \tilde{A}(t) \gamma^T(\tilde{X}(t), \tilde{\mu}(t))\,dW(t), \\
\il f, \tilde \mu(t) \ir &= \bar E[\theta(\tilde{A}(t))f(\tilde{X}(t))], \quad f\in \sC_b(\R^d),\quad t \ge 0, \quad (\tilde{X}(0), \tilde A(0)) \sim \eta_0.
\end{aligned}
\end{equation}
Let $\tilde\clz_2 \doteq \clx \times \cly \times \clw$, and denote the canonical coordinate maps on this space as $(\tilde X, \tilde A, W)$.
Let $\tilde \clh_t \doteq \sigma\{\tilde X(s), \tilde A(s), W(s), s \le t\}$ be the canonical filtration on this space.
By a {\it weak solution} of \eqref{XAtildeSDE} we mean a probability measure $\Theta$ on $\tilde\clz_2$ such that, under $\Theta$, $W$ is a standard $\tilde \clh_t$-Brownian motion and
the system of equations \eqref{XAtildeSDE} are satisfied a.s. 

As before, let $\rho^n_i = r^o$ for all $i= 1, \ldots, n$ and $n \in \NN$. Define $Q^n$ by
\eqref{eq:eq906} by replacing $(\bar X^n_i, \bar A^n_i)$ with $(X^n_i, A^n_i)$. By Lemmas \ref{tightness2} and \ref{limit}, $\{Q^n\}$ is tight 
and any weak limit point $Q$ satisfies $Q\in \cle_2[0]$ (we use \eqref{eq:1202}(ii) here). In particular, this shows that
 $[Q]_{1,2,4}$ is a weak solution of \eqref{XAtildeSDE}.
  The following result shows the equation in fact has a unique weak solution.
\begin{lemma}
	\label{lem:weakuniqnlwted}
	Under Conditions  \ref{empmzrinit}, \ref{Lip}, \ref{LipA}, \ref{theta}, and \ref{wtdmzrinit}, equation \eqref{XAtildeSDE} has a unique weak solution.
\end{lemma}
\begin{proof}
	It sufices to show that the equation has a unique pathwise solution, namely that if $(\tilde X^{(i)}, \tilde A^{(i)}, \tilde \mu^{(i)})$, $i=1,2$ are two solutions of \eqref{XAtildeSDE} given 
	on some filtered probability space $(\bar\Om, \bar\clf, \bar P, \{\bar\clf_t\})$ equipped with an $m$-dimensional $\bar\clf_t$-Brownian motion $W$ (namely, $(\tilde X^{(i)}, \tilde A^{(i)})$ are continuous $\{\bar\clf_t\}$ adapted processes and \eqref{XAtildeSDE} is satisfied with $(\tilde X, \tilde A, \tilde \mu)$
	replaced with $(\tilde X^{(i)}, \tilde A^{(i)}, \tilde \mu^{(i)})$, $i=1,2$),  and such that
	$(\tilde X^{(1)}(0), \tilde A^{(1)}(0)) = (\tilde X^{(2)}(0), \tilde A^{(2)}(0))$ $\bar P$-a.s., then 
	\begin{align}\label{eq:951}
		(\tilde X^{(1)}, \tilde A^{(1)}, \tilde \mu^{(1)}) = (\tilde X^{(2)}, \tilde A^{(2)}, \tilde \mu^{(2)}) \quad \bar P\mbox{-a.s.}\end{align}
	Using Conditions \ref{empmzrinit} and \ref{wtdmzrinit} on the initial random variables and Conditions \ref{Lip} and \ref{LipA} on the coefficients, it is easy to check by Gronwall's inequality that
	\begin{equation}
		\bar E\left[ \sup_{0\le t \le T} \left(\|\tilde X^{(i)}(t)\|^2 + \tilde A^{(i)}(t)^2\right)\right] <\infty \;\mbox{ for } i=1,2. \label{eq:eq1247}
	\end{equation}
Let 
\[
	g(t) = \bar E \left[\sup_{0\le s\le t}\|\tilde{X}^{(1)}(s) - \tilde{X}^{(2)}(s)\|^2\right] \mbox{ and } h(t) = \left(\bar E \left[\sup_{0\le s\le t} |\tilde{A}^{(1)}(s) - \tilde{A}^{(2)}(s)| \right]\right)^2. 
\]
Then, exactly as for \eqref{eq:eqdblnu1nu2},	 there is a $c_1 \in (0,\infty)$ such that for all  $0\le s\le t\le T$, 
\begin{align*}
	d_{BL} \left( \tilde{\mu}^{(1)}(s),  \tilde{\mu}^{(2)}(s) \right)^2 
						&\le c_1 (g(t) + h(t)). 
\end{align*}
By the  Lipschitz property of $b$ and $\sigma$, we then have that   for some  $c_2 \in (0,\infty)$ and all $0\le t \le T$, 
\begin{align*}
	g(t) 
		\le c_2 \int_0^t \left( g(s) + h(s) \right)\, ds. 
\end{align*}
Writing $\tilde{A}^{(i)}(t) = e^{\tilde{Y}^{(i)}(t)}$ for  $i=1,2$ and using the bounded Lipschitz properties of $c$ and $\gamma$, we see as in \eqref{eq:eq434} that for 
some  $c_3 \in (0,\infty)$ and all $0\le t \le T$, 
$$h(t) \le c_3\int_0^t \left( g(s) + h(s) \right)\, ds.$$
Thus, $g(t) + h(t) \le (c_2+c_3) \int_0^t \left( g(s) + h(s) \right)\,ds$ for all $t \in [0,T]$ which, by Gronwall's inequality, then shows that $g(T) + h(T) = 0$ . Thus $(\tilde{X}^{(1)}, \tilde{A}^{(1)})$ and $(\tilde{X}^{(2)}, \tilde{A}^{(2)})$ are indistinguishable on $[0,T]$ which proves \eqref{eq:951}.
\end{proof}

We now complete the proof of Theorem \ref{resulttype2}(1). 
Denoting the unique weak solution of \eqref{XAtildeSDE} as $Q^*_{(1,2,4)}$ we now have that $[Q^n]_{(1,2,4)} \to Q^*_{(1,2,4)}$ in probability as $n \to \infty$. Let $Q^*(dx,\, da,\, dr, dw) \doteq Q^*_{(1,2,4)}(dx,\, da,\,  dw) \delta_{r^o}(dr)$.
Then $Q^n \to Q^*$ in probability. Note that $Q^*$ is the unique element $\Theta$ in $\clp(\clz_2)$ that is a weak solution of $\cls_2[0, \nu_{\Theta}]$ and satisfies $[\Theta]_3 = \delta_{r^o}$.
Using the estimate 
$$\sup_{n\in \NN}E \left[\frac{1}{n}\sum_{i=1}^n \sup_{0\le t \le T} A^n_i(t)^2\right] < \infty,$$ 
which follows by the argument in
\eqref{eq:eq1048b}, it now follows exactly as in the proof of the Laplace upper bound (see arguments below the proof of Lemma \ref{limit}) that
 $\nu_{Q^n} \to \nu_{Q^*}$ in $\clk$, in probability, where for $\Theta \in \clp_2(\clz_2)$, $\nu_{\Theta}$ is defined as in \eqref{eq:eq503}. \hfill \qed

\vspace{\baselineskip}\noindent \textbf{Acknowledgements:} 
This research was  supported in part by the NSF (DMS-1814894, DMS-1853968).
We thank an anonymous reviewer for directing us to the related works \cite{husalspi,sunwanxuyan}.

\scriptsize{\noindent \textsc{A. Budhiraja and M. Conroy\newline
Department of Statistics and Operations Research\newline
University of North Carolina\newline
Chapel Hill, NC 27599, USA\newline
email:  budhiraj@email.unc.edu , mconroy@live.unc.edu \vspace{\baselineskip} }

}


\begin{thebibliography}{10}

\bibitem{antfrirobtib}
N.~Antunes, C.~Fricker, P.~Robert, and D.~Tibi.
\newblock Stochastic networks with multiple stable points.
\newblock {\em Ann. Probab.}, 36:255--278, 2008.

\bibitem{balfasfautou}
J.~Baladron, D.~Fasoli, O.~Faugeras, and J.~Touboul.
\newblock Mean-field description and propagation of chaos in networks of
  {H}odgkin-{H}uxley and {F}itz{H}ugh-{N}agumo neurons.
\newblock {\em J. Math. Neurosci.}, 2(1):10, 2012.

\bibitem{bormacpro}
C.~Bordenave, D.~Macdonald, and A.~Proutiere.
\newblock Performance of random medium access control, an asymptotic approach.
\newblock In {\em Proc. ACM Sigmetrics}, pages 1--12, 2008.

\bibitem{bormacpro2}
C.~Bordenave, D.~McDonald, and A.~Proutiere.
\newblock Random multi-access algorithms in networks with partial interaction:
  A mean field analysis.
\newblock {\em Netw. Heterog. Media}, 5(1):31--62, 2010.

\bibitem{boudup}
M.~Bou\'{e} and P.~Dupuis.
\newblock A variational representation for certain functionals of {B}rownian
  motion.
\newblock {\em Ann. Probab.}, 26:1641--1659, 1998.

\bibitem{brahep1}
W.~Braun and K.~Hepp.
\newblock The {V}lasov dynamics and its fluctuations in the 1/{N} limit of
  interacting classical particles.
\newblock {\em Comm. Math. Phys.}, 56(2):101--113, 1977.

\bibitem{buddup3}
A.~Budhiraja and P.~Dupuis.
\newblock A variational representation for positive functionals of infinite
  dimensional {B}rownian motion.
\newblock {\em Probab. Math. Statist.}, 20:39--61, 2000.

\bibitem{buddupfis}
A.~Budhiraja, P.~Dupuis, and M.~Fischer.
\newblock Large deviation properties of weakly interacting processes via weak
  convergence methods.
\newblock {\em Ann. Probab.}, 40:74--102, 2012.

\bibitem{daw1}
D.~A. Dawson.
\newblock Critical dynamics and fluctuations for a mean-field model of
  cooperative behavior.
\newblock {\em J. Stat. Phys.}, 31(1):29--85, 1983.

\bibitem{dawgar}
D.~A. Dawson and J.~G{\"a}rtner.
\newblock Large deviations from the {M}c{K}ean-{V}lasov limit for weakly
  interacting diffusions.
\newblock {\em Stochastics}, 20(4):247--308, 1987.

\bibitem{del1}
P.~Del~Moral.
\newblock {\em Feynman-{K}ac formulae}.
\newblock Probability and its Applications (New York). Springer-Verlag, New
  York, 2004.
\newblock Genealogical and interacting particle systems with applications.

\bibitem{dupell4}
P.~Dupuis and R.~S. Ellis.
\newblock {\em A Weak Convergence Approach to the Theory of Large Deviations}.
\newblock John Wiley \& Sons, New York, 1997.

\bibitem{dupspi}
P.~Dupuis and K.~Spiliopoulos. 
\newblock Large deviations for multiscale problems via weak convergence methods. 
\newblock{\em Stoch. Proc. Appl.}, 122:1947--1987, 2012. 

\bibitem{garpapyan}
J.~Garnier, G.~Papanicolaou, and T-W. Yang.
\newblock Large deviations for a mean field model of systemic risk.
\newblock {\em SIAM J. Financ. Math.}, 4(1):151--184, 2013.

\bibitem{giespisow}
K.~Giesecke, K.~Spiliopoulos, and R.~B. Sowers.
\newblock Default clustering in large portfolios: Typical events.
\newblock {\em Ann. Appl. Probab.}, 23(1):348--385, 2013.

\bibitem{gomgraleb}
J.~Gom\'{e}z-Serrano, C.~Graham, and J-Y.~L. Boudec.
\newblock The bounded confidence model of opinion dynamics.
\newblock {\em Math. Models Methods Appl. Sci.}, 22(02):1150007, 2012.

\bibitem{gramel3}
C.~Graham and S.~M{\'e}l{\'e}ard.
\newblock Stochastic particle approximations for generalized {B}oltzmann models
  and convergence estimates.
\newblock {\em Ann. Probab.}, 25(1):115--132, 1997.

\bibitem{grarob2}
C.~Graham and P.~Robert.
\newblock A multi-class mean-field model with graph structure for {TCP} flows.
\newblock In {\em Progress in industrial mathematics at ECMI 2008}, pages
  125--131. Springer, 2010.
  
  

\bibitem{husalspi} % Suggested citation
W.~Hu, M.~Salins, and K.~Spiliopoulos. 
\newblock Large deviations and averaging for systems of slow-fast stochastic reaction-diffusion equations. 
\newblock {\em  Stoch. PDE: Anal. Comp.} 7:808--874, 2019. 
  
  

\bibitem{ikewat}
N.~Ikeda and S.~Watanabe.
\newblock {\em Stochastic {D}ifferential {E}quations and {D}iffusion
  {P}rocesses}, volume~24 of {\em North-Holland Mathematical Library}.
\newblock North-Holland Publishing Co., Amsterdam; Kodansha, Ltd., Tokyo,
  second edition, 1989.

\bibitem{kurxio1}
T.~G. Kurtz and J.~Xiong.
\newblock Particle representations for a class of nonlinear {SPDE}s.
\newblock {\em Stochastic Process. Appl.}, 83(1):103--126, 1999.

\bibitem{mck1}
H.~P.~McKean.
\newblock A class of {M}arkov processes associated with nonlinear parabolic
  equations.
\newblock {\em Proc. Natl. Acad. Sci. USA}, 56(6):1907, 1966.

\bibitem{mel}
S.~M{\'e}l{\'e}ard.
\newblock Asymptotic behaviour of some interacting particle systems;
  {M}c{K}ean-{V}lasov and {B}oltzmann models.
\newblock In {\em Probabilistic models for nonlinear partial differential
  equations ({M}ontecatini {T}erme, 1995)}, volume 1627 of {\em Lecture Notes
  in Math.}, pages 42--95. Springer, Berlin, 1996.

\bibitem{oel}
K.~Oelschl{\"a}ger.
\newblock A martingale approach to the law of large numbers for weakly
  interacting stochastic processes.
\newblock {\em Ann. Probab.}, 12(2):458--479, 1984.

\bibitem{orr}
Carlo Orrieri.
\newblock Large deviations for interacting particle systems: joint mean-field
  and small-noise limit.
\newblock {\em arXiv preprint arXiv:1810.12636}, 2018.

\bibitem{shitan}
T.~Shiga and H.~Tanaka.
\newblock Central limit theorem for a system of {M}arkovian particles with mean
  field interactions.
\newblock {\em Probab. Theory Related Fields}, 69(3):439--459, 1985.


\bibitem{sunwanxuyan} % Suggested reference
X.~Sun, R.~Wang, L.~Xu, and X.~Yang. 
\newblock Large deviations for two-time-scale stochastic {Burgers} equation. 
\newblock {\em Stoch. Dyn.}, 2020. 


\bibitem{szn1}
A-S. Sznitman.
\newblock Nonlinear reflecting diffusion process, and the propagation of chaos
  and fluctuations associated.
\newblock {\em J. Funct. Anal.}, 56(3):311--336, 1984.

\bibitem{szn91}
A-S. Sznitman.
\newblock {\em Topics in propagation of chaos}, volume 1464, pages 167--251.
\newblock Springer, Berlin, 1991.

\bibitem{tan2}
H.~Tanaka.
\newblock Limit theorems for certain diffusion processes with interaction.
\newblock In {\em Stochastic analysis ({K}atata/{K}yoto, 1982)}, volume~32 of
  {\em North-Holland Math. Library}, pages 469--488. North-Holland, Amsterdam,
  1984.

\end{thebibliography}
\end{document}